\documentclass[11pt,a4paper]{amsart}
\usepackage[utf8]{inputenc}
\usepackage{amsmath,hyperref}
\usepackage{amsfonts}
\usepackage{amssymb}
\usepackage{amsthm}
\usepackage{mathtools}
\usepackage{graphicx}
\usepackage[left=2.5cm,right=2.5cm,top=2.5cm,bottom=2.5cm]{geometry}
\newtheorem{theorem}{Theorem}[section]
\newtheorem{lemma}[theorem]{Lemma}
\newtheorem{problem}[theorem]{Problem}
\newtheorem{condition}[theorem]{Condition}
\newtheorem{proposition}[theorem]{Proposition}
\newtheorem{definition}[theorem]{Definition}
\newtheorem{corollary}[theorem]{Corollary}

\newtheorem*{theorem*}{Theorem}
\DeclareMathOperator{\Tr}{tr}

\DeclareMathOperator{\Id}{Id}

\DeclareMathOperator{\Var}{Var}
\DeclareMathOperator{\argmin}{arg\,min}

\usepackage{color}

\title{Spectral deconvolution of unitarily invariant matrix models}

\author[P. Tarrago]{Pierre Tarrago}
\email{pierre.tarrago@upmc.fr}

\begin{document}

\begin{abstract}
In this paper, we implement a complex analytic method to build an estimator of the spectrum of a matrix perturbed by either the addition or the multiplication of a random matricial noise. This method, which has been previously introduced by Arizmendi, Tarrago and Vargas, is done in two steps : the first step consists in a fixed point method to compute the Stieltjes transform of the desired distribution in a certain domain, and the second step is a classical deconvolution by a Cauchy distribution, whose parameter depends on the intensity of the noise. We also provide explicit bounds for the mean squared error of the first step under the assumption that the distribution of the noise is unitarily invariant. Using known results from the classical deconvolution problem, this implies bounds on the accuracy of our estimation of the unknown spectral distribution.
\end{abstract}
\maketitle
\begin{flushright}
\textit{To Roland Speicher, for his 60th birthday.}
\end{flushright}
\section{Introduction}
Recovery of data from noisy signal is a recurrent problem in many areas of mathematics (geology, wireless communication, finance, electroencephalography...). From a statistical point of view, this can be seen as the recovery of a probability distribution from a sample of the distribution perturbed by a noise. In the simplest case, the perturbation is a convolution of the original distribution with a distribution representing the noise, and the process of recovering the original probability distribution from a sample of the convolved one is called deconvolution. In \cite{Fan1,Fan2}, Fan presented a first general approach to the deconvolution of probability distributions, which allowed to both recover the original data and to get a bound on the accuracy of the recovery. Since this seminal paper, several progresses have been made towards a better understanding of the classical deconvolution of probability measures.

In this paper, we are interested in the broader problem of the recovery of data in a non-commutative setting. Namely, we are given a matrix $g(A,B)$, which is an algebraic combination of a possibly random matrix $B$ representing the data we want to recover and a random matrix $A$ representing the noise, and the goal is to recover the matrix $B$. Taking $A$ and $B$ diagonals and independent with entries of each matrix iid and considering the case $g(A,B)=A+B$ is equivalent to the classical deconvolution problem. This non commutative generalization has already seen many applications in the simplest cases of $g$ being the addition or multiplication of matrices, \cite{BoPo,LeWo1,AlBoBu}. Yet, the recovery of $B$ is a complicated process already in those situations and we propose to address these two cases in the present manuscript. Although our aim is to provide a concrete method to tackle the problem and to give explicit bounds the estimator we build, let us first discuss some important theoretical aspects of the non-commutative setting.

A first difference with the classical case is the notion of independence. In the classical case, independence is a fundamental hypothesis in the succcess of the deconvolution, which allows to translate sum of random variables into convolution of distributions. In the non-commutative setting, one can generally consider two main hypotheses of independence: either the entries of $A$ and $B$ are assumed to be independent and the entries of $A$ are assumed iid (up to a symmetry if $A$ is self-adjoint), or the distribution of the noise matrix $A$ is assumed to be invariant by unitary conjugation. Both notions generally yield similar results but require different tools. In this paper, we focus on the second hypothesis of a unitarily invariant noise, which has already been studied in \cite{AlBoBu,AlBeEn,LePe}. Note that in the case of Gaussian matrices with independent entries, the hypothesis of unitarily invariance of the distribution is also satisfied, and both notion of independence coincide. The results of the present paper extend of course to the case of orthogonally invariant noises, up to numerical constants.

The second question is the scope of the deconvolution process : assuming $B$ self-adjoint, a perfect recovery of $B$ would mean the recovery of both its eigenvalues and its eigenbasis. The recovery of the eigenbasis heavily depends on the model. Indeed, if we consider the model $ABA^*$ where the law of $A$ is invariant by right multiplication by a Haar unitary, then for any unitary matrix $U$ the law of $ABA^*$ and $AUBU^*A^*$ are the same, which prevents any hope to recover the eigenbasis of $B$. On the contrary, we will show that it is always possible to recover, to some extent, the eigenvalues of $B$, with an accuracy improving when the size of the matrices grows. In some cases, obtaining the spectrum of $B$ is a first step towards the complete recovery of $B$. This is the main approach of \cite{LePe} in the estimation of large covariance matrices, which has led to the succesful shrinkage method of \cite{LeWo1,LeWo2}. This method has been generalized in \cite{AlBoBu, AlBeEn} to provide a general method to build estimators of the matrix $B$ in the additive and multiplicative case when the distribution of the noise matrix $A$ is assumed unitarily invariant: once again, this approach uses the knowledge of the spectral distribution of $B$ as an oracle, and the missing step of the latter method is precisely a general way of estimating the spectral distribution of $B$. To summarize the above paragraph, we are led to consider the \textit{spectral deconvolution of unitarily invariant models}.

In the classical deconvolution, the known fact that the Fourier transform of the convolution of two probability measures is the product of the Fourier transform of both original measures has been the starting point of the pioneering work of Fan \cite{Fan1}. Indeed, apart from definition issues, one can see the classical deconvolution as the division of the Fourier transform of the received signal by the Fourier transform of the noise. In the non-commutative setting, there is no close formula describing the spectrum of algebraic combination of finite size matrices, which prevents any hope of concrete formulas in the finite case. However, as the size goes to infinity, the spectral properties of sums and products of independent random matrices is governed by the free probability theory \cite{Voi1}. The spectral distribution of the sum of independent unitarily invariant random matrices is closed to the so-called free additive convolution of the specral distributions of each original matrices, and the one of the product is closed to the free multiplicative convolution of the spectral dsitributions. Based on this theory and complex analysis, the subordination method (see \cite{Bia,Bel,Bebe,Voi2,BeMaSp}) provides us tools to compute very good approximations of the spectrum of sums and multiplications of independent random matrices in the same flavor as the multiplication of the Fourier transforms in the classical case. In the important case of the computation of large covariance matrices, the subordination method reduces to the Marchenko-Pastur equation, which lies at the heart of the nonlinear schrinkage method \cite{LeWo1}. 

In \cite{ATV}, Arizmendi, Vargas and the author developed an approach to the spectral deconvolution by inverting the subordination method. This approach showed promising results on simulations, and the goal of this manuscript is to shows theoretically that it succesfully achieves the spectral deconvolution of random matrix models in the additive and multiplicative case. We also provide first concentration bounds on the result of the deconvolution, in the vein of Fan's results on the classical deconvolution \cite{Fan1}. In his first two papers dealing with deconvolution, Fan already noted that the accuracy of the deconvolution greatly worsens as the noise gets smoother, and improves as the distribution to be recovered gets smoother. This can be seen at the level of the Fourier transform approach. Indeed, the Fourier transform of a smooth noise is rapidly decreasing to zero at infinity and thus the convolution with a smooth noise sets the Fourier transform of the original distribution exponentially close to zero for higher modes, acting as a low pass filter. Hence, when the original distribution has non-trivial higher modes, it is thus extremely difficult to recover those higher frequencies in the deconvolution, which translates into a poor concentration bound on the accuracy of the process. When the original distribution is also very smooth, those higher modes do not contribute to the distribution and thus the recovery is still accurate. In the supersmooth case where the Fourier transform of the noise is decreasing exponentially to zero at infinity, the accuracy is logarithmic, except when the original distribution is also supersmooth.

In \cite{Bebe2}, Belinschi and Bercovici proved that the free additive and multiplicative convolutions of probability measures are always analytic, except at some exceptional points. As the spectral deconvolution is close to reversing a free convolution, we should expect the behavior of the spectral convolution to be close to the ultrasmooth case of Fan. This phenomenon appears in the method proposed in \cite{ATV}, which first builds an estimator $\widehat{\mathcal{C}_{B}}$ of the convolution $\mathcal{C}_{B}$ of the desired distribution by a certain Cauchy distribution, and then achieve the classical deconvolution of $\widehat{\mathcal{C}_{B}}$ by this Cauchy distribution, which is a supersmooth. Therefore, the accuracy of the spectral deconvolution method should be approximately the one of a deconvolution by a Cauchy transform. We propose then to measure the accuracy of the method by two main quantities: the parameter of the Cauchy transform involved in the first step of the deconvolution, and the size of the matrices. We show that the parameter of the Cauchy transform, which gives the range of Fourier modes we can recover, depends mainly on the intensity of the noise, while the precision of the recovery of  $\mathcal{C}_{B}$ depends on the size $N$ of the model. This is similar to the situation in the classical case \cite{Fan1}. The concentration bounds we get for the estimator of $\mathcal{C}_{B}$ depend on the first six moments of the spectral distribution of $A$ and $B$ in the additive case, and also on the bound of the support of $A$ in the multiplicative case. Parallel to our work, Maïda et al. \cite{Mal} have succesfully used the method from \cite{ATV} to study the backward free Fokker-Planck equation. In the course of their study, they also managed to improve the method of \cite{ATV} in the case of a semi-circular noise and to measure the accuracy of the method in the case of a backward Dyson Brownian motion.

Let us describe the organization of the manuscript. In Section \ref{Section:statement}, we explain precisely the models, recall the deconvolution procedure implemented in \cite{ATV} and states the concentration bounds. This section is self-contained for a reader only interested in an overview of the deconvolution and its practical implementation and accuracy, and in particular the free probabilistic background is postponed to next section.  The method for the multiplicative deconvolution has been improved from the one in \cite{ATV}, and the proof of the improved version is postponed to Appendix \ref{Appendix:proof_multiplicative_subordination}. We also provide simulations to illustrate the deconvolution procedure and to show how the concentration bounds compare to simulated errors. In Section \ref{Section:free_probabilistic_background}, we introduce all necessary background to prove the concentration bounds, and we introduce matricial subordinations of Pastur and Vasilchuk \cite{PaVa}, which is the main tool of our study. The proof of the concentration of the Stieltjes transform of the original measure is done in Section \ref{Section:bound_subordination}, \ref{Section:pointwise_concentration}. These proofs heavily rely on integration formulas and concentration bounds on the unitary groups, which are respectively described in Appendix \ref{Appendix:weingarten} and \ref{Appendix:concentration_unitary_group}. 
\subsubsection*{Acknowledgments} 
We would like to thank Emilien Joly for fruitful discussions. We also thank Claire Boyer, Antoine Godichon-Baggioni and Viet Chi Tran for their knowledge on the classical deconvolution and for giving us important references on the subject.

\section{Description of the model and statement of the results}\label{Section:statement}
\subsection{Notations} \label{Section:notation}
In the sequel, $N$ is a positive number denoting the dimension of the matrices, $\mathbb{C}$ denotes the field of complex numbers, and $\mathbb{C}^+$ denotes the half-space of complex numbers with positive imaginary part. For $K>0$, we denote by $\mathbb{C}_{K}$ the half-space of complex numbers with imaginary part larger than $K$.

We write $\mathcal{H}_{N}(\mathbb{C})$ for the space of $N$-dimensional self-adjoint matrices. When $X\in\mathcal{H}_{N}(\mathbb{C})$, we denote by $X=X^{+}+X^{-}$ the unique decomposition of $X$ such that $X^{+}\geq 0$ and $X^{-}\leq 0$. The matrix $X^+$ is called the positive part of $X$ and $X^-$ its negative part. We recall that the normalized trace $\Tr(X)$ of $X$ is equal to $\frac{1}{N}\sum_{i=1}^NX_{ii}$. The resolvent $G_{X}$ of $G$ is defined on $\mathbb{C}^+$ by 
$$G_{X}(z)=(X-z)^{-1}.$$

When $X\in\mathcal{H}_{N}(\mathbb{C})$, we denote by $\lambda_{1}^{X},\ldots,\lambda_{N}^{X}$ its eigenvalues and by 
$$\mu_X=\frac{1}{N}\sum_{i=1}^{N}\delta_{\lambda_{i}^X}$$
its spectral distribution. We use the convention to use capital letters to denotes matrices, and corresponding small letter with index $i\in \mathbb{N}$ to denotes the $i$-th moment of the corresponding spectral distribution, when it is defined. For example, if $X$ is Hermitian and $i\in \mathbb{N}$, then
$$x_{i}=\Tr(X^i)=\frac{1}{N}\sum_{i=1}^N\lambda_{i}^X.$$
We also write $x_i^0$ for the $i-$th centered moment of $X$, namely
$$x^0_i=\Tr((X-\Tr(X))^i).$$
In particular, $x_1^0=0$ and $x_2^0=\Var(\mu_X)$, the variance of $\mu_X$. Finally, we write $\sigma_X=\sqrt{\Var(\mu_X)}$ for the standard deviation of $\mu_X$, $\theta_{X}=\frac{x_4^0}{\sigma_X^4}$ for the kurtosis of $X$ and $x_{\infty}$ for the norm of $X$.

When $\mu$ is a probability distribution on $\mathbb{R}$ and $f:\mathbb{R}\rightarrow \mathbb{R}$ is a measurable function, we set $\mu(f)=\int_{\mathbb{R}}f(t)d\mu(t)$ and we write $\mu(k)$ for the $k$-th moment of $\mu$, when it is well defined. When $\mu$ admits moments of order $2$, we denote by $\Var(\mu)=\mu(2)-\mu(1)^2$ the variance of $\mu$. The Stieltjes transform of a probability measure $\mu$ is the analytic function defined on $\mathbb{C}^+$ by
$$m_{\mu}(z)=\int_{\mathbb{R}}\frac{1}{t-z}d\mu(t).$$
In the special case where $\mu=\mu_X$ for some Hermitian matrix $X$, we simply write $m_{X}$ instead of $m_{\mu_X}$. 
\subsection{Unitarily invariant model and reduction of the problem}
The main topic of this paper is the estimation of the spectral density of a matrix which is modified by an additive or multiplicative matricial noise. We fix a Hermitian matrix $B=B^* \in\mathcal{M}_{N}(\mathbb{C})$, the \textit{signal matrix}. We denote by $\lambda_{1},\ldots,\lambda_{N}$ its eigenvalues and by $\mu_{B}=\frac{1}{N}\sum_{i=1}^N\delta_{\lambda_{i}}$ its spectral distribution. Additionally, we consider a random Hermitian matrix $A\in \mathcal{M}_{N}(\mathbb{C})$, the \textit{noise matrix}, whose spectral distribution $\mu_A$ is therefore random. We suppose that the random distribution $\mu_A$ satisfies the following properties.
\begin{condition}\label{concentration_noise}
There exists a \textit{known} probability measure $\mu_1$ with moments of order $6$ and a constant $C_{A}>0$ such that :
\begin{enumerate}
\item $\mu_1(1)=0$ in the additive case and $\mu_1(1)=1$ in the multiplicative case,
\item there exists a constant $c>0$ such that 
$$\vert a_i\vert\leq \left(1+\frac{c}{\sqrt{N}}\right)^{i}\vert \mu_1(i)\vert,$$
for $1\leq i\leq 6$, where we recall that $a_{i}=\mu_A(i)=\Tr(A^i)$, and
\item for any $C^1$ function $f:\mathbb{R}\rightarrow \mathbb{C}$,
\begin{equation*}
\mathbb{E}(\vert \mu_A(f)-\mu_1(f)\vert^2)\leq \frac{C_{A}^2\mathbb{E}\Vert \nabla f \Vert_{2}^2}{N}, 
\end{equation*}
where $f$ is considered as a function from $\mathcal{H}_{N}(\mathbb{C})\rightarrow \mathbb{C}$ with $f(A)=\frac{1}{N}\sum_{i=1}^Nf(\lambda_{i}^A)$, and $\mathbb{E}$ denotes the expectation with respect to the random matrix $A$.
\end{enumerate} 
\end{condition}
The first assumption of Condition \ref{concentration_noise} is a simple scaling to simplify the formulas of the manuscript. The second assumption is mostly technical, and can be relaxed at the cost of coarsening the concentration bounds. Indeed, we use several constants involving moments of the unknown distribution $\mu_A$, and the bounding assumption of Condition \ref{concentration_noise} allows us to use the moments of $\mu_1$ instead. This bound generally holds with probability $1-\exp(-c'N)$ for some $c'$ depending on the moment and on the class of matrix model. Finally, the last condition is usually also satisfied in most known cases. See \cite{GuiZei} for concentrations inequalities in the case where $A$ is either Wigner or Wishart (see also \cite[Section 4.4.1]{AnGuZe}). Then, we consider the additive problem
\begin{problem}[Additive case]\label{problem:additif}
Given $H=B+UAU^*$ with $U$ Haar unitary, $\mu_A(1)=\mu_B(1)=0$ and $\mu_A$ satisfying Condition \ref{concentration_noise}, reconstruct $\mu_B$,
\end{problem}
and the multiplicative one,
\begin{problem}[Multiplicative case]\label{problem:multiplicatif}
Given $M=A^{1/2}UBU^{*}A^{1/2}$ with $A,B\geq 0$, $U$ Haar unitary, $\mu_A(1)=\mu_B(1)=1$ and $\mu_A$ satisfying Condition \ref{concentration_noise}, reconstruct $\mu_B$.
\end{problem}
The normalization on $\mu_A(1),\,\mu_B(1)$ can easily be removed and its only role is to simplify the formulas of the manuscript. Our main assumption is therefore that the distribution of the noise is unitarily invariant. This is a sufficient condition to ensure asymptotic freeness between the unknown matrix $B$ and the noise $UAU^*$, see Section \ref{Section:free_probabilistic_background}.  We could as well assume orthogonal invariance with the same results, up to a numerical constant. 

Note that in the multiplicative case, the more general model $M=TUBU^*T^*\in \mathcal{M}_{N'}(\mathbb{C})$, where $T\in \mathcal{M}_{N',N}(\mathbb{C})$ is a random matrix  with $N'\in \mathbb{N}$ and $B$ is Hermitian without the positivity assumption, can be reduced to the one stated above. Writing $A=T^*T$, then the spectral distribution of $M$ is also equal to
$$\mu_{TUBUT^*}=\frac{N}{N'}\mu_{A^{1/2}UBU^*A^{1/2}}+\frac{N'-N}{N'}\delta_{0},$$
and we can up to a shift by a known constant assume that $T=A^{1/2}$. Hence, in the multiplicative case, we can assume without loss of generality that $M=A^{1/2}UBU^*A^{1/2}$ with $A\geq 0$ (not necessarily invertible). Then, since the positive part $M^{+}$ of $M$ is equal to $A^{1/2}UB^{+}U^*A^{1/2}$, and the negative part $M^{-}$ of $M$ is equal to $A^{1/2}UB^{-}UA^{1/2}$, we can directly separate the recovery of $B^+$ and $B^-$ at the level of $M$. Hence, we can assume that $B\geq 0$ and $M=A^{1/2}UBU^*A^{1/2}$ with $A,B\geq 0$. 
 
\subsection{Deconvolution procedure}\label{Section:statement_deconvolution}
We now explain the deconvolution procedure leading to an estimator $\widehat{\mu_{B}}$ of $\mu_B$. This deconvolution is done in two steps. The first step is to build an estimator $\widehat{\mathcal{C}_{B}}$ of the classical convolution $\mathcal{C}_{B}=\mu_B\ast Cauchy[\eta]$ of $\mu_B$ with a Cauchy distribution $Cauchy[\eta]$ of parameter $\eta$. We recall that 
$$dCauchy[\eta](t)=\frac{1}{\pi}\frac{\eta}{t^2+\eta^2},$$
for $t\in\mathbb{R}$. The estimator only exists for $\eta$ larger that some threshold depending on the moments of the noise (and also on ones of $B$ in the multiplicative case). Then, the second step is to build an estimator $\widehat{\mu_B}$ of $\mu_B$ from $\widehat{\mathcal{C}_B}$ by simply doing the classical deconvolution of $\widehat{\mathcal{C}_B}$ by the noise $Cauchy[\eta]$. The first step is quite new \cite{ATV} and requires complex analytic tools. Recall the Stieltjes inversion formula, which says that for $t\in\mathbb{R},$ 
$$\mathcal{C}_{B}(t)=\frac{1}{\pi}\Im m_{B}(t+i\eta),$$
where $m_B$ is the Stieltjes transform of $\mu_B$ introduced in Section \ref{Section:notation}. Using this formula, we build $\widehat{\mathcal{C}_B}$ by first constructing an estimator of $m_B$ which exists on the upper half-plane $\mathbb{C}_{\eta}$. In the additive case, we can simply take $\eta=2\sqrt{2\Var(\mu_1)}$, while the multiplicative case is more complicated, due to the higher instability of the free convolution. 

\subsubsection*{Additive case}
Set $\sigma_1=\sqrt{\Var(\mu_1)}$ and consider the additive case $H=B+UAU^*$. Then, we have the following convergence result from \cite{ATV}.
\begin{theorem}\label{decon_additif_first_step}\cite{ATV}
There exist two analytic functions $\omega_1,\omega_3:\mathbb{C}_{2\sqrt{2}\sigma_1}\rightarrow\mathbb{C}^+$ such that for all $z\in\mathbb{C}_{2\sqrt{2}\sigma_1}$, 
\begin{itemize}
\item $\Im \omega_1(z)\geq \frac{\Im z}{2}, \Im \omega_3(z)\geq \frac{3\Im z}{4},$
\item $\omega_1(z)+z=\omega_3(z)-\frac{1}{m_{\mu_1}(\omega_1(z))}=\omega_3(z)-\frac{1}{m_{H}(\omega_3(z))}$.
\end{itemize}
Moreover, setting $h_{\mu_1}(w)=-w-\frac{1}{m_{\mu}(w)}$, $\omega_3(z)$ is the unique fixed point of the function $K_{z}(w)=z-h_{\mu_1}(w-\frac{1}{m_{H}(w)}-z)$ in $\mathbb{C}_{3\Im(z)/4}$ and we have 
$$\omega_3(z)=\lim K_{z}^{\circ n}(w),$$
for all $w\in \mathbb{C}_{3/4\Im(z)}$.
\end{theorem}
The last part of this theorem is important, since it yields a concrete method to build the function $\omega_3$ by iteration of the map $K_{z}$. This iteration converges quickly because it is a contraction of the considered domain with respect to the Schwartz distance. The constant $2\sqrt{2}$ has been improved  to $2$ in \cite{Mal} in the case where $\mu_1$ is a semi-circular distribution. The above theorem leads then to the construction of $\widehat{\mathcal{C}_{B}}$. 
\begin{definition}
The additive Cauchy estimator of $\mu_B$ at $t\in\mathbb{R}$ is 
$$\widehat{\mathcal{C}_{B}}(t)=\frac{1}{\pi}\Im\left[m_{H}(\omega_3(t+2\sqrt{2}\sigma_1i))\right],$$
where $\omega_3$ is defined in Theorem \ref{decon_additif_first_step}.
\end{definition}
Let us explain the intuition behind this definition. The functions $\omega_1,\omega_3$ are called subordination functions of the free deconvolution for the following reason : suppose that $\mu_H=\mu_1\boxplus\mu_B$ (in the sense of Section \ref{Subsection:FreeConvolution}), then $m_{\mu_B}(z)=m_{H}(\omega_3(z))=m_{\mu_1}(\omega_1(z))$ for all $z\in\mathbb{C}_{2\sqrt{2}\sigma_1}$ (see Section \ref{Subsection:FreeConvolution}). We never have the exact relation $\mu_H=\mu_1\boxplus\mu_B$, but by Theorem \ref{Voiculescu_Speicher} $\mu_H\simeq \mu_A\boxplus\mu_B$ and by Condition \ref{concentration_noise}, $\mu_A\simeq \mu_1$; hence we have the approximate free convolution $\mu_H\simeq \mu_1\boxplus \mu_B$, and thus $m_{\mu_B}(z)\simeq m_H(\omega_3(z))$ on $\mathbb{C}_{2\sqrt{2}\sigma_1}$. Then, taking the imaginary part gives the approximated value of $\mathcal{C}_{B}$.
\subsubsection*{Multiplicative case}
The nice property of the additive case is that the domain on which the fixed point procedure works is relatively well described by $\sigma_1$, which measures the magnitude of $\mu_1$. In the multiplicative case $M=A^{1/2}UBU^*A^{1/2}$, the fixed point method is not so efficient (see the bound in \cite[Proposition 3.4]{ATV}). We propose here a different approach which yields better results at a cost of increased complexity. In the multiplicative case, we are looking for subordination functions $\omega_1(z)$ and $\omega_3(z)$ satisfying the relations
\begin{equation}\label{equation_subordination_multi}
z\omega_1(z) =\omega_3(z)\frac{\omega_3(z)m_{M}(\omega_3(z))}{1+\omega_3(z)m_{M}(\omega_3(z))}=\omega_3(z)\frac{\omega_1(z)m_{\mu_1}(\omega_1(z))}{1+\omega_1(z)m_{\mu_1}(\omega_1(z))}.
\end{equation}
Equation \eqref{equation_subordination_multi} is more unstable than in the additive case, and thus the region $\mathbb{C}_{K}$ on which it can be solved depends on higher moments of $\mu_1$ and $\mu_M$. Set 
$$\tilde{\sigma}_1^2=\mu_1(3)\mu_1(1)-\mu_1(2)^2,\, \sigma_M^2=h_2-h_1^2\,\text {  and  }\,\tilde{\sigma}_M^2=h_3h_1-h_2^2.$$
Then for $t\geq 2$, define
\begin{equation}\label{definition_g}
g(\xi)=\xi+\frac{1}{k(\xi)}\left(1+\left(\frac{1}{k(\xi)}+\frac{\vert\sigma_M^2-\sigma_1^2\vert}{k(\xi)\tilde{\sigma_1}}+\frac{\tilde{\sigma}_M^2}{\tilde{\sigma}_1^2\xi}\right)\left(\frac{\sigma_1^2}{\tilde{\sigma}_1}+\frac{1}{k(\xi)}\right)\right),
\end{equation}
where $k(t)=\frac{t+\sqrt{t^2-4}}{2}$ is real and greater than $1$ for $t\geq 2$, and set
\begin{equation}\label{definition_t}
t(\xi)=\left(\frac{\sigma_1^2}{k(\xi)\tilde{\sigma}_1}+\frac{\tilde{\sigma}_1^2+\sigma_1^4/2}{(k(\xi)\tilde{\sigma}_1)^2}\right)\left(2+\frac{\sigma_M^2}{\xi\tilde{\sigma}_1}+\frac{\tilde{\sigma}_M^2+\sigma_M^4/2}{\xi^2\tilde{\sigma}_1^2}\right).
\end{equation}
The function $g$ controls the imaginary part of the multiplicative subordination function $\omega_3(z)$ by the one of $z$ (see Lemma \ref{upper_bound_z_multi}), whereas the function $t$ controls the stability of the subordination equation \eqref{equation_subordination_multi} according to the imaginary part of $\omega_3(z)$ (see Lemma \ref{definition_domain_tildeh}). A quick computation shows that $g'$ is strictly increasing and tends to $1$ at infinity, so that there exists a maximal interval $[\xi_g,\infty[\subset [2,+\infty[$ on which $g$ is strictly increasing. Hence, we can define $g^{-1}$ on $[g(\xi_g),\infty[$. Moreover, $t$ is decreasing in $\xi$ and converges to $0$ as $\xi$ goes to infinity, and thus we can define $\xi_{0}$ as
$$\xi_0=\inf\left(\xi\geq \xi_{g},t(\xi)<1\right).$$
The reader should refer to Appendix \ref{Appendix:constants} for a quick overview of the constants involved in the following theorem.
\begin{theorem}\label{subordination_deconvolution_multi}
There exist two analytic functions $\omega_1,\omega_3:\mathbb{C}_{g(\xi_0)\tilde{\sigma}_1}\rightarrow\mathbb{C}^+$ such that 
$$z\omega_1(z) =\omega_3(z)\frac{\omega_3(z)m_{M}(\omega_3(z))}{1+\omega_3(z)m_{M}(\omega_3(z))}=\omega_3(z)\frac{\omega_1(z)m_{\mu_1}(\omega_1(z))}{1+\omega_1(z)m_{\mu_1}(\omega_1(z))}$$ 
for all $z\in \mathbb{C}_{g(\xi_0)\tilde{\sigma}_1}$. Moreover, setting $K_{z}(w)=-h_{\mu_1}\left(w^2\frac{ m_{M}(w)}{1+w m_{M}(w)}/z\right)z$ for $z \in\mathbb{C}_{g(\xi_0)\tilde{\sigma_1}}$ and $w\in \mathbb{C}^+$, then
\begin{enumerate}
\item if $\Re z< -K_0$ with $K_0$ given in Lemma \ref{convergence_negative_real}, then 
$$\omega_3(z)=\lim_{n\rightarrow \infty} K_{z}^{\circ n}(z),$$
\item if $z\in \mathbb{C}_{g(\xi_0)\tilde{\sigma}_1}$, then for all $z'\in \mathbb{C}_{g(\xi_0)\tilde{\sigma}_1}\cap B(z,R(g^{-1}(\Im z)))$, with $R(g^{-1}(\Im(z)))>0$ given in \eqref{definition_R},
$$\omega_3(z')=\lim_{n\rightarrow \infty} K_{z'}^{\circ n}(\omega_3(z)).$$
\end{enumerate}
\end{theorem}
To summarise the second part of the latter theorem, we can construct $\omega_3$ on $\mathbb{C}_{g(\xi_0)\tilde{\sigma}_1}$ by applying the first fixed point procedure for negative real parts far enough from zero, and then move to increasing real parts with the second fixed point procedure. The quantity $g(\xi_0)$ plays a similar role as the constant $2\sqrt{2}$ in the additive case, the important change being that $g(\xi_0)$ now depends on the moments of $A$ and $B$. The proof of this theorem is postponed to Appendix \ref{Appendix:proof_multiplicative_subordination}. We deduce from the latter theorem a definition of $\widehat{\mathcal{C}_B}$ at some $\eta>g(\xi_0)\tilde{\sigma}_1$. 
\begin{definition}
The multiplicative Cauchy estimator of $\mu_B$ at $\eta>g(\xi_0)\tilde{\sigma}_1$ is the function $\widehat{\mathcal{C}_{B}}[\eta]$ whose value at $t\in\mathbb{R}$ is 
$$\widehat{\mathcal{C}_{B}}[\eta](t)=\frac{1}{\pi}\Im\left[\frac{\omega_3(t+i\eta)}{t+i\eta}m_{M}(\omega_3(t+i\eta))\right],$$
where $\omega_3,g$ and $\xi_0$ are defined above.
\end{definition}
An intuitive explanation of this construction using free probability can be given like in the additive case. One difference with the additive case is the more complicated subordination relation $\omega_3(z)m_{\mu_3}(\omega_3(z))=zm_{\mu_2}(z)$ when $\mu_3=\mu_1\boxtimes\mu_2$. This explains the change in the formula of $\widehat{\mathcal{C}_B}[\eta]$.
\subsubsection*{Estimating the distribution $\mu_B$}
The last step is to recover $\mu_B$ from $\widehat{\mathcal{C}_{B}}[\eta]$ (simply written $\widehat{\mathcal{C}_B}$ in the additive case), which is a classical deconvolution of $\widehat{\mathcal{C}_{B}}[\eta]$ by the Cauchy distribution $Cauchy[\eta]$. This is a classical problem in statistic which has been deeply studied since the first results of Fan \cite{Fan1}. The main feature of our situation is the supersmooth aspect of the Cauchy distribution. In particular, the convergence of the deconvolution may be very slow depending on the smoothness of the original measure. There are two main situations, which are solved differently :
\begin{itemize}
\item the original measure $\mu_B$ is \textit{sparse}, meaning that it consists of few atoms. In this case, one solves the deconvolution problem by solving the Beurling LASSO problem
\begin{equation}\label{deconvolution_classical_atom}
\widehat{\mu_B}=\argmin_{\mu\in \mathcal{M}(\mathbb{R})}\Vert \mu\ast Cauchy(\eta)-\widehat{\mathcal{C}_{B}}[\eta]\Vert_{L^2}^2+\lambda \mu(\mathbb{R}),
\end{equation}
where $\mathcal{M}(\mathbb{R})$ denotes the space of positive measures on $\mathbb{R}$, and $\lambda>0$ is a parameter to tune depending on the expected distance between $\widehat{C}_{B}[\eta]$ and $\mu\ast Cauchy(\eta)$ (see \cite{DuPe} for more information on the choice of $\lambda$). This minimization problem can be solved by a constrained quadratic programming method (see \cite{BoVa}). The constraints of the domain on which the minimization is achieved actually enforces the sparsity of the solution.
\item the original $\mu_B$ is close to a probability distribution with a density in $L^2(\mathbb{R})$ : in this case, it is better to take a Fourier approach. The convolution of $\mu_B$ by a Cauchy distribution $Cauchy(\eta)$ on $L^2(\mathbb{R})$ is a multiplication of $\mathcal{F}(\mu_B)$ by the map $\xi\mapsto e^{-\eta\vert \xi\vert}$. Hence, a naive estimator of $d\mu_B$ would be to consider the estimator $\widehat{\mu_B}=\mathcal{F}^{-1}(H_{\eta})$, where $H_{\eta}(\xi)=e^{\eta \vert \xi\vert}\mathcal{F}(\widehat{\mathcal{C}_{B}}[\eta])$. This estimator does not work properly due to the fast divergence of the map $ \xi\mapsto e^{\eta\vert \xi\vert}$. A usual way to circumvent this problem is to consider instead the estimator
\begin{equation}\label{deconvolution_classical_density}
\widehat{\mu_B}=\mathcal{F}^{-1}(K_{\epsilon}H_{\eta}),
\end{equation}
where $K_{\epsilon}$ is a regularizing kernel depending on a parameter $\epsilon$ to choose. For example, one can simply take $k_{\epsilon}=\mathbf{1}_{[-\epsilon^{-1},\epsilon^{-1}]}$ with $\epsilon$ a function pf $\eta$ and $\mathbb{E}\left\Vert \widehat{\mathcal{C}_{B}}[\eta]-\mathcal{C}_{B}[\eta]\right\Vert_{L^2}^2$. The regularizing kernel allows to reduce the instability in the higher modes of the Fourier transform, at the cost of loosing some information on the density to estimate. See \cite{La} for an explicit method to choose $\epsilon$ given $\eta$ and the bound on $\mathbb{E}\left\Vert \widehat{\mathcal{C}_{B}}[\eta]-\mathcal{C}_{B}[\eta]\right\Vert_{L^2}^2$ that is provided in the next section. Several more advanced techniques (see for example \cite{Hu} for density with compact support) can also be used for refined results.
\end{itemize}

\subsection{Concentration bounds}\label{Section:concentration_bound}
Recall that
$\mathcal{C}_{B}[\eta]=\mu_B\ast Cauchy(\eta)$. We now state the concentration bounds for the estimators we constructed before. Our inequalities involve moments of $A$ and $B$ up to order $6$ in the additive case, and also the infinite norm of $A$ in the multiplicative case. There are several constants involved in the following results. We chose to avoid any simplification which would hinder the accuracy of the constants or restrict their domain of validity, since any numerical computing environment can easily compute the expressions obtained. Despite some increased complexity, the simulations in the next section show some promising result on the precision in known cases. The reader should refer to Appendix \ref{Appendix:constants} to get a full picture of the constants involved.
\begin{theorem}[Additive case]\label{Theorem:concentration_cauchy_additif}
Suppose that $N^2\geq C_{threshold}$, with
$$C_{threshold}=\frac{2\sqrt{2}\max(C_{thres,A}(3\sigma_1/\sqrt{2}),C_{thres,B}(3\sigma_1/\sqrt{2}))}{3^3\sigma_1^3}.$$
Then, 
\begin{align*}
&MSE:=\mathbb{E}\left(\Vert\widehat{\mathcal{C}_{B}}-\mathcal{C}_{B}\Vert_{L^2}^2\right)\\
\leq& \frac{1}{2\sqrt{2}\pi\sigma_1 N^2}\left(\frac{C_2(2\sqrt{2})C_A\left(1+\frac{(1+c/N)\sqrt{\mu_1(2)}}{\sqrt{2}\sigma_1}\right)}{\sqrt{2}\sigma_1}+\frac{4C_{3}(2\sqrt{2})}{3\sigma_1}\sqrt{\sigma_A^2+2\frac{\sigma_A^2\sigma_B^2+a_4}{3^2\sigma_1^2}}+\frac{C_{1}(2\sqrt{2})}{N}\right)^2,
\end{align*}
with the functions $C_{1},\,C_{2}$ and $C_{3}$ respectively given in \eqref{first_constant_additif}, \eqref{second_constant_additif} and \eqref{third_constant_additif}, and $C_{thres,A},\,C_{thres,B}$ given in Proposition \ref{minoration_omega}.
\end{theorem}
In the multiplicative case, we have the following concentration bound which holds for any $\eta>g(\xi_0)\tilde{\sigma}_1$.
\begin{theorem}[Multiplicative case]\label{Theorem:concentration_cauchy_multiplicatif}
Let $\eta=\kappa\tilde{\sigma}_1$ with $\kappa>g(\xi_0)$, and  suppose that $N^2\geq C_{threshold}$, with
$$C_{threshold}=\frac{2\kappa\max(C_{thres,A}(g^{-1}(\kappa)\tilde{\sigma}_1),\,C_{thres,B}(g^{-1}(\kappa)\tilde{\sigma}_1))}{\xi^3\tilde{\sigma}_1^2}\left(1+\frac{1}{k\circ g^{-1}(\kappa)}\right).$$
Then, 
\begin{align*}
MSE:=&\mathbb{E}(\Vert\widehat{\mathcal{C}}_{B}[\eta]-\mathcal{C}_{B}[\eta]\Vert_{L^2}^2)\\
\leq & \frac{1}{\kappa\pi\tilde{\sigma}_1N^2}\left(\frac{3C_2(\kappa)C_A\left(1+\frac{3(1+c/N)\sqrt{\mu_1(2)}}{2g^{-1}(\kappa)\tilde{\sigma_1}}\right)}{2g^{-1}(\kappa)\tilde{\sigma_1} }+\frac{C_3(\kappa)\sqrt{\Delta(\kappa)}}{g^{-1}(\kappa)\tilde{\sigma}_1}+\frac{C_1(\kappa)}{N}\right)^2+\frac{C_4(\kappa)}{N^6},
\end{align*}
where $C_{1}(\kappa),C_{2}(\kappa),\, C_{3}(\kappa)$ and $C_{4}(\kappa)$ are respectively given in \eqref{first_constant_multi}, \eqref{second_constant_multi}, \eqref{third_constant_multi} and \eqref{definition_threshold_bound}, $\Delta(\kappa)$ is given in \eqref{definition_delta_kappa} and $C_{thres,A}(g^{-1}(\kappa)\tilde{\sigma}_1)$, $C_{thres,B}(g^{-1}(\kappa)\tilde{\sigma}_1)$ are given in Proposition \ref{minoration_omega_multiplicative} and Proposition \ref{minoration_omega_B_multi}.
\end{theorem}
\subsection{Accuracy of the classical deconvolution}\label{Section:Accuracy_classical} Concentration properties of the classical deconvolution are already known, in the atomic or in the continuous case. We quickly review some general results in this framework, since we plan to deeper study this question in a forthcoming paper \cite{JoTa}. 
\begin{itemize}
\item In the atomic case, the precision of the deconvolution depends on the number $m$ of atoms and on the minimum separation $t=\min \lbrace\vert x-x'\vert, x,x'\in Supp(\mu_B)\rbrace$ between atoms. There exist then constants $C(\eta,m),\Delta(\eta)$ such that for $t\geq\Delta(\eta)$ (see \cite{DuPe, Ben}),
$$\mathbb{E}(\vert \widehat{\mu_B}-\mu_2\vert_{W_2(\mathbb{R)}}^2)\leq C(\eta,m) \frac{MSE}{mt^{4m-2}},$$
where $W_1$ denotes the Wasserstein distance. Two important remarks have to be done on the limitations of this result. First, the exponent $m$ in the error term shows that the recovery of $\mu_B$ is very hard when $m$ is large, whence the sparsity hypothesis of the data. This can directly be seen at the level of the deconvolution procedure \eqref{deconvolution_classical_atom}, since the $L^1$-penalization generally yields a result with few atoms.  More importantly, the threshold $\Delta$ is a up to a constant the inverse of the Nyquist frequency of a low pass filter with a cut-off in the frequency domain around $\frac{1}{\eta}$. Hence, the resolution of the deconvolution depends dramatically on the imaginary line $i\eta$ on which the first step of the deconvolution is done. This limit can be overcome when we assume that the signal is clustered around a certain value, see \cite{DeDuPe} for such results for the recovery of positive measures in this case.
\item In the continuous case, Fan already gave in \cite{Fan2} first bounds for the deconvolution by a supersmooth noise, when the expected density $d\mu_B$ of $\mu_B$ is assumed $C^k$ for some $k>0$. Due to the exponential decay of the Fourier transform of the noise, the rate of convergence is logarithmic. Later, Lacour \cite{La} proved that choosing appropriately the parameter $\epsilon$ in the deconvolution procedure leads to a convergence with power decay in $N$ in the case where the density is analytic, with an exponent depending on the complex domain on which the density can be analytically extended. This yields the following inequality, from whom the accuracy of the deconvolution can be deduced ; suppose that $d_{W_1}(\mu_B,\mu_f)\leq \delta$, with $\mu_f$ being a probability distribution with density $f$. Then, with $\widehat{\mu_B}$ defined in \eqref{deconvolution_classical_density},
\begin{enumerate}
\item if $f$ is $C^k$, with $\Vert f^{(k)}\Vert_{L^2}\leq K$, then there exists $C(K,\eta)$ such that 
$$d_{W_2}(\widehat{\mu_B},\mu_2)\leq \delta +\frac{C(K,\eta)}{\left\vert\log\left(\left\Vert \widehat{\mathcal{C}_B}[\eta]-\mathcal{C}_{B}[\eta]\right\Vert_{L^2}^2+\frac{\delta^2}{\eta^2}\right)\right\vert^{k}},$$ 
\item and if $f$ can be analytically extended to the complex strip $\{x+iy, -a<y<a\}$, and $\Vert f(\cdot+iy)\Vert_{L^2}\leq K$ for all $-a<y<a$, then there exists $C(a,K,\eta)$ such that  
$$d_{W_1}(\widehat{\mu_{B}},\mu_2)\leq \delta +C(a,K,\eta)\left\vert\left\Vert \widehat{\mathcal{C}_B}[\eta]-\mathcal{C}_{B}[\eta]\right\Vert_{L^2}^2+\frac{\delta^2}{\eta^2}\right\vert^{\frac{a}{2(a+\eta)}},$$
and a mean squared estimate can be deduced from the above bound. Improved bounds also exist when more regularity is assumed (see \cite[Theorem 3.1]{La}). From example, if $\mu_B$ is the discretization of the Gaussian density, so that $\delta\simeq \frac{1}{N}$, then $d_{W_1}(\widehat{\mu_B},\mu_2)$ shrinks almost linearly with $\left\Vert \widehat{\mathcal{C}_B}[\eta]-\mu_B\ast Cauchy(\eta)\right\Vert_{L^2}$.
\end{enumerate}
\end{itemize}
\subsection{Simulations}\label{Section:Simulations} We provide here some simulations to show the accuracy and limits of the concentration bounds we found on the mean squared error in Section \ref{Section:concentration_bound}. In the additive and multiplicative cases, we take an example, perform the first step of the deconvolution as explained in Section \ref{Section:statement_deconvolution} and compute the error with $\mathcal{C}_{B}(\eta)$, and then compare this error with the constant we computed according to the formulas in Theorem \ref{Theorem:concentration_cauchy_additif} and Theorem \ref{Theorem:concentration_cauchy_multiplicatif}. 
\subsubsection*{Additive case} We consider a data matrix $B$ which is diagonal with iid entries following a real standard Gaussian distribution, and a noise matrix $A$ which follows a GUE distribution (namely, $A=(X+X^*)/\sqrt{2}$, with the entries of $X$ iid following a complex centered distribution with variance $1/N$). Hence, $\mu_A$ is close to a standard semi-circular distribution $\mu_1$ in the sense of Condition \ref{concentration_noise}. Then, we consider the additive model $H=B+UAU^*$ (even if the presence of $U$ is redundant, since the distribution of $A$ is already unitarily invariant). We performed the iteration procedure explained in Theorem \ref{decon_additif_first_step} at $\eta=2\sqrt{2}\sigma_1=2\sqrt{2}$. In Figure \ref{Fig:decon_add}, we show an example of the spectral distribution of $H$, the result of the first step of the deconvolution, and then the result of the deconvolution after the classical deconvolution by a Cauchy distribution (we used here a constrained Tychonov method see \cite{Neu}), and a comparison with $\mu_B$. 
\begin{figure}[h!]
\includegraphics[scale=0.25]{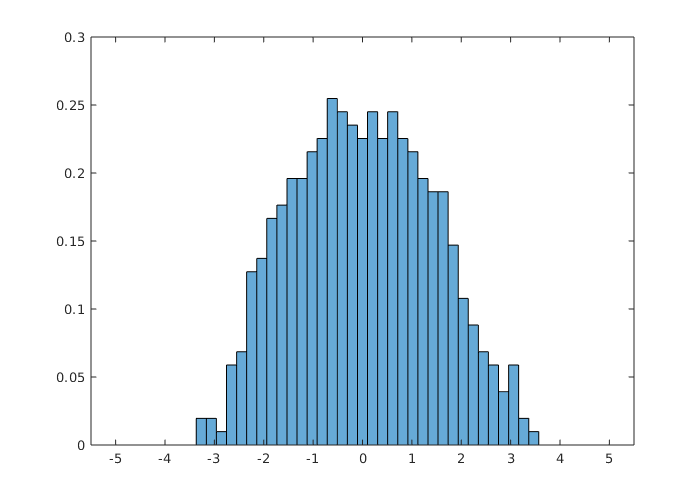}
\includegraphics[scale=0.25]{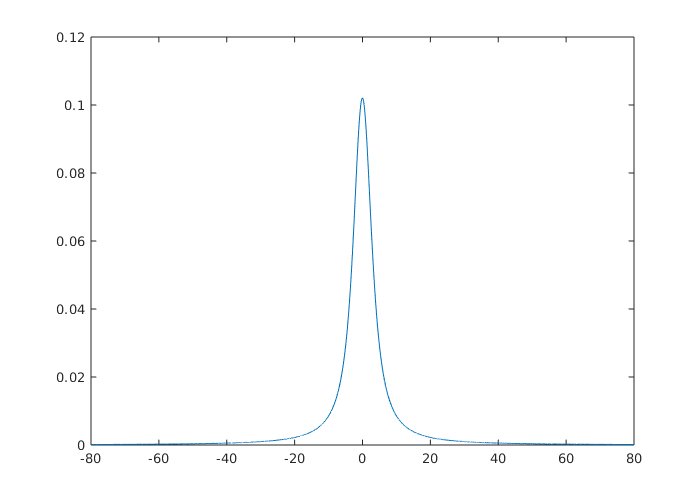}
\includegraphics[scale=0.25]{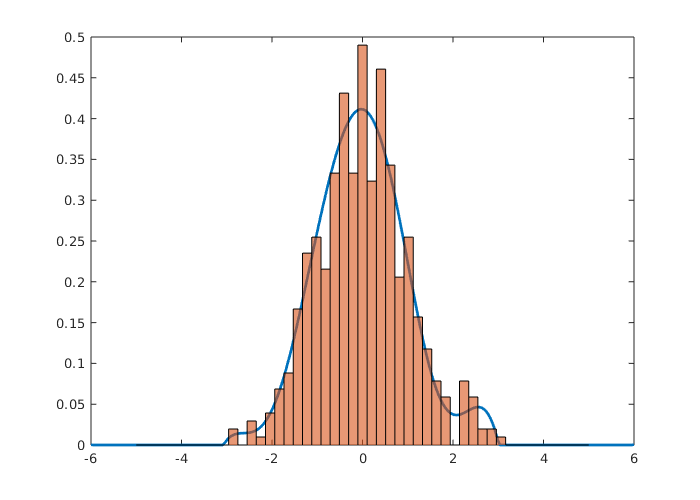}
\caption{\label{Fig:decon_add}Histogram of the eigenvalues of $H$, result of the first step of the deconvolution, result of the second step of the deconvolution and comparison with the histogram of $\mu_B$ ($N=500$).}
\end{figure}

The result is very accurate, which is not surprising due to the analyticity property of the Gaussian distribution (see the discussion in Section \ref{Section:Accuracy_classical}). Then, we simulate the standard error $\sqrt{MSE}$ with a sampling of deconvolutions with the size $N$ going from $50$ to $2000$. The lower bound on $N$ for the validity of Theorem \ref{Theorem:concentration_cauchy_additif} is $4$, which is directly satisfied.  We can then compare the simulated standard deviation to the square root of the bound given in Theorem \ref{Theorem:concentration_cauchy_additif}. The results are displayed in Figure \ref{Fig:comparison_error_add}. The first diagram is a graph of the estimated square root of $MSE$ and the second one is  the graph of the theoretical constant we computed according to $N$. The third graph is a ratio of both quantities according to $N$.

\begin{figure}[h!]
\includegraphics[scale=0.3]{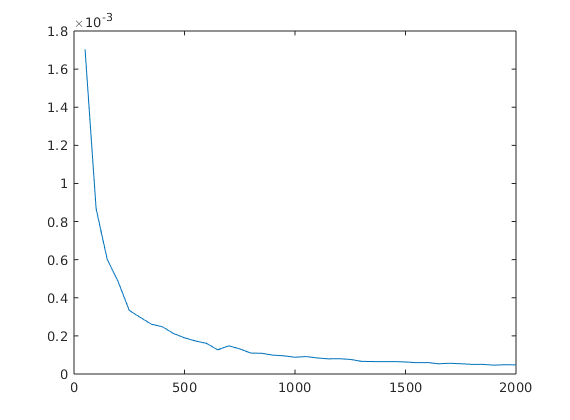}
\includegraphics[scale=0.3]{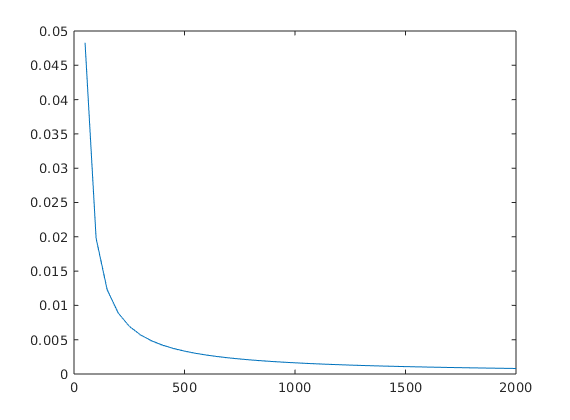}
\includegraphics[scale=0.3]{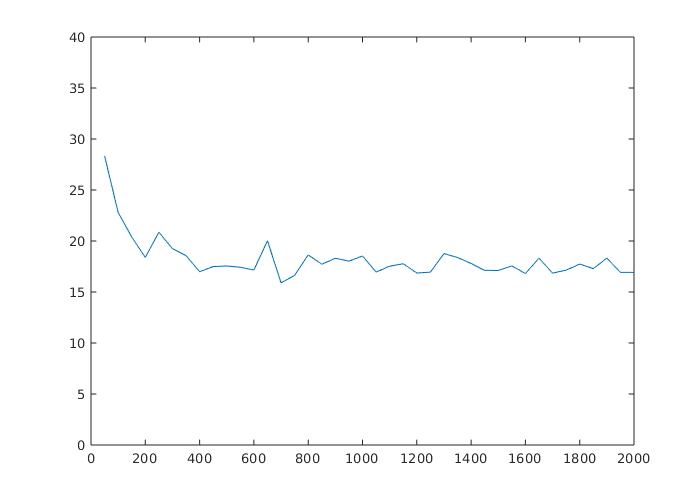}
\caption{\label{Fig:comparison_error_add} Simulation of $\sqrt{MSE}$ in the additive case for $N$ from $50$ to $2000$ (with a sampling of size $100$ for each size) , theoretical bound on $\sqrt{MSE}$ provided in Theorem \ref{Theorem:concentration_cauchy_additif}, and ratio of the theoretical bound on the simulated error.}
\end{figure}

We see that the error on the bound is better when $N$ is larger. When $N$ is small, the term $C_1N^{-1}$ is non negligible, and approximations in the concentration results of the subordination function in Section \ref{Section:bound_subordination} contribute to this higher ratio. When $N$ gets larger, the term $C_1N^{-1}$ vanishes and the ratio between the theoretical constant and the estimated error gets better. There is certainly room for improvement, even if this specific example may behave particularly well compared to the general case of Theorem \ref{Theorem:concentration_cauchy_additif}.
\subsubsection*{Multiplicative case} In the multiplicative case, we consider for the data matrix a shifted Wigner matrix $B=(X+X^*)/(2\sqrt{2})+1$, with the entries of $X$ iid following a complex centered distribution with variance $1/N$. Hence, $\mu_B$ is close to a semicircular distribution with center $1$ and variance $1/4$. Then, we consider a noise matrix $A=YY^*$, with $Y$ a square matrix of size $N$ iid following a complex centered distribution with variance $1/N$. Hence, $\mu_A$ is close to a Marchenko-Pastur distribution $\mu_1$ with parameter $1$ in the sense of Condition \ref{concentration_noise}. Then, we consider the multiplicative model $M=A^{1/2}UBU^{*}A^{1/2}$ and we apply the deconvolution procedure explained in Section \ref{Section:statement_deconvolution}. First, we compute $\xi_0\simeq 3.5$ and then $\eta_0=g(\xi_0)\tilde{\sigma}_1\simeq 4.1$. Remark that this constant is quite sharp, since in the simulations for this example the fixed point procedure converged until $\eta\simeq 3.6$. In Figure \ref{Fig:decon_multi}, we show an example of such a deconvolution, with the histogram of the eigenvalues of $M$, the first and second steps of the deconvolution and a comparison with $\mu_B$. Like in the additive case, the result is accurate thanks to the good analyticity property of the semi-circular distribution.
\begin{figure}[h!]
\includegraphics[scale=0.3]{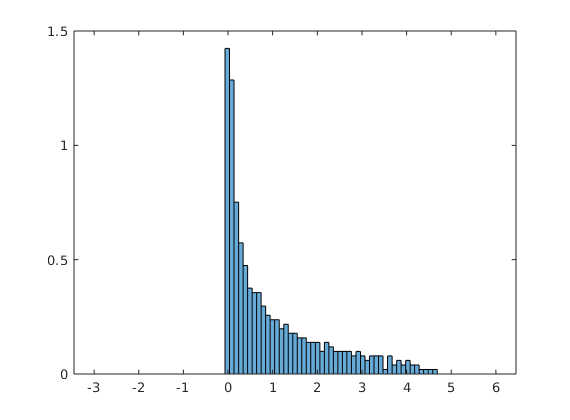}
\includegraphics[scale=0.25]{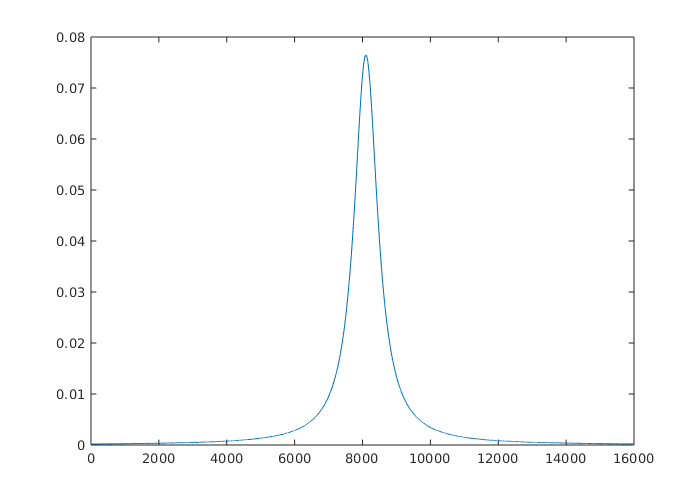}
\includegraphics[scale=0.3]{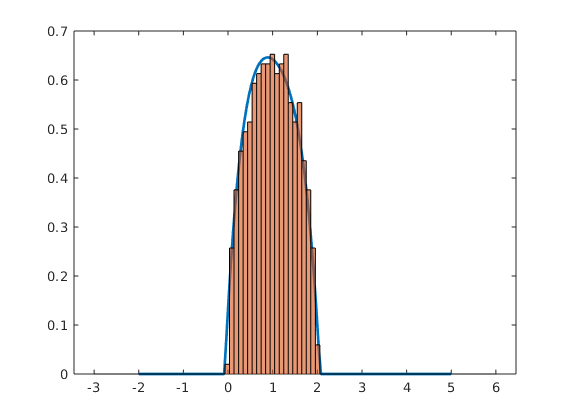}
\caption{\label{Fig:decon_multi}Histogram of the eigenvalues of $M$, result of the first step of the deconvolution, result of the second step of the deconvolution and comparison with the histogram $\mu_B$ ($N=500$).}
\end{figure}

Then, we do the same study than in the additive case. The lower bound on $N$ given in Theorem \ref{Theorem:concentration_cauchy_multiplicatif} is in our case $72$, hence we chose to compare the theoretical and simulated deviation for $N$ going from $100$ to $2000$. This gives the result depicted in Figure \ref{Fig:comparison_error_multi} (we follow the same convention than in the additive case).

\begin{figure}[h!]
\includegraphics[scale=0.3]{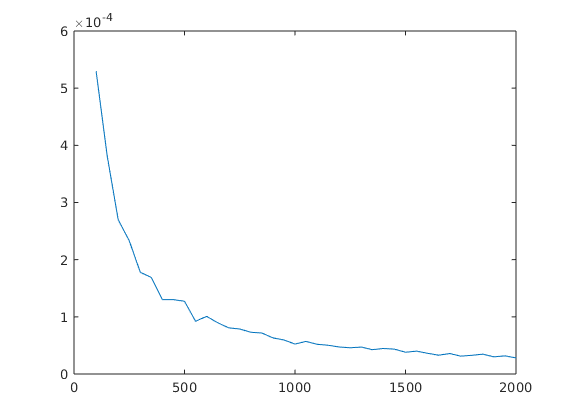}
\includegraphics[scale=0.3]{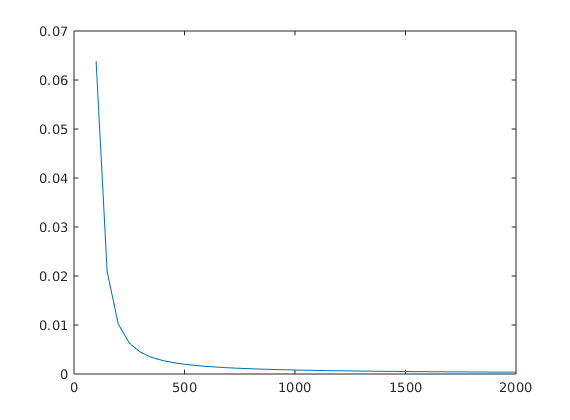}
\includegraphics[scale=0.3]{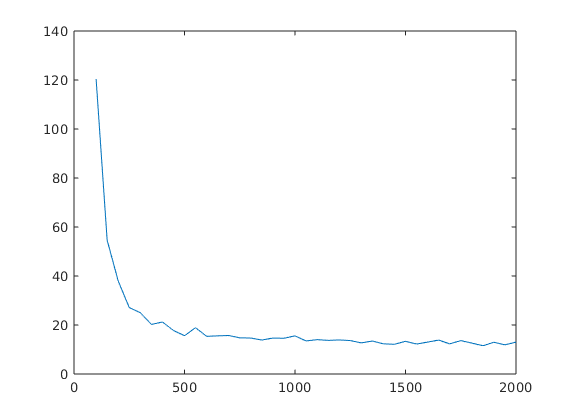}
\caption{\label{Fig:comparison_error_multi} Simulation of $\sqrt{MSE}$ in the multiplicative case for $N$ from $100$ to $2000$ (with a sampling of size $100$ for each size) , theoretical bound on $\sqrt{MSE}$ provided in Theorem \ref{Theorem:concentration_cauchy_additif}, and ratio of the theoretical bound on the simulated error.}
\end{figure}

The result is similar to the additive case, with a ratio which gets worse for small $N$. This can be explained by the more complicated study of the multiplicative case, which induces additional approximations.
\section{Unitarily invariant model and free convolution}\label{Section:free_probabilistic_background}
We introduce here necessary backgrounds for the proof of the theorems of this manuscript.
\subsection{Probability measures, cumulants and analytic transforms}
Let $\mu$ be a probability measure on $\mathbb{R}$.  Recall that $\mu(k)$ denotes the $k$-th moment of $\mu$, when it is defined.
\subsubsection{Free cumulants} Throughout this manuscript, free probability theory will be present without being really mentioned. In particular, several quantities involve free cumulants of probability measures and mixed moments of free random variables, which have been introduced by Speicher in \cite{Spe}. Since we will only use moments of low orders, we won't develop the general theory of free cumulants and the interested reader should refer to \cite{NiSp} for more information on the subject, in particular to learn about the non-crossing partitions picture explaining the formulas below. 

The free cumulant of order $r$ of $\mu$ is denoted by $k_r(\mu)$. In this paper, we use only the first three free cumulants, which are the following :
$$k_{1}(\mu)=\mu(1),\, k_{2}(\mu)=\Var(\mu)=\mu(2)-\mu(1)^2,\, k_3(\mu)=\mu(3)-3\mu(2)\mu(1)+2\mu(1)^3.$$
If $\mu,\mu'$ are two probability measures on $\mathbb{R}$ and $\vec{k},\,\vec{k}'$ are words of integers of length $r$ with $r>0$ we denote by $m_{\mu,\mu'}(\vec{k},\vec{k}')$ the mixed moments of $\mu_1,\mu_2$ when they are assumed in free position (see \cite{NiSp} for more background on free random variables). Once again, we only need the formulas of $m_{\mu,\mu'}(\vec{k},\vec{k}')$ for few values of $\vec{k},\vec{k'}$, which are as follow :
\begin{flalign*}
m_{\mu,\mu'}(k,k')=\mu(k)\mu'(k'),&&
\end{flalign*}
\begin{flalign*}
m_{\mu,\mu'}(k_1\cdot k_2,k'_1\cdot k'_2)=\mu(k_1+k_2)\mu'(k'_1)\mu'(k'_2)
&+\mu(k_1)\mu(k_2)\mu'(k'_1+k'_2)-\mu(k_1)\mu(k_2)\mu'(k'_1)\mu'(k'_2),&
\end{flalign*}
and, writing $1^3$ for the word $1\cdot1\cdot1$,
\begin{align*}
m_{\mu,\mu'}(k_1\cdot k_2\cdot k_3,1^3)=&\mu'(1)^3\mu(k_{1}+k_{2}+k_{3})\\
+&\mu'(1)\Var(\mu')\Big(\mu(k_1+k_2)\mu(k_3)+\mu(k_2+k_3)\mu(k_1)+\mu(k_3+k_1)\mu(k_2)\Big)\\
+&k_{3}(\mu')\mu(k_1)\mu(k_2)\mu(k_3).
\end{align*}
By abuse of notation, we simply write $k_{r}(X)$ for $k_{r}(\mu_X)$ and $m_{X,X'}(\vec{k},\vec{k}')$ for $m_{\mu_X,\mu_{X'}}(\vec{k},\vec{k}')$, when $X,X'$ are self-adjoint matrices. 

\subsubsection{Analytic transforms of probability distributions}
The Stieltjes transform of a probability distribution $\mu$ is the analytic function $m_{\mu}:\mathbb{C}^+\rightarrow \mathbb{C}$ defined by the formula
$$m_{\mu}(z)=\int_{\mathbb{R}}\frac{1}{t-z}d\mu(t),\, z\in\mathbb{C}^+.$$
We can recover a distribution from its Stieltjes transform through the Stieltjes Inversion formula, which gives $\mu$ in terms of $m_\mu$ as
$$ d\mu(t)=\frac{1}{\pi}\lim_{y\rightarrow 0} \Im m_{\mu}(t+iy)$$
in a weak sense. We will mostly explore spectral distributions through their Stieltjes transforms, since the latter have very good analytical properties. The first important property is that $m_{\mu}(\mathbb{C}^+)\subset\mathbb{C}^+$. Actually, Nevanlinna's theory provides a reciprocal result.
\begin{theorem}\cite[Theorem 3.10]{MiSp}\label{Nevanlina}
Suppose that $m:\mathbb{C}^+\rightarrow \mathbb{C}^+$ is such that $$-iym(iy)\xrightarrow[n\rightarrow\infty]{}1,$$ then there exists a probability measure $\rho$ such that $m=m_{\rho}$.
\end{theorem}

We will use the following transforms of $m_{\mu}$, whose given properties are direct consequences of Nevanlinna's theorem and the expansion at infinity $m_{\mu}(z)=-\sum_{k=0}^r\frac{\mu(r)}{z^{r+1}}+o(z^{-(r+2)})$, when $\mu$ admits moments of order up to $r>0$.
\begin{itemize}
\item the reciprocal Cauchy transform of $\mu$, $F_{\mu}:\mathbb{C}^+\rightarrow\mathbb{C}^+$ with $F_{\mu}(z)=\frac{-1}{m_{\mu}(z)}$. If $\mu$ admits moments of order two, we have the following important formula \cite[Lemma 3.20]{MiSp}, which will be used throughout the paper,
\begin{equation}\label{F_mu_expression}
F_{\mu}(z)=z-\mu(1)+\Var(\mu)m_{\rho}(z),
\end{equation}
for some probability measure $\rho$. In particular,
\begin{equation}\label{F_mu_increasing_imaginary}
\Im[F_{\mu}(z)]\geq \Im z.
\end{equation}
When $\mu$ admits a moment of order three, then $\rho$ has a moment of order one which is given by the formula
\begin{equation}\label{mean_rho}
\rho(1)=\frac{\mu(3)-2\mu(1)\mu(2)+\mu(1)^3}{\Var(\mu)}.
\end{equation}
\item the $h$-transform of $\mu$, $h_{\mu}=F_{\mu}(z)-z$. By \eqref{F_mu_increasing_imaginary}, $h_\mu:\mathbb{C}^+\rightarrow\mathbb{C}^+$ and $h_{\mu}(z)=\Var(\mu)m_{\rho}(z)-\mu(1)$ for $z\in\mathbb{C}^+$.
\end{itemize}
We write $F_X$ and $h_{X}$ instead of $F_{\mu_X}$ and $h_{\mu_X}$ for $X\in\mathcal{M}_{N}(\mathbb{C})$ self-adjoint.
\subsubsection{Probability measures with positive support}\label{Section :Probability measure with positive support} Suppose that $\mu$ has a positive support; up to a rescaling, we can assume that $\mu(1)=1$. Then several new analytic transforms will be useful in the sequel. Note first that 
$$\tilde{m}_{\mu}(z):=1+zm_{\mu}(z)=\int_{\mathbb{R}^+}\frac{t}{t-z}d\mu(t)=\int_{\mathbb{R}^+}\frac{1}{t-z}d\tilde{\mu}(t),$$
with $\tilde{\mu}$ being the probability measure which is absolutely continuous with respect to $\mu$ and has density $d\tilde{\mu}(t)=td\mu(t)$. Moments of $\tilde{\mu}$ are directly related to moments of $\mu$ by the relation $\tilde{\mu}(k)=\mu(k+1)$. In particular,
$$\Var(\tilde{\mu})=\tilde{\mu}(2)-\tilde{\mu}(1)^2=\tilde{\mu}(3)-\tilde{\mu}(2)^2.$$
We denote by $\tilde{F}_{\mu}$ the reciprocal Cauchy transform of $\tilde{\mu}$, and set 
$$\hat{F}_{\mu}:=1+\tilde{F}_{\mu}.$$ 
Remark that $\hat{F}_{\mu}$ is again the reciprocal Cauchy transform of a measure $\hat{\mu}$. Indeed, $\frac{-1}{\hat{F}_{\mu}}$ takes values in $\mathbb{C}^+$ and $\frac{-1}{\hat{F}_{\mu}}\sim \frac{-1}{z}$ as $z$ goes to infinity, so by Theorem \ref{Nevanlina}, there exists a measure $\hat{\mu}$ such that $\frac{-1}{\hat{F}_{\mu}}=m_{\hat{\mu}}$. Moreover, at $t_0<0$, 
$$\tilde{F}_{\mu}(t_0)=\frac{-1}{\tilde{m}_{\mu}(t_0)}=\frac{-1}{\int_{\mathbb{R}^+}\frac{t}{t-t_0}d\mu(t)}< -1,$$
because $\frac{t}{t-t_0}<1$ for $t\geq 0$ and $t_0<0$. Hence, $\hat{F}_{\mu}(t)<0$ for $t<0$ and $\frac{-1}{\hat{F}_{\mu}}$ extends continuously on $\mathbb{R}_{<0}$ with values in $\mathbb{R}$, which by Stieltjes inversion formula implies that $\hat{\mu}(\mathbb{R}_{<0})=0$. The probability distribution $\hat{\mu}$ has thus again a positive support. Actually, $\hat{F}_{\mu}$ is related to $h_\mu$ by the relation
\begin{equation}\label{relation_F_hat}
\frac{z}{\hat{F}_{\mu}(z)}=\frac{z}{1-\frac{1}{1+zm_{\mu}(z)}}=\frac{z}{\frac{zm_{\mu}(z)}{1+z m_{\mu}(z)}}=z-F_{\mu}(z)=-h_{\mu}(z).
\end{equation}

We finally introduce a last transform which is useful in the multiplicative case. When $\mu$ is a probability measure on $\mathbb{R}^+$ with $\mu(1)=1$, we define on $\mathbb{C}^+$ the $\log h$-transform of $\mu$, denoted by $L_{\mu}$, as 
$$L_{\mu}(z)=-\log(-h_{\mu}(z)),$$
where $\log$ is the complex logarithm with branch cut on $\mathbb{R}_{<0}$. Since $h_{\mu}$ takes values in $\mathbb{C}^+$, $L_{\mu}(\mathbb{C}^+)\subset \mathbb{C}^+$. By \eqref{F_mu_expression}, as $z$ goes to infinity, $h_{\mu}(z)= -\mu(1)-\frac{\mu(2)-\mu(1)^2}{z}-\frac{\mu(3)-2\mu(1)\mu(2)+\mu(1)^3}{z^2}+o(z^{-2})$, so that using $\mu(1)=1$ yields as $z$ goes to infinity
\begin{align*}
\log -h_{\mu}=& \frac{\mu(2)-\mu(1)^2}{z}+\frac{\mu(3)-2\mu(1)\mu(2)+\mu(1)^3-(\mu(2)-\mu(1)^2)^2/2}{z^2}+o(z^2)\\
=&\frac{\Var(\mu)}{z}+\frac{\Var(\tilde{\mu})+ \Var(\mu)^2/2}{z^2}+o(z^2).
\end{align*}
Thus, by Theorem \ref{Nevanlina}, there exists a probability measure $\rho_{L}$ with mean $\frac{\Var(\tilde{\mu})+ \Var(\mu)^2/2}{\Var(\mu)}$ such that 
\begin{equation}\label{nevanlinna_LH}
L_{\mu}(z)=\Var(\mu)m_{\rho_{L}}(z).
\end{equation}
\subsection{Free convolution of measures}\label{Subsection:FreeConvolution}
From the seminal work of Voiculescu \cite{Voi1}, it is known that for $N$ large, the spectral distribution of $H=UAU^*+B$ (resp. $M=A^{1/2}UBU^*A^{1/2}$) with $U$ Haar unitary is close in probability to a deterministic measure called the free additive (resp. multiplicative) convolution of $\mu_{A}$ and $\mu_{B}$ and denoted by $\mu_{A}\boxplus \mu_{B}$ (resp. $\mu_{A}\boxtimes\mu_{B}$), see below for a more precise statement. For more background on free convolutions and their relation with random matrices, see \cite{MiSp}. In this manuscript, we will only use the following characterization of the free additive and multiplicative convolutions, called the {\it subordination phenomenon}. This characterization has been fully developed by \cite{Bebe, Bel}, after having been introduced by \cite{Bia} and \cite{Voi2}. For readers not familiar with free probabilistic concepts, the following can be understood as a definition of the free additive and multiplicative convolutions.
\begin{itemize}
\item Suppose that $\mu_{1}\boxplus\mu_2=\mu_3$. Then, for $z\in\mathbb{C}^+$, we have $m_{\mu_{3}}(z)=m_{\mu_2}(\omega_{2}(z))=m_{\mu_1}(\omega_1(z))$, where $\omega_{2}(z)$ is the unique fixed point of the function $K_{z}:\mathbb{C}^+\rightarrow \mathbb{C}^+$ given by
$$K_{z}(w)=h_{\mu_{1}}(h_{\mu_{2}}(w)+z)+z,$$
and $\omega_1$ and $\omega_2$ satisfy the relation
\begin{equation}\label{equation_subordination_addi_free}
\omega_1(z)+\omega_2(z)=z-\frac{1}{m_{\mu_3}(z)}.
\end{equation}
Moreover,  $\omega_1, \omega_2$ are analytic functions on $\mathbb{C}^+$ and we have
$$\omega_{2}(z)=\lim_{n\rightarrow \infty} K_{z}^{\circ n}(w)$$
for all $w\in \mathbb{C}^+$. The functions $\omega_1$ and $\omega_2$ are called the \textit{subordination functions} for the free additive convolution.
\item Suppose that $\mu_1\boxtimes \mu_2=\mu_3$. Then, for $z\in\mathbb{C}^+$, we have $\tilde{m}_{\mu_{3}}(z)=\tilde{m}_{\mu_2}(\omega_{2}(z))=\tilde{m}_{\mu_1}(\omega_1(z))$, where $\omega_{2}(z)$ is the unique fixed point of the function $H_{z}:\mathbb{C}^+\rightarrow \mathbb{C}^+$ given by
$$H_{z}(w)=-\frac{z}{h_{\mu_1}\left(\frac{-z}{h_{\mu_2}(w)}\right)},$$
and $\omega_1(z)$ and $\omega_2(z)$ satisfy the relation
\begin{equation}\label{equation_subordination_multi_free}
\omega_1(z)\omega_2(z)=z\frac{zm_{\mu_3}(z)}{1+zm_{\mu_3}(z)}=z\hat{F}_{\mu_3}(z).
\end{equation}
Moreover,  $\omega_1, \omega_2$ are analytic functions on $\mathbb{C}^+$ and we have
$$\omega_{2}(z)=\lim_{n\rightarrow \infty} H_{z}^{\circ n}(w)$$
for all $w\in \mathbb{C}^+$. The functions $\omega_1$ and $\omega_2$ are called the \textit{subordination functions} for the free multiplicative convolution.
\end{itemize}
These two iterative procedures should be understood as the main implementation scheme for concrete applications, whereas the fixed point equations give the precise definition of both convolutions. The fundamental result relating free probability to random matrices is the convergence of the spectral distribution of sums or products of random matrices conjugated by Haar unitaries towards free additive or multiplicative convolutions.
\begin{theorem}\label{Voiculescu_Speicher}\cite{Voi1,Spe2,PaVa,Vas}
Suppose that $(A_{N},B_{N})_{N\geq 0}$ are two sequences of matrices, with $A_{N},B_{N}\in \mathcal{M}_N(\mathbb{C})$ self-adjoint, and let $U_{N}$ be a random unitary matrix distributed according to the Haar measure. Then, if $\mu_{A_{N}}\xrightarrow[\text{weakly}]{\text{a.s}} \mu_1$ and $\mu_{B_{N}}\xrightarrow[\text{weakly}]{\text{a.s}} \mu_2$ with $\sup_{N}(\max(\mu_{A_N}(2),\mu_{B_N}(2)))<+\infty$, then
$$\mu_{A_N+UB_NU^*}\xrightarrow[\text{weakly}]{\text{a.s}}\mu_1\boxplus \mu_2,$$
and, assuming $A\geq 0$,
$$\mu_{A^{1/2}UB_NU^*A^{1/2}}\xrightarrow[\text{weakly}]{\text{a.s}}\mu_1\boxtimes \mu_2.$$
\end{theorem}
Since those first results, several progresses have been made towards a better comprehension of the above convergences. In particular, concentration inequalities for the convergence of the spectral distribution are given in \cite{BES,Kar2, MeMe} in the additive case, leading to the so-called local laws of the spectral distribution up to an optimal scale (see also \cite{ErKrNe} for concentration inequalities for arbitrary polynomials of matrices). Let us mention also the recent results of \cite{BeGuHu}, which establish a large deviation principle for the convergence of the spectral distribution in the additive case.
\subsection{Matrix subordination}\label{Section:Matrix_subordination}
In \cite{PaVa}, Pastur and Vasilchuk noticed that, since the asymptotic spectral behavior of the addition/multiplication of matrices is close to a free additive/multiplicative convolution, and since the latter are described by subordination functions, there may exist subordination functions directly at the level of random matrices. They actually found such subordination functions and used them to study the convergence of the spectral distribution of the matrix models towards the free convolution. This approach is in particular fundamental to remove any boundedness assumption on the support of $\mu_1$ and $\mu_2$ in Theorem \ref{Voiculescu_Speicher}. In \cite{Kar1,Kar2}, Kargin greatly improved the subordination method of Pastur and Vasilchuk to provide concentration bounds for the additive convolution, when the support of $\mu_A$ and $\mu_B$ remain bounded. 

The goal of Section \ref{Section:bound_subordination} is to improve Kargin's results in the additive case by removing the boundedness assumption on the support and computing explicit bounds, and to provide similar results in the multiplicative case. We review here the matricial subordinations functions in the additive and multiplicative case. Note that in the multiplicative case, we replaced the subordination functions of \cite{Vas} by new subordination functions which are more convenient for our approach. In this paragraph and in the following section, the symbol $\mathbb{E}$ generally refers to the expectation with respect to the Haar unitary $U$.

\subsubsection*{Additive case}
Since $H=UAU^*+B$ with $U$ Haar unitary, we can assume without loss of generalities that $A$ and $B$ are diagonal for any result regarding the spectral distribution of $H$. Hence, the hypothesis of $A$ and $B$ being diagonal will be kept throughout the rest of the manuscript. Set 
$$H'=U^*HU=A+U^*BU,$$
and remark that $m_{H'}=m_{H}$. For $z\in\mathbb{C}^+$, set $f_{A}(z)=\Tr(AG_{H'}(z))$ and $f_{B}(z)=\Tr(BG_{H}(z))$. Then, define 
\begin{equation}\label{matrix_subordination_definition}
\omega_{A}(z)=z-\frac{\mathbb{E}(f_{B}(z))}{\mathbb{E}(m_{H}(z))}, \;\omega_{B}(z)=z-\frac{\mathbb{E}(f_{A}(z))}{\mathbb{E}(m_{H}(z))}.
\end{equation} 
An important point \cite[Eq. 11]{Kar2} is that 
\begin{equation}\label{matrix_subordination_equation}
\omega_A(z)+\omega_B(z)=z-\frac{1}{\mathbb{E}m_{H}(z)},
\end{equation}
which is the same relation as the one satisfied by the subordination functions for the free additive convolution in \eqref{equation_subordination_addi_free}. After a small modification of Kargin's formulation \cite{Kar2}, we get the following approximate subordination relation. 
\begin{lemma}\label{subor_kargin}
For $z\in \mathbb{C}^+$,
\begin{equation}\label{eq:expression_subo_additif}
\mathbb{E} G_{H'}(z)=G_{A}(\omega_{A}(z))+R_{A}(z),
\end{equation}
with $R_{A}(z):=\frac{1}{\mathbb{E}m_{H}(z)}G_{A}(\omega_{A}(z))\mathbb{E}\Delta_{A}(z)$, and
$$\Delta_{A}=(m_{H}-\mathbb{E}m_{H})(U^*BUG_{H'}-\mathbb{E}(U^*BUG_{H'}))-(f_{B}-\mathbb{E}(f_{B}))(G_{H'}-\mathbb{E}(G_{H'})).$$ 
Moreover, $\mathbb{E}\Delta_{A}$ is diagonal and $\Tr\mathbb{E}\Delta_A=0$.
\end{lemma}
Of course, the same result holds for the expression of $\mathbb{E}G_H$ in terms of $G_{B}(\omega_B)$ after switching $A$ and $B$ and $H$ and $H'$.
\begin{proof}
By \cite[Eqs. (12), (13)]{Kar2}, 
$$\mathbb{E} G_{H'}(z)=G_{A}(\omega_{A}(z))+R_{A}(z),$$
with $R_{A}(z):=\frac{1}{\mathbb{E}m_{H}(z)}G_{A}(\omega_{A}(z))(A-z)\mathbb{E}\tilde{\Delta}_{A}(z)$, and
$$\tilde{\Delta}_{A}=-(m_{H}-\mathbb{E}m_{H})G_{H'}-(f_{B}-\mathbb{E}(f_{B}))G_{A}G_{H'}.$$ 
Since $(A-z)$ is deterministic, $(A-z)\mathbb{E}\tilde{\Delta}_A(z)=\mathbb{E}[(A-z)\Delta_A(z)]$, and we have, forgetting the dependence in $z$,
\begin{align*}
(A-z)\mathbb{E}\tilde{\Delta}_{A}=&\mathbb{E}(-(m_{H}-\mathbb{E}m_{H})(A-z)G_{H'}-(f_{B}-\mathbb{E}(f_{B}))G_{H'})\\
=&\mathbb{E}(-(m_{H}-\mathbb{E}m_{H})(1-U^*BUG_{H'})-(f_{B}-\mathbb{E}(f_{B}))G_{H'})\\
=&\mathbb{E}((m_{H}-\mathbb{E}m_{H})U^*BUG_{H'}-(f_{B}-\mathbb{E}(f_{B}))G_{H'})\\
=&\mathbb{E}\left[(m_{H}-\mathbb{E}m_{H})(U^*BUG_{H'}-\mathbb{E}(U^*BUG_{H'}))-(f_{B}-\mathbb{E}(f_{B}))(G_{H'}-\mathbb{E}(G_{H'}))\right]\\
:=&\mathbb{E}\Delta_A.
\end{align*}
where we have used on the penultimate step that $\mathbb{E}(X-\mathbb{E}(X))=0$ for any random variable $X$. This proves the first part of the lemma. For the second part, note that if $V$ is any diagonal unitary matrix, noting that $UV^*$ is again Haar distributed and using that $VAV^{*}=A$ yields that 
\begin{align*}
V\mathbb{E}((m_{H}-\mathbb{E}m_{H})G_{H'})=&V\mathbb{E}((\Tr((A+U^{*}BU-z)^{-1})-\mathbb{E}m_{H})(A+U^{*}BU-z)^{-1})\\
=&V\mathbb{E}((\Tr(V^{*}(VAV^{*}+VU^{*}BUV^{*}-z)^{-1}V)-\mathbb{E}m_{H})\\
&\hspace{5cm}V^*(VAV^{*}+VU^{*}BUV^{*}-z)^{-1})V)\\
=&\mathbb{E}((\Tr((A+VU^{*}BUV^{*}-z)^{-1})-\mathbb{E}m_{H})(A+VU^{*}BUV^{*}-z)^{-1}))V\\
=&\mathbb{E}((\Tr((A+U^{*}BU-z)^{-1})-\mathbb{E}m_{H})(A+UBU^{*}-z)^{-1})V,
\end{align*}
where we used the trace property on the third equality. Likewise,
$$V\mathbb{E}((f_{B}-\mathbb{E}(f_{B}))G_{A}G_{H'})=\mathbb{E}((f_{B}-\mathbb{E}(f_{B}))G_{A}G_{H'})V,$$
and thus $V$ commutes with $\mathbb{E}\tilde{\Delta}_{A}$. Since $\mathbb{E}\tilde{\Delta}_{A}$ commutes with any diagonal unitary matrix, it is also diagonal, and so is $\mathbb{E}\Delta_A=(A-z)\mathbb{E}\tilde{\Delta}_{A}$. Finally,
\begin{align*}
\Tr\mathbb{E}\Delta_{A}=&\mathbb{E}\left[(m_{H}-\mathbb{E}m_{H})\Tr(U^*BUG_{H'}-\mathbb{E}(U^*BUG_{H'}))-(f_{B}-\mathbb{E}(f_{B}))\Tr(G_{H'}-E(G_{H'}))\right]\\
=&\mathbb{E}((m_{H}-\mathbb{E}m_{H})(f_{B}-\mathbb{E}(f_{B}))-(f_{B}-\mathbb{E}(f_{B}))(m_{H}-\mathbb{E}m_{H}))=0.
\end{align*}
\end{proof}
Moreover, an algebraic manipulation of \eqref{eq:expression_subo_additif} yields 
\begin{equation}\label{omega_value}
\omega_{A}=A-(\mathbb{E}G_{H'})^{-1}+(-\mathbb{E}G_{H'})^{-1}\frac{1}{\mathbb{E}m_{H}}\mathbb{E}_{U}\Delta_{A},
\end{equation}
Following \cite[Lemma 2.1]{Kar2} (see also Lemma \ref{bound_expectation_matrix}), remark that we also have
\begin{equation}\label{spectrum_H}
 -(\mathbb{E}_{U}G_{H'})^{-1}+A-z\in \mathbb{H}(\mathcal{M}_N(\mathbb{C})),
 \end{equation}
where $\mathbb{H}(\mathcal{M}_{n}(\mathbb{C}))$ denotes the half-space $\{M\in \mathcal{M}_N(\mathbb{C}), \frac{1}{i}(M-M^*)\geq 0\}$.
\subsubsection*{Multiplicative case} 
This section adapts Kargin's approach to the multiplicative case. Matricial subordination functions already appeared in the multiplicative case in \cite{Vas}, but we chose to create new matricial subordination functions which are closer to the ones encoding the free multiplicative convolution in Section \ref{Section:free_probabilistic_background}. 

Recall here that $M=A^{1/2}UBU^*A^{1/2}$ with $A,B\geq 0$ non-zero, $m_{M}(z)=\Tr((M-z)^{-1})$ and $\tilde{m}_{M}(z)=\Tr(M(M-z)^{-1})=1+zm_{M}(z)$. Like in the additive case, we define $f_{A}(z)=\Tr(A(M-z)^{-1})$ and introduce for $z\in\mathbb{C}^+$ the subordination functions 
\begin{equation}\label{definition_subor_multi}
\omega_{A}=\frac{z\mathbb{E}f_{A}(z)}{\mathbb{E}\tilde{m}_{M}(z)},\;  \omega_{B}=\frac{z\mathbb{E}m_{M}(z)}{\mathbb{E}f_{A}(z)} .
\end{equation}
Remark that there is an asymmetry between $\omega_A$ and $\omega_B$, which reflects the different roles played by $A$ and $B$ in $M$. This symmetry can be restored by studying $AUBU^{*}$ instead of $A^{1/2}UBU^{*}A^{1/2}$ at the cost of loosing self-adjointness. The two subordination functions however still satisfy the symmetric relation 
\begin{equation}\label{matrix_subordination_equation_multi}
\omega_A(z)\omega_B(z)=z\frac{z\mathbb{E}m_{M}(z)}{1+z\mathbb{E}m_{M}(z)},
\end{equation}
which is similar to \eqref{equation_subordination_multi_free}.
\begin{lemma}\label{Lem:subor_multi}
For $z\in\mathbb{C}^+$,
\begin{equation}\label{subor_karg_multiplication}
\mathbb{E}(MG_{M}(z))=AG_{A}(\omega_{A}(z))+R_{A}(z),
\end{equation}
with $R_{A}(z)=\omega_{A}(z)G_{A}(\omega_{A}(z))\mathbb{E}\Delta_{A}(z)$, where 
$$\Delta_{A}(z)=\frac{z}{\mathbb{E}(f_{A}(z))}\big((f_{A}(z)-\mathbb{E}(f_{A}(z)))G_{M}-(m_{M}(z)-\mathbb{E}(m_{M}(z)))AG_{M}\big).$$
Similarly, setting $M'=B^{1/2}U^*AUB^{1/2}$,
\begin{equation}\label{subor_karg_multiplication_B}
\mathbb{E}(M'G_{M'})=BG_{B}(\omega_{B}(z))+R_{B}(z),
\end{equation}
with $R_{B}(z)=BG_{B}(\omega_{B})\mathbb{E}\Delta_{B}$, where
$$\Delta_{B}(z)=\frac{z}{\mathbb{E}f_{A}(z)}\left(-(f_{A}(z)-\mathbb{E}f_{A}(z))G_{M'}+(m_{M}(z)-\mathbb{E}m_{M}(z))U^{*}A^{1/2}G_{M}A^{1/2}U\right).$$
Moreover, $\mathbb{E}\Delta_A$ and $\mathbb{E}\Delta_B$ are diagonal and $\mathbb{E}\Tr\Delta_A=\mathbb{E}\Tr\Delta_B=0$.
\end{lemma} 
\begin{proof}
This lemma is deduced from \cite[Eqs. (12), (13)]{Kar2} recalled in the proof of Lemma \ref{subor_kargin}. Remark that these results were only stated in \cite{Kar2} for $A$ and $B$ self-adjoint, but they can painlessly be extended to the case of $A$ and $B$ normal matrices with spectrum in $\mathbb{R}\cup\mathbb{C}^-$ and for $z$ satisfying $\Im z>\sup_{1\leq i\leq N} \Im \lambda_i^A$ or $\Im z>\sup_{1\leq i\leq N} \Im \lambda_i^B$ (so that all quantities are still well defined). Suppose first that $A$ is invertible. Then, we have 
\begin{align}
G_{M}(z)=(A^{1/2}UBU^*A^{1/2}-z)^{-1}=&(A^{1/2}(-zA^{-1}+UBU^*)A^{1/2})^{-1}\nonumber\\
=&A^{-1/2}(UBU^*-zA^{-1})^{-1}A^{-1/2}.\label{invert_trick}
\end{align}
Set $\tilde{A}=-zA^{-1}$. The matrices $\tilde{A}$ and $B$ are diagonal with spectrum having non-positive imaginary part. Applying \cite[Eqs. (12), (13)]{Kar2} to $\tilde{A},B$ and $\tilde{H}=\tilde{A}+UBU^*$ for $w\in\mathbb{C}$ with $\Im w>\sup_{1\leq i\leq N} \Im \lambda_i^{\tilde{A}}$ yields
\begin{equation}\label{subor_mtilde}
\mathbb{E}G_{\tilde{H}}(w)=G_{\tilde{A}}(\omega_{\tilde{A}}(w))+R_{\tilde{A}}(w),
\end{equation}
where $\omega_{\tilde{A}}$ and $R_{\tilde{A}}$ are respectively given by 
$$\omega_{\tilde{A}}(w)=w-\frac{\mathbb{E}(f_{B}(w))}{\mathbb{E}(m_{\tilde{H}}(w))}$$
with 
$$f_{B}(w)=\Tr(UBU^*G_{\tilde{H}}(z)), $$
and 
$$R_{\tilde{A}}(w)=\frac{1}{\mathbb{E}(m_{\tilde{H}}(w))}G_{\tilde{A}}(\omega_{\tilde{A}}(w))(\tilde{A}-w)\mathbb{E}_{U}\Delta_{\tilde{A}}(w),$$
where 
$$\Delta_{\tilde{A}}=-(m_{\tilde{H}}-\mathbb{E}(m_{\tilde{H}}))G_{\tilde{H}}-(f_{B}-\mathbb{E}(f_{B}))G_{\tilde{A}}G_{\tilde{H}}.$$
Since $z\in\mathbb{C}^+$ and $A>0$, $\sup_{1\leq i\leq N} \Im \lambda_i^{\tilde{A}}<0$ so that we can apply the above subordination relations for $w=0$. First,
\begin{align*}
f_{B}(0)=&1+\Tr(zA^{-1}(UBU^*-zA^{-1})^{-1})\\
=&1+z\Tr(A^{-1/2}(UBU^*-zA^{-1})^{-1}A^{-1/2})\\
=&1+zm_{M}(z)=\tilde{m}_{M}(z),
\end{align*}
where we used \eqref{invert_trick} in the last equality. Similarly,
$$m_{\tilde{H}}(0)=\Tr((UBU^*-zA^{-1})^{-1})=\Tr(A^{1/2}G_{M}(z)A^{1/2})=f_{A}(z).$$
Hence,
$$\omega_{\tilde{A}}(0)=0-\frac{\mathbb{E}(f_{B}(0))}{\mathbb{E}(m_{\tilde{H}}(0))}=-\frac{\mathbb{E}(\tilde{m}_{M}(z))}{\mathbb{E}(\Tr(A^{1/2}G_{M}(z)A^{1/2})}=-z\omega_{A}(z)^{-1},$$
and, using again \eqref{invert_trick} and the latter computations,
\begin{align*}
\Delta_{\tilde{A}}(0)=&-(f_{A}(z)-\mathbb{E}(f_{A}(z)))(UBU^*-zA^{-1})^{-1}\\
&\hspace{4cm}-(zm_{M}(z)-\mathbb{E}(zm_{M}(z)))(-z^{-1}A)(UBU^*-zA^{-1})^{-1}\\
=&-(f_{A}(z)-\mathbb{E}(f_{A}(z)))A^{1/2}G_M(z)A^{1/2}\\
&\hspace{4cm}+z^{-1}(zm_{M}(z)-\mathbb{E}(zm_{M}(z)))AA^{1/2}G_{M}(z)A^{1/2}\\
=&A^{1/2}\left((m_{M}(z)-\mathbb{E}(m_{M}(z)))AG_{M}(z)-(f_{A}(z)-\mathbb{E}(f_{A}(z)))G_M(z)\right)A^{1/2}.
\end{align*}
Taking the expectation on $\Delta_{\tilde{A}}(0)$ yields then
\begin{align*}
R_{\tilde{A}}(0)=&\frac{1}{\mathbb{E}(f_{A}(z))}(-zA^{-1}+z\omega_{A}(z)^{-1})^{-1}(-zA^{-1})\\
&\hspace{2cm}A^{1/2}\mathbb{E}\left((m_{M}(z)-\mathbb{E}(m_{M}(z)))AG_{M}(z)-(f_{A}(z)-\mathbb{E}(f_{A}(z)))G_{M}(z))\right)A^{1/2}\\
=& \frac{1}{\mathbb{E}(f_{A}(z))}(A^{-1}-\omega_{A}(z)^{-1})^{-1}A^{-1/2}\\
&\hspace{2cm}\mathbb{E}\left((m_{M}(z)-\mathbb{E}(m_{M}(z)))AG_{M}(z)-(f_{A}(z)-\mathbb{E}(f_{A}(z)))G_{M}(z)\right)A^{1/2}.\\
\end{align*}
Putting the latter expression in \eqref{subor_mtilde} and using \eqref{invert_trick} gives then
\begin{align*}
\mathbb{E}(G_{M}(z))=&A^{-1/2}G_{\tilde{H}}(0)A^{-1/2}\\
=&A^{-1/2}(-zA^{-1}+z\omega_{A}(z)^{-1})^{-1}A^{-1/2}+A^{-1/2}R_{A}(z)A^{-1/2}\\
=&z^{-1}\omega_{A}(z)(A-\omega_{A}(z))^{-1} +\frac{\omega_{A}(z)}{\mathbb{E}(f_{A}(z))}(A-\omega_{A}(z))^{-1}\mathbb{E}\big((f_{A}(z)-\mathbb{E}(f_{A}(z)))G_{M}(z)\\
&\hspace{4cm}-(m_{M}(z)-\mathbb{E}(m_{M}(z)))AG_{M}(z)\big).
\end{align*}
Hence, we get
\begin{equation}
z\mathbb{E}(G_{M}(z))=\omega_{A}(z)G_{A}(\omega_{A}(z))+R_A(z),
\end{equation}
with $R_A(z)=\omega_{A}(z)G_{A}(\omega_{A}(z))\mathbb{E}\Delta_A(z)$, and 
$$\Delta_{A}(z)=\frac{z}{\mathbb{E}(f_{A}(z))}\mathbb{E}\big((f_{A}(z)-\mathbb{E}(f_{A}(z)))G_{M}(z)-(m_{M}(z)-\mathbb{E}(m_{M}(z)))AG_{M}(z)\big).$$
Finally, we have 
$$\mathbb{E}(MG_{M}(z))=1+z\mathbb{E}(G_{M}(z))=1+\omega_{A}(z)G_{A}(\omega
_{A}(z))+R_{A}(z)=AG_{A}(\omega_{A}(z))+R_{A}(z).$$
Let us do the same computation for the subordination involving $\omega_{B}$. Using the subordination on $B$ for $B-zU^*A^{-1}U=B+U^*\tilde{A}U$ at $w=0$ together with \eqref{invert_trick} yields 
$$\mathbb{E}(U^*A^{1/2}G_{M}(z)A^{1/2}U)=\mathbb{E}((B-U^*zA^{-1}U)^{-1})=(B-\omega_{B}(z))^{-1}+\tilde{R}_{B}(z),$$
with $\omega_{B}(z)=\frac{-\mathbb{E}f_{\tilde{A}}(0)}{\mathbb{E}m_{\tilde{H}}(0)}=\frac{z\mathbb{E}m_{M}(z)}{\mathbb{E}f_{A}(z)}$ and $\tilde{R}_{B}(z)=G_{B}(\omega_{B}(z))B\mathbb{E}\tilde{\Delta}_{B}(z)$ with 
$$\tilde{\Delta}_{B}=\frac{1}{\mathbb{E}f_{A}}\left(-(f_{A}-\mathbb{E}f_{A})U^*A^{1/2}G_{M}A^{1/2}U+z(m_{M}-\mathbb{E}m_{M})B^{-1}U^*A^{1/2}G_{M}A^{1/2}U\right).$$
Since $B^{1/2}U^*A^{1/2}G_{M}A^{1/2}UB^{1/2}=B^{1/2}U^*AUB^{1/2}G_{M'}=1+zG_{M'}$, where we recall that $M'=B^{1/2}U^*AUB^{1/2}$. Hence,
$$B^{1/2}\mathbb{E}\tilde{\Delta}_{B}B^{1/2}=\frac{z}{\mathbb{E}f_{A}}\mathbb{E}\left(-(f_{A}-\mathbb{E}f_{A})G_{M'}+(m_{M}-\mathbb{E}m_{M})U^*A^{1/2}G_{M}A^{1/2}U\right),$$
where we used that $\mathbb{E}(f_A-\mathbb{E}f_{A})=0$. Hence, 
$$\mathbb{E}(M'G_{M'}(z))=B^{1/2}\mathbb{E}(U^*A^{1/2}G_{M}A^{1/2}U)B^{1/2}=BG_{B}(\omega_{B}(z))+R_{B}(z),$$
with $R_{B}(z)=BG_{B}(\omega_{B}(z))\mathbb{E}\Delta_{B}(z)$, where
$$\Delta_{B}=\frac{z}{\mathbb{E}f_{A}}\left(-(f_{A}-\mathbb{E}f_{A})G_{M'}+(m_{M}-\mathbb{E}m_{M})U^*A^{1/2}G_{M}A^{1/2}U\right).$$
The proof that $\Tr(\Delta_A)=\Tr(\Delta_B)=0$ and that $\Delta_A,\Delta_B$ are diagonal is then the same as in Lemma \ref{subor_kargin}. 

In order to end the proof, it remains to deal with the case where $A$ is non invertible. Let $z\in\mathbb{C}^+$ be fixed. Then,
$$0=z\mathbb{E}(G_{M}(z))-w_{A}G_{A}(w_{A})-R_{A}(z):=\Phi(A),$$
where $\Phi(A)$ is a map from $\mathcal{H}_N^{++}$ to $\mathcal{H}_{N}^{++}$, where $\mathcal{H}_{N}^{++}$ denotes the $N^2$-dimensional open manifold of positive definite Hermitian matrices of dimension $N$. By \cite[Proposition 3.1]{Vas}, $A\mapsto G_{A}(z)$ is Lipschitz with Lipschitz constant $\frac{1}{\Im(z)^2}$. Hence, $\Phi(A)$ is a rational expression of continuous functions of $A$, each of them being defined and continuous on the closed manifold $\mathcal{H}_{N}^+$ of non-negative Hermitian matrices of dimension $N$. In order to extend the relation $\Phi(A)=0$ to $\mathcal{H}^+_N\setminus\{0\}$, it suffices therefore to prove that $\Phi$ can be extended by continuity to $\mathcal{H}^+_N\setminus\{0\}$, meaning that no denominator in $\Phi$ vanishes when $A\in \mathcal{H}^+_N$ is non zero. When checking each term in $\Phi$, the only non trivial ones are $1+z\mathbb{E}m_{M}(z)$ and $\mathbb{E}f_{A}(z)$. First, expanding $\mathbb{E}m_{M}$ at infinity yields
 $$\mathbb{E}m_{M}(z)=-\frac{1}{z}-\frac{\mathbb{E}(\Tr(A^{1/2}UBU^*A^{1/2}))}{z^2}-\frac{\Tr(A^{1/2}UBU^*AUBU^*A^{1/2})}{z^3}+ o(z^{-3}).$$
Moreover, set $v\in\mathbb{C}^n$ be such that $A^{1/2}v:=w\not=0$. Then, $U^*w$ is uniformly distributed on the sphere of radius $\vert w\vert$, and thus $\langle BU^*w,U^*w\rangle$ is almost surely non-zero (provided $B$ is non-zero). Hence, $A^{1/2}UBU^*A^{1/2}$ is almost surely non-zero, which implies  that $\Tr(A^{1/2}UBU^*AUBU^*A^{1/2})$ is a random variable almost-surely positive. Hence, $\mathbb{E}(\Tr(A^{1/2}UBU^*AUBU^*A^{1/2}))>0$, and thus $m_{M}$ is almost-surely not equal to the function $z\mapsto -z^{-1}$. Therefore, for any fixed $z\in\mathbb{C}^+$, $\Im (1+zm_{M}(z))$ is almost surely positive and after averaging $1+z\mathbb{E}m_{M}(z)$ does not vanish. The function $f_{A}$ is analytic from $\mathbb{C}^+$ to $\mathbb{C}^+$, and $f_{A}=-\frac{\Tr(A)}{z}+o(z)$ at infinity, thus by Theorem \ref{Nevanlina} there exists a positive measure $\rho$ on $\mathbb{R}$ of mass $\Tr(A)$ such that 
$$f_{A}(z)=\int_{\mathbb{R}}\frac{1}{t-z}d\rho(z).$$
Therefore, $\Im(f_{M}(z))>0$ almost surely for $z\in\mathbb{C}^+$, which implies that $\mathbb{E}(f_{M}(z))$ never vanishes.
\end{proof}
Remark that rearranging terms in \eqref{subor_karg_multiplication} yields
\begin{equation}\label{expression_subor_other}
\omega_A A=A^2-(A+\omega_A\mathbb{E}\Delta_{A})(A\mathbb{E}[MG_{M}]^{-1}),
\end{equation}
where $A\mathbb{E}[MG_{M}]^{-1}=\mathbb{E}[UB^{1/2}G_{M'}B^{1/2}U^{*}]^{-1}$ is always defined (see Lemma \ref{lower_bounds_multi}). Likewise, rearranging terms in \eqref{subor_karg_multiplication_B}  yields
\begin{equation}\label{expression_subor_other_B}
\omega_{B}(z)=B-B(\mathbb{E}M'G_{M'})^{-1}+\mathbb{E}\Delta_{B}B(\mathbb{E}M'G_{M'})^{-1},
\end{equation}
where $B\mathbb{E}[M'G_{M'}]^{-1}=\mathbb{E}[U^{*}A^{1/2}G_{M}A^{1/2}U]^{-1}$ is always defined.
\section{Bounds on the subordination method}\label{Section:bound_subordination}
We have seen in the previous section that matricial subordination functions already satisfy similar relations as the one fulfilled by the subordinations functions for the free convolutions. In this section we quantify this similarity by estimating the error terms in \eqref{eq:expression_subo_additif} and \eqref{matrix_subordination_equation_multi}. Namely we show that $\mathbb{E}m_{H}$ (resp. $\mathbb{E}\tilde{m}_{M}$) and $m_{A}(\omega_A)$ or $m_{B}(\omega_B)$ (resp. $\tilde{m}_{A}$ or $\tilde{m}_{B}$) are approximately the same in the additive (resp. multiplicative) case.  In the additive case, this has been already done in \cite{Kar2}; hence the goal of the study of the additive case is just to give precise estimates in the approach of Kargin, without any assumption on the norm of $A$ and $B$. Up to our knowledge, the multiplicative case has not been done with the subordination approach of Kargin (see however \cite{ErKrNe} for similar result for general polynomials in Wigner matrices).

\subsection{Subordination in the additive case}
The goal of this subsection is to prove the following convergence result. Recall notations from Section \ref{Section:Matrix_subordination}, and recall also the notations from \ref{Section:notation}. In particular, we write $a_i,\, b_i$ for $\Tr(A^i), \,\Tr(B^i)$ for $i\geq 1$.
\begin{proposition}\label{convergence_subord_B}
For $z\in\mathbb{C}^+$ with $\Im z=\eta=\kappa\sigma_1$ and for
$$N\geq \sqrt{\frac{\max(C_{thres,A}(\eta),C_{thres,B}(\eta))}{\eta^3}},$$
with $C_{thres,A}(\eta),\,C_{thres,B}(\eta)$ given in Proposition \ref{minoration_omega}, then $\Im \omega_A\geq 2\eta/3, \Im \omega_B\geq 2\eta/3$ and 
$$\vert \mathbb{E}m_{H}(z)-m_{A}(\omega_{A}(z))\vert\leq \frac{C_{bound,A}(\kappa)}{\vert z\vert N^2},$$
and 
$$\vert \mathbb{E}m_{H}(z)-m_{B}(\omega_{B}(z))\vert\leq \frac{C_{bound,B}(\kappa)}{\vert z\vert N^2},$$
with
\begin{align*}
\bullet& C_{bound,A}(\kappa)=\\
&\frac{12\sqrt{6}\sigma_B^2\sigma_A}{\kappa^3\sigma_1^3}\left(1+\frac{\sigma_A^2+\sigma_B^2}{\kappa^2\sigma_1^2}\right)\sqrt{1+\frac{\sigma_A^2+\theta_B\sigma_B^2}{\kappa^2\sigma_1^2}}\sqrt{1+\frac{m_{A^2\ast B^2}(1^2,1^2)^{1/2}a_4^{1/2}+b_6^{2/3}a_6^{1/3}}{a_2b_2\kappa^2\sigma_1^2}},
\end{align*}
and $C_{bound,B}$ is obtained from $C_{bound,A}$ by switching $A$ and $B$.
\end{proposition}
We postpone the proof of Proposition \ref{convergence_subord_B} to the end of the section, proving first some intermediary steps. First, remark that $m_{H}=m_{H'}$, where $H'=A+U^*BU$. Hence, we can apply \eqref{eq:expression_subo_additif} to either $H$ or $H'$ (switching $A$ and $B$) to deduce informations on $m_{H}$. Then, by Lemma \ref{trace_Weingarten} and the hypothesis $\Tr(A)=\Tr(B)=0$, we have 
$$\mathbb{E}\Tr((A+U^*BU)^2)=\Tr(A^2)+\Tr(B^2)+2\Tr(A)\Tr(B)=\Tr(A^2)+\Tr(B^2)=a_2+b_2,$$
where we used notations from Section \ref{Section:notation}. Hence, by \eqref{F_mu_expression} and the fact that $\mathbb{E}\Tr(A+U^*BU)=0$,
\begin{equation}\label{lower_bound_expectation}
\left\vert\mathbb{E}(m_{H}(z))^{-1}+z\right\vert\leq\frac{\Tr((A+U^*BU)^2)}{\Im(z)} \leq \frac{a_2+b_2}{\Im(z)}
\end{equation}
for all $z\in\mathbb{C}^+$. We can obtain a similar bound for $(\mathbb{E}(G_{H'}))^{-1}$, as next lemma shows.
\begin{lemma}\label{bound_expectation_matrix}
The matrix $\mathbb{E}(G_{H'})^{-1}$ is diagonal with diagonal entries satisfying the bound
$$\left\vert [\mathbb{E}(G_{H'})^{-1}]_{ii}-\lambda_{i}^{A}+z\right\vert\leq \frac{b_2}{\eta}.$$
\end{lemma}
\begin{proof}
We know by Lemma \ref{subor_kargin} that $\mathbb{E}(G_{H'})$ is diagonal. Define the map $I:\mathbb{C}^+\mapsto \mathbb{C}$ by $I(z)=-[\mathbb{E}(G_{H'})^{-1}]_{ii}=-[\mathbb{E}(G_{H'})_{ii}]^{-1}$.
By \eqref{spectrum_H}, $I$ maps $\mathbb{C}^+$ to $\mathbb{C}^+$. Moreover, as $z$ goes to infinity, $\mathbb{E}G_{H'}(z)=-z^{-1}-\mathbb{E}(A+U^*BU)z^{-2}-\mathbb{E}(A+U^*BU)^2 z^{-3}+o(z^{-3})$. By Lemma \ref{Weingarten_calculus}, $\mathbb{E}(U^*BU)=\Tr(B)=0$ and 
$$\mathbb{E}((A+U^*BU)^2)=A^2+\mathbb{E}(U^*BU)A+A\mathbb{E}(U^*BU)+\mathbb{E}(UB^2U^*)=A^2+b_2.$$
Hence,
$$\mathbb{E}(G_{H'})_{ii}=-z^{-1}-\lambda_{i}^Az^{-2}-((\lambda_{i}^A)^2+b_2)z^{-3}+o(z^{-3}).$$
Applying Theorem \ref{Nevanlina} to the map $I$ and then using \eqref{F_mu_expression} yield the existence of a probability measure $\rho$ on $\mathbb{R}$ such that 
$$(-\mathbb{E}(G_{H'})_{ii})^{-1}=z-\lambda_{i}^A+b_2m_\rho(t).$$
In particular, 
$$\left\vert [\mathbb{E}(G_{H'})^{-1}]_{ii}+z-\lambda_{i}^A\right\vert\leq \frac{b_2}{\eta}.$$
\end{proof}
We now provide a bound on $T\Delta_{A}$ for $T\in \mathcal{M}_N(\mathbb{C})$, where $\Delta_A$ is given in \eqref{eq:expression_subo_additif}. In the following lemma, the dependence in $z$ of $\Delta_A$ is omitted. 
\begin{lemma}\label{bound_Delta_additif}
For $z\in\mathbb{C}^+$ with $\Im z=\eta$
and for $T\in \mathcal{M}_N(\mathbb{C})$, 
$$\mathbb{E}\vert \Tr(T\Delta_{A})\vert\leq \frac{4b_4^{1/4}}{\eta^4 N^2}\Big[\sqrt{b_2}\Tr(\vert TA\vert^4)^{1/4}+b_4^{1/4}a_4^{1/4}\Tr(\vert T\vert^4)^{1/4}\Big],$$
and
$$ \mathbb{E}\vert\Tr(T\Delta_{A})\vert\leq \frac{8b_4^{1/4}\sqrt{b_2}a_4^{1/4}\Vert T\Vert_{\infty}}{\eta^4 N^2}.$$
\end{lemma}
\begin{proof}
Using the definition of $\Delta_{A}$ in \eqref{eq:expression_subo_additif}, we get 
$$\Tr(T\Delta_{A})=(m_{H}-\mathbb{E}m_{H})\Tr(TU^*BUG_{H'}-\mathbb{E}(TU^*BUG_{H'}))-(f_{B}-\mathbb{E}(f_{B}))\Tr(TG_{H'}-\mathbb{E}(TG_{H'})).$$
Since $U^*BUG_{H'}=1-(A-z)G_{H'}$ and $\Tr(T)-\mathbb{E}\Tr(T)=0$, we deduce
\begin{align*}
\Tr(T\Delta_{A})=&-(m_{H}-\mathbb{E}m_{H})\Tr(T(A-z)G_{H'}-\mathbb{E}(T(A-z)G_{H'}))\\
&\hspace{5cm}-(f_{B}-\mathbb{E}(f_{B}))\Tr(TG_{H'}-\mathbb{E}(TG_{H'}))\\
=&-(m_{H}-\mathbb{E}m_{H})(f'_{TA}-\mathbb{E}f'_{TA})+z(m_{H}-\mathbb{E}m_{H})(f'_{T}-\mathbb{E}f'_{T})-(f_{B}-\mathbb{E}f_{B})(f'_{T}-\mathbb{E}f'_{T}),
\end{align*}
with $f'_{X}=\Tr(XG_{H'})$ for $X\in\mathcal{M}_{N}(\mathbb{C})$.
Using the fact that 
$zm_{H}=tr(UAU^*G_{H})+\Tr(BG_{H})-1=f'_{A}+f_{B}-1$ yields finally
\begin{equation}\label{eq:expression_TDelta}
\Tr(T\Delta_{A})=-(m_{H}-\mathbb{E}m_{H})(f'_{TA}-\mathbb{E}f'_{TA})+(f'_{A}-\mathbb{E}f'_{A})(f'_{T}-\mathbb{E}f'_{T}).
\end{equation}
Then, on the first hand, Cauchy-Schwartz inequality and Lemma \ref{concentration_result_additive} with $A$ and $B$ switched give
\begin{align*}
\mathbb{E}\vert \Tr(T\Delta_{A})\vert \leq& \sqrt{\Var f'_{TA} \Var m_H}+\sqrt{\Var f'_{A}\Var f'_{T} }\\
\leq&\frac{4}{\eta^4 N^2}\Big[\sqrt{\Tr(B^2)}(\Tr(B^4)\Tr(\vert TA\vert^4))^{1/4}+\sqrt{\Tr(B^4)}(\Tr(T^4)\Tr(A^4))^{1/4}\Big],
\end{align*}
where in Lemma \ref{concentration_result_additive} we chose $\alpha=\beta=\frac{1}{4}$ for $f'_{TA},\,f'_{A},\,f'_{T}$ and $\alpha=2,\,\beta=\infty$ for $m_{H}=f_{\Id}$. On the second hand, choosing instead $\alpha=\beta=\frac{1}{4}$ for $f'_{TA},\,f'_{A}$ and $\alpha=2,\,\beta=\infty$ for $m_{H}=f_{\Id}\,f'_{T}$ in Lemma \ref{concentration_result_additive} gives
\begin{align*}
\mathbb{E}\vert \Tr(T\Delta_{A})\vert \leq& \sqrt{\Var f'_{TA} \Var m_H}+\sqrt{\Var f'_{A}\Var f'_{T} }\\
\leq&\frac{4}{\eta^4 N^2}\Big[\sqrt{\Tr(B^2)}(\Tr(B^4)\Tr(\vert TA\vert^4))^{1/4}+\sqrt{\Tr(B^2)}\Tr(B^4)^{1/4}\Vert T\Vert_{\infty}\Tr(A^4)^{1/4}\Big]\\
\leq&\frac{8\Tr(B^4)^{1/4}\sqrt{\Tr(B^2)}\Tr(A^4)^{1/4}\Vert T\Vert_{\infty}}{\eta^4 N^2}.
\end{align*}
\end{proof} 
We deduce the following bound on the subordination functions $\omega_{A}$.
\begin{proposition}\label{minoration_omega}
Let $z\in\mathbb{C}$ with $\Im(z):=\eta$. Then, 
$$\vert \omega_{A}-z\vert \leq \frac{\sigma_B^2}{\eta}+\frac{C_{thres,A}}{3N^2}\eta,$$
and
$$\Im \omega_A\geq\eta-\frac{C_{thres,A}}{3N^2}\eta,$$
with
\begin{align*}
&C_{thres,A}(\eta)=\\
&\frac{12\sigma_B^2\sigma_A}{\eta^3}\left(1+\frac{\sigma_A^2+\sigma_B^2}{\eta^2}\right)\Bigg(\sqrt{2\left(1+\frac{\sigma_A^2+\sigma_B^2\theta_B}{\eta^2}\right)\cdot\left(1+\sqrt{\theta_A\theta_B}+\frac{2\sqrt{m_{A^2\ast B^2}(1^2,1^2)\theta_A}}{\sigma_B^2\eta^2}\right)}\\
&\hspace{3.5cm}+\sqrt{3\frac{\sqrt{\theta_B\theta_A}\sigma_A^2}{\eta^2}\left(1+\frac{m_{A^2\ast B^2}(1^2,1^2)^{1/2}a_4^{1/2}+b_6^{2/3}a_6^{1/3}}{\sigma_A^2\sigma_B^2\eta^2}\right)}+2\frac{\theta_B^{1/4}\sigma_B^3\theta_A^{1/4}}{\eta^3}\Bigg).
\end{align*}
\end{proposition}

\begin{proof}
We modify the original proof of Kargin to get the most explicit bound as possible. From \eqref{omega_value}, we get 
\begin{align*}
\omega_{A}=A-(\mathbb{E}G_{H'})^{-1}+(-\mathbb{E}G_{H'})^{-1}\frac{1}{\mathbb{E}m_{H}}\mathbb{E}_{U}\Delta_{A}\\
=A+z-A+\epsilon_{1}+\frac{1}{\mathbb{E}m_{H}}(z-A+\epsilon_{1})\mathbb{E}_{U}\Delta_{A},
\end{align*}
with $\epsilon_1\in\mathbb{H}(\mathcal{M}_{N}(\mathbb{C}))$ by \eqref{spectrum_H} and $\epsilon_1$ is diagonal with $\vert (\epsilon_{1})_{ii}\vert\leq \frac{b_2}{\Im(z)}$ by Lemma \ref{bound_expectation_matrix}. Hence, taking the trace yields 
\begin{equation}\label{expression_omega_A_z}
\omega_A=z+\Tr(\epsilon_1)+\delta,
\end{equation}
with $\delta=\Tr[(z-A+\epsilon_{1})\frac{1}{\mathbb{E}m_{H}}\mathbb{E}_{U}\Delta_{A}]$ and $\Tr(\epsilon_1)\in\mathbb{C}^{+}$. By \eqref{lower_bound_expectation}, $\frac{1}{\mathbb{E}(m_{H})}=-z+\epsilon_{2}$ with $\vert \epsilon_{2}\vert \leq \frac{a_2+b_2}{\Im(z)}$. Therefore, using $\Tr(\mathbb{E}\Delta_{A})=0$ from Lemma \ref{subor_kargin},
\begin{align*}
\delta=&\Tr\left((A-z+\epsilon_{1})(-z+\epsilon_{2})\mathbb{E}_{U}(\Delta_{A})\right)\\
=&\Tr\left((-z+\epsilon_{2})(A+\epsilon_{1})\mathbb{E}_{U}(\Delta_{A})-z(-z+\epsilon_{2})\mathbb{E}_{U}(\Delta_{A})\right)\\
=&(-1+\epsilon_{2}/z)\mathbb{E}_{U}\left[\Tr\left((A+\epsilon_{1})(z\Delta_{A})\right)\right].
\end{align*}
First, by \eqref{eq:expression_subo_additif} we have
\begin{align*}
z\Tr\left(A\Delta_{A}\right)=&(zm_{H}-z\mathbb{E}m_{H})\left[\Tr(AUBU^*G_{H'})-\mathbb{E}\Tr(AUBU^*G_{H'})\right]\\
&\hspace{4cm}-\left[(f_{B}-\mathbb{E}f_B)(\Tr(zAG_{H'})-\mathbb{E}\Tr(zAG_{H'})\right].
\end{align*}
Hence, by Cauchy-Schwartz inequality,
\begin{align*}
&\mathbb{E}\vert z\Tr\left(A\Delta_{A}\right)\vert
\leq \sqrt{\Var(\tilde{f}_{A})\Var(zm_{H})}+\sqrt{\Var(f_{B})\Var(zf'_{A})},
\end{align*}
with $\tilde{f}_{A}=\Tr(AU^*BUG_{H'}),\, f'_{A}=\Tr(AG_{H'})$. Then, using Lemma \ref{concentration_result_additive_z} with $A$ and $B$ switched gives 
$$\Var(zm_{H})\leq \frac{8}{N^2\eta^2}\left(b_2+\frac{a_2b_2+b_4}{\eta^2}\right),$$
and using the same lemma with $\alpha_{1},\beta_1=4$ and $\alpha_2=3,\beta_2=6,$
\begin{align*}
\Var(zf'_{A})\leq &\frac{12}{N^2\eta^2}\Big(a_2b_2+\frac{\mathbb{E}(\Tr((A\tilde{B}^2A)^2))^{1/2}a_4^{1/2}+b_6^{2/3}a_6^{1/3}}{\eta^2}\Big)\\
\leq& \frac{12}{N^2\eta^2}\Big(a_2b_2+\frac{m_{A^2\ast B^2}(1^2,1^2)^{1/2}a_4^{1/2}+b_6^{2/3}a_6^{1/3}}{\eta^2}\Big),
\end{align*}
where we used Lemma \ref{trace_Weingarten} on the last inequality. Then, by Lemma \ref{concentration_result_additive},
$$\Var(f_{B})\leq \frac{4\sqrt{b_4a_4}}{\eta^4N^2},$$
and by Lemma \ref{concentration_result_additive_tilde} with $A$ and $B$ switched,
\begin{align*}
\Var(\tilde{f}_{A})&\leq \frac{4}{N^2\eta^2}\left(a_2b_2+\sqrt{a_4b_4}+\frac{2\sqrt{m_{A^2\ast B^2}(1^2,1^2)}a_4^{1/2}}{\eta^2}\right).
\end{align*}
Putting all previous bounds together gives then 
\begin{align*}
\mathbb{E}\vert z\Tr\left(A\Delta_{A}\right)\vert\leq& \sqrt{\frac{8}{N^2\eta^2}\left(b_2+\frac{a_2b_2+b_4}{\eta^2}\right)\cdot\frac{4}{N^2\eta^2}\left(a_2b_2+\sqrt{a_4b_4}+\frac{2\sqrt{m_{A^2\ast B^2}(1^2,1^2)}a_4^{1/2}}{\eta^2}\right)}\\
&\hspace{2cm}+\sqrt{\frac{4\sqrt{b_4a_4}}{\eta^4N^2}\cdot\frac{12}{N^2\eta^2}\Big(a_2b_2+\frac{m_{A^2\ast B^2}(1^2,1^2)^{1/2}a_4^{1/2}+b_6^{2/3}a_6^{1/3}}{\eta^2}\Big)}\\
\leq &\frac{4}{N^2\eta^2}\Bigg(\sqrt{2\left(b_2+\frac{a_2b_2+b_4}{\eta^2}\right)\cdot\left(a_2b_2+\sqrt{a_4b_4}+\frac{2\sqrt{m_{A^2\ast B^2}(1^2,1^2)}a_4^{1/2}}{\eta^2}\right)}\\
&\hspace{3cm}+\sqrt{3\frac{\sqrt{b_4a_4}}{\eta^2}\left(a_2b_2+\frac{m_{A^2\ast B^2}(1^2,1^2)^{1/2}a_4^{1/2}+b_6^{2/3}a_6^{1/3}}{\eta^2}\right)}\Bigg).
\end{align*}
On the other hand, by Lemma \ref{bound_Delta_additif} 
$$\mathbb{E}\vert \Tr(\epsilon_1\Delta_A)\vert\leq \frac{8b_4^{1/4}b_2^{1/2}a_4^{1/4}\Vert \epsilon_1\Vert_{\infty}}{\eta^4N^2}\leq \frac{8b_4^{1/4}b_2^{3/2}a_4^{1/4}}{\eta^5N^2},$$
where we used Lemma \ref{bound_expectation_matrix} on the last inequality. Therefore,
\begin{align*}
\vert z&\mathbb{E}\Tr\left((A+\epsilon_{1})\Delta_{A}\right)\vert\\
\leq&\frac{4}{N^2\eta^2}\Bigg(\sqrt{2\left(b_2+\frac{a_2b_2+b_4}{\eta^2}\right)\cdot\left(a_2b_2+\sqrt{a_4b_4}+\frac{2\sqrt{m_{A^2\ast B^2}(1^2,1^2)}a_4^{1/2}}{\eta^2}\right)}\\
&\hspace{3cm}+\sqrt{3\frac{\sqrt{b_4a_4}}{\eta^2}\left(a_2b_2+\frac{m_{A^2\ast B^2}(1^2,1^2)^{1/2}a_4^{1/2}+b_6^{2/3}a_6^{1/3}}{\eta^2}\right)}+2\frac{b_4^{1/4}b_2^{3/2}a_4^{1/4}}{\eta^3}\Bigg)\\
\leq &\frac{4b_2\sqrt{a_2}}{N^2\eta^2}\Bigg(\sqrt{2\left(1+\frac{a_2+b_4/b_2}{\eta^2}\right)\cdot\left(1+\frac{\sqrt{a_4b_4}}{a_2b_2}+\frac{2\sqrt{m_{A^2\ast B^2}(1^2,1^2)}a_4^{1/2}}{a_2b_2\eta^2}\right)}\\
&\hspace{3cm}+\sqrt{3\frac{\sqrt{b_4a_4}}{b_2\eta^2}\left(1+\frac{m_{A^2\ast B^2}(1^2,1^2)^{1/2}a_4^{1/2}+b_6^{2/3}a_6^{1/3}}{a_2b_2\eta^2}\right)}+2\frac{b_4^{1/4}b_2^{1/2}a_4^{1/4}}{\eta^3\sqrt{a_2}}\Bigg).
\end{align*}
Since $\Tr(B)=\Tr(A)=0$, $b_2=\sigma_B^2$ and $a_2=\sigma_A^2$, yielding
\begin{align*}
\vert z&\mathbb{E}\Tr\left((A+\epsilon_{1})\Delta_{A}\right)\vert\\
\leq &\frac{4\sigma_B^2\sigma_A}{N^2\eta^2}\Bigg(\sqrt{2\left(1+\frac{\sigma_A^2+\sigma_B^2\theta_B}{\eta^2}\right)\cdot\left(1+\sqrt{\theta_A\theta_B}+\frac{2\sqrt{m_{A^2\ast B^2}(1^2,1^2)\theta_A}}{\sigma_B^2\eta^2}\right)}\\
&\hspace{3cm}+\sqrt{3\frac{\sqrt{\theta_B\theta_A}\sigma_A^2}{\eta^2}\left(1+\frac{m_{A^2\ast B^2}(1^2,1^2)^{1/2}a_4^{1/2}+b_6^{2/3}a_6^{1/3}}{\sigma_A^2\sigma_B^2\eta^2}\right)}+2\frac{\theta_B^{1/4}\sigma_B^3\theta_A^{1/4}}{\eta^3}\Bigg),
\end{align*}
where we recall that $\theta_X=\frac{x_4^0}{\sigma_X^4}$ is the kurtosis of $\mu_X$ for $X$ self-adjoint. Finally, taking into account the term $(1+\epsilon_2/\vert z\vert)\leq (1+\frac{a_2+b_2}{\eta^2})$ \eqref{expression_omega_A_z} yields
$$\vert \delta\vert\leq \frac{C_{thres,A}}{3N^2}\eta,$$
with
\begin{align*}
C_{thres,A=}&\frac{12\sigma_B^2\sigma_A}{\eta^3}\left(1+\frac{\sigma_A^2+\sigma_B^2}{\eta^2}\right)\Bigg(\sqrt{2\left(1+\frac{\sigma_A^2+\sigma_B^2\theta_B}{\eta^2}\right)\cdot\left(1+\sqrt{\theta_A\theta_B}+\frac{2\sqrt{m_{A^2\ast B^2}(1^2,1^2)\theta_A}}{\sigma_B^2\eta^2}\right)}\\
&\hspace{2cm}+\sqrt{3\frac{\sqrt{\theta_B\theta_A}\sigma_A^2}{\eta^2}\left(1+\frac{m_{A^2\ast B^2}(1^2,1^2)^{1/2}a_4^{1/2}+b_6^{2/3}a_6^{1/3}}{\sigma_A^2\sigma_B^2\eta^2}\right)}+2\frac{\theta_B^{1/4}\sigma_B^3\theta_A^{1/4}}{\eta^3}\Bigg).
\end{align*}
The two bounds of the statement are deduced from the latter expressions and \eqref{expression_omega_A_z} with the fact that $\Tr(\epsilon_1)\in \mathbb{C}^+$.
\end{proof} 

\begin{proof}[Proof of Proposition \ref{convergence_subord_B}]
By Lemma \ref{subor_kargin}, we have to estimate $\Tr(R_{A}(z))=\frac{1}{\mathbb{E} m_{H}}\Tr(G_{A}(\omega_{A})\mathbb{E}_{U}\Delta_{A})$. By Proposition \ref{minoration_omega}, for $N\geq \sqrt{C_{thres,A}}$, $\Im\omega_{A}\geq 2\eta/3$, which implies
$$\Vert G_{A}(\omega_{A})\Vert_{\infty}\leq\frac{3}{2\eta}.$$
Hence, \eqref{eq:expression_TDelta} and Cauchy-Schwartz inequality yield
\begin{align*}
\vert \Tr(R_{A}(z))\vert=&\left\vert\frac{1}{\mathbb{E} m_{H}}\Tr(G_{A}(\omega_{A})\mathbb{E}_{U}\Delta_{A})\right\vert
\\
\leq& \frac{1}{\vert z^2\mathbb{E}m_{H}(z)\vert}(\sqrt{\Var(zm_{H})\Var(zf'_{AG_A(\omega_A)})}+\sqrt{\Var(zf'_{A})\Var(zf'_{G_{A}}(\omega_A)})
\\
\leq& \frac{2\Vert G_{A}(\omega_A)\Vert_{\infty}}{ \vert z\vert\cdot\vert z\mathbb{E}m_{H}(z)\vert}\sqrt{\Var(zf'_{A})\Var(zm_{H})}\\
\leq& \frac{3}{ \eta\vert z\vert\cdot\vert z\mathbb{E}m_{H}(z)\vert}\sqrt{\Var(zf'_{A})\Var(zm_{H})}.
\end{align*}
By Lemma \ref{concentration_result_additive_z} with $A$ and $B$ switched, we get
$$\Var(zm_{H})\leq \frac{8}{N^2\eta^2}\left(b_2+\frac{b_2a_2+b_4}{\eta^2}\right),$$
and
$$\Var(zf_{A})\leq \frac{12}{N^2\eta^2}\Big(a_2b_2+\frac{m_{A^2\ast B^2}(1^2,1^2)^{1/2}a_4^{1/2}+b_6^{2/3}a_6^{1/3}}{\eta^2}\Big).$$
Hence,
\begin{align*}
\sqrt{\Var(zf_{A})\Var(zm_{H})}\leq& \frac{4\sqrt{6}b_2}{N^2\eta^2}\sqrt{1+\frac{a_2+b_4/b_2}{\eta^2}}\sqrt{a_2+\frac{m_{A^2\ast B^2}(1^2,1^2)^{1/2}a_4^{1/2}+b_6^{2/3}a_6^{1/3}}{b_2\eta^2}}.
\end{align*}
Then, using \eqref{lower_bound_expectation} yields $\frac{1}{\vert z\mathbb{E}m_{H}(z)}\leq 1+\frac{a_2+b_2}{\eta^2}$. Therefore, since $a_2=\sigma_A^2$ and $b_2=\sigma_B^2$,
$$\vert\Tr(R_{A}(z))\vert\leq \frac{C_{bound,A}}{\vert z\vert N^2},$$
with 
\begin{align*}
&C_{bound,A}\\
=&\frac{12\sqrt{6}\sigma_B^2\sigma_A}{\eta^3}\left(1+\frac{\sigma_A^2+\sigma_B^2}{\eta^2}\right)\sqrt{1+\frac{\sigma_A^2+\theta_B\sigma_B^2}{\eta^2}}\sqrt{1+\frac{m_{A^2\ast B^2}(1^2,1^2)^{1/2}a_4^{1/2}+b_6^{2/3}a_6^{1/3}}{a_2b_2\eta^2}}.
\end{align*}
Writing $\eta=\kappa\sigma_1$ in the latter expression yields the result.
\end{proof}
\subsection{Subordination in the multiplicative case}
Building on the latter method, we prove an analogue of Proposition \ref{convergence_subord_B} in the multiplicative case, which gives the following. Recall that $\tilde{m}_{\mu}=1+zm_{\mu}$ for $\mu$ probability measure. 
\begin{proposition}\label{convergence_subord_multi}
Let $z\in \mathbb{C}^+$ with $\Im z=\eta=\kappa\tilde{\sigma}_1$ and suppose that
$$N^{2}\geq  \frac{\vert z\vert}{\eta^2}C_{thres,A}(\eta),$$
with $C_{thres,A}$ given in Proposition \ref{minoration_omega_multiplicative}. Then $\Im \omega_{A}\geq 2\eta/3$ and 
$$\vert \tilde{m}_{M}-\tilde{m}_{A}(\omega_{A})\vert  \leq \frac{C_{bound,A}(\kappa)}{N^2},$$
with 
$$C_{bound,A}(\kappa)=24\frac{a_{\infty}^3b_2}{\kappa^3\tilde{\sigma}_1^3}\left(1+\frac{m_{A\ast B}^N(1^{3},21^{2})}{\kappa^2\tilde{\sigma}_1^2b_2}\right)\cdot\left(1+\frac{a_2}{\kappa\tilde{\sigma}_1}+\frac{a_2\sigma_B^2+\tilde{\sigma}_1^2}{(1-N^{-2})\kappa^2\tilde{\sigma}_1^2}\right).$$
\end{proposition}
The result for the subordination involving $\omega_B$ slightly differ, since we wish to avoid any boundedness assumption on the support of $B$.
\begin{proposition}\label{bound_subord_B}
For $z\in\mathbb{C}^+$ with $\eta=\Im z=\kappa\tilde{\sigma}_1$ and for $N^2\geq \frac{C_{thres,B}(\eta)\vert z\vert }{\eta^3}$ with $C_{thres, \,B}$ given in Lemma \ref{minoration_omega_B_multi}, then $\Im \omega_{B}\geq 2\eta/3$ and 
$$\vert \tilde{m}_{M}(z)-\tilde{m}_{B}(\omega_B)\vert\leq\frac{C_{bound,B}(\kappa)}{N^2},$$
with 
\begin{align*}
&C_{bound,B}(\kappa)=\frac{4\sqrt{2}a_{\infty}b_2}{\kappa^2\tilde{\sigma}_1^2}\left(1+\frac{1}{\kappa\tilde{\sigma}_1}+\frac{\sigma_A^2+\sigma_B^2}{(1-N^{-2})\kappa^2\tilde{\sigma}_1^2}\right)\cdot\sqrt{1+\frac{m_{A\ast B}^N(1^3,21^2)}{b_2\eta^2}}\\
&\hspace{1cm}\cdot\Bigg(\frac{a_\infty^{3/2}}{\kappa\tilde{\sigma}_1}\left(\sqrt{\frac{b_4}{b_2}}+\sqrt{\frac{9b_6}{4b_2\kappa^2\tilde{\sigma}_1^2}}\right)+\sqrt{2}\sqrt{1+\frac{m_{A\ast B}^N(1^3,21^2)}{b_2\kappa^2\tilde{\sigma}_1^2}}+\frac{3}{\sqrt{2b_2}\kappa\tilde{\sigma}_1}\sqrt{b_4+\frac{m_{A\ast B}^N(1^3,2^3)}{\kappa^2\tilde{\sigma}_1^2}}\Bigg),
\end{align*}
\end{proposition}
As in the additive case, we first need to control the behavior of $\omega_{A}$ and $\omega_{B}$. Let us first apply Nevanlinna's theory to the various analytic functions involved in the subordination. 
\begin{lemma}\label{lower_bounds_multi}
There exist a probability measure $\rho$ and $N$ probability measures $\rho_{i}$, $1\leq i\leq N$ on $\mathbb{R}$ such that 
$$\frac{1}{\mathbb{E}(f_{A})}=-z+a_2-\gamma m_{\rho}(z),$$
with $\gamma\leq \frac{\tilde{\sigma}_A^2+a_2\sigma_B^2}{1-N^{-2}}$,
$$\mathbb{E}(G_{M}(z))_{ii}^{-1}=-z+A_{ii}-\gamma_{i}m_{\rho_{i}}(z),$$
with $\gamma_{i}\leq \frac{1}{1-N^{-2}}A_{ii}\sigma_B^2$, and 
$$\mathbb{E}(U^*A^{1/2}G_{M}(z)A^{1/2}U)_{ii}^{-1}=-z+\tilde{B}_{ii}+\frac{\sigma_A^2}{1-1/N^2}-\gamma'_{i}m_{\rho'_{i}}(z),$$
where $\tilde{B}=\beta B$ with $\beta<1$ and $\gamma'_{i}\leq \gamma'_{A}$ with
$$\gamma'_{A}=\frac{k_{3}(A)+(b_2-\sigma_A^2)\sigma_A^2+\delta_{N}}{(1-1/N^2)^2(1-4/N^2)} ,$$
where 
$$\delta_{N}\leq \frac{(10+4a_2+5a_3)b_2}{N}.$$
\end{lemma}
\begin{proof}
First, note that for $z\in\mathbb{C}^+$, $f_{A}(z)=\Tr(A^{1/2}G_{M}A^{1/2})\in \mathbb{C}^+$. Moreover, as $z$ goes to infinity, 
$$\mathbb{E}(f_{A}(z))=-\Tr(A)z^{-1}-\mathbb{E}(\Tr(AM))z^{-2}-\mathbb{E}(\Tr(AM^2))z^{-3}+o(z^{-3}).$$
On the one hand, writing $\tilde{B}=UBU^*$,
$$\mathbb{E}(\Tr(AM))=\mathbb{E}(\Tr(AA^{1/2}\tilde{B}A^{1/2}))=\mathbb{E}(\Tr(A^2\tilde{B}))=\Tr(A^2)\Tr(B)=a_2,$$
where we used Lemma \ref{trace_Weingarten} and the hypothesis $\Tr(B)=1$. On the other hand, by Lemma \ref{trace_Weingarten},
\begin{align*}
\mathbb{E}(\Tr(AM^2))=\mathbb{E}(\Tr(A^2\tilde{B}A\tilde{B}))=m_{A\ast B}(21,1^2)=&\frac{1}{1-N^{-2}}\left(a_3b_1^2+a_2a_1b_2-a_2a_1b_1^2-\frac{1}{N^2}a_3b_2\right)\\
=&\frac{1}{1-N^{-2}}\left(a_3+a_2\sigma_B^2-\frac{1}{N^2}a_3b_2\right),
\end{align*}
where we used $a_1=b_1=1$ on the last equality. Hence, by Theorem \ref{Nevanlina} and \eqref{F_mu_expression}, there exists a probability measure $\rho$ such that 
$$\frac{1}{\mathbb{E}(f_{A})}=-z+a_2-\gamma m_{\rho}(z),$$
with 
$$\gamma=\frac{1}{1-N^{-2}}\left(a_3+a_2\sigma_B^2-\frac{1}{N^2}a_3b_2\right)-a_2^2= \tilde{\sigma}_A^2+a_2\sigma_B^2+\frac{a_3-a_3b_2+a_2\sigma_B^2}{N^2-1}\leq\frac{\tilde{\sigma}_A^2+a_2\sigma_B^2}{1-N^{-2}},$$
where $\tilde{\sigma}_A^2= a_3-a_2^2$ and we used $b_2\geq b_1\geq1$. Likewise, as $n$ goes to infinity,
$$\mathbb{E}(G_{M})=-z^{-1}-\mathbb{E}(M)z^{-2}-\mathbb{E}(M^2)z^{-3}+o(z^{-3}).$$
By Lemma \ref{Weingarten_calculus}, using $\Tr(B)$ gives $\mathbb{E}(M)=\Tr(B)A=A$ and for $1\leq i\leq N$
\begin{align*}
\mathbb{E}(M^2)_{ii}=&(A^{1/2}\mathbb{E}(UBU^*AUBU^*)A^{1/2})_{ii}\\
=&\frac{1}{1-1/N^2} A_{ii}\left(\Tr(A)\Tr(B^2)-\Tr(A)\Tr(B)^2+A_{ii}(\Tr(B)^2-\frac{1}{N^2}\Tr(B^2))\right)\\
=&\frac{1}{1-1/N^2}\left(A_{ii}^2\left(1-\frac{b_2}{N^2}\right)+A_{ii}\sigma_B^2\right).
\end{align*}
Hence, by Theorem \ref{Nevanlina}, \eqref{F_mu_expression} and the fact that $\Tr(A)=1$, there exists a probability measure $\rho_{i}$ such that 
$$ \mathbb{E}(G_{M})_{ii}^{-1}=-z+A_{ii}-\gamma_{i}m_{\rho_{i}}(z),$$
where 
\begin{align*}
\gamma_{i}=&\mathbb{E}(M^2)_{ii}-(\mathbb{E}(M)_{ii})^2=\frac{1}{1-1/N^2}\left(A_{ii}^2\left(1-\frac{b_2}{N^2}\right)+A_{ii}\sigma_B^2\right)-A_{ii}^2\\
=&\frac{A_{ii}\sigma_{B}^2}{1-1/N^2}+A_{ii}^2\frac{1-b_2N^{-2}-1+N^{-2}}{1-1/N^2}\leq \frac{1}{1-N^2}A_{ii}\sigma_B^2.
\end{align*}
Similarly, $\mathbb{E}(U^*A^{1/2}G_{M}(z)A^{1/2}U)$ maps $\mathbb{C}^+$ to $\mathbb{C}^+$, and as $N$ goes to infinity, 
\begin{align*}
\mathbb{E}(U^*A^{1/2}G_{M}&(z)A^{1/2}U)\\
=&-\mathbb{E}(U^*AU)z^{-1}-\mathbb{E}(U^*AUBU^*AU)z^{-2}-\mathbb{E}(U^*AUBU^*AUBU^*AU)z^{-3}+o(z^{-3}).
\end{align*}
Since $\mathbb{E}(U^*AU)=\Tr(A)\Id=\Id$, by Theorem \ref{Nevanlina} and \eqref{F_mu_expression} there exists for each $1\leq i\leq N$ a probability measure $\rho'_{i}$ such that 
\begin{align*}
\mathbb{E}(U^*A^{1/2}G_{M}(z)A^{1/2}U)_{ii}^{-1}=&-z+\mathbb{E}(U^*AUBU^*AU)_{ii}+\Big[\mathbb{E}(U^*AUBU^*AUBU^*AU)_{ii}\\
&\hspace{3cm}-\mathbb{E}(U^*AUBU^*AU)_{ii}^{2}\Big]m_{\rho'_{i}}(z).
\end{align*}
By Lemma \ref{Weingarten_calculus}, 
$$\mathbb{E}(U^*AUBU^*AU)_{ii}= \frac{1}{1-1/N^2}B_{ii}\left(1-\frac{a_2}{N^2}\right)+\frac{1}{1-1/N^2}\sigma_A^2.$$
Since $a_2\geq 1$, $\frac{1-\frac{a_2}{N^2}}{1-1/N^2}\leq 1$, which implies that $\mathbb{E}(U^*AUBU^*AU)_{ii}-\frac{\sigma_A^2}{1-1/N^2}=\alpha B_{ii}$ with $\alpha>1$ independent of $i$. Likewise, we have by Lemma \ref{Weingarten_calculus}
\begin{align*}&(1-1/N^2)(1-4/N^2)\mathbb{E}\left[U^{*}AUBU^{*}AUBU^{*}AU\right]\\
=&B^2\Big(1+(1+4/N^2)a_3/N^2-6a_2/N^2\Big)
+B\Big(2(a_2-1)+4/N^2(a_2-a_3)\Big)
+\Big(a_3+a_2b_2+2-b_2-3a_2\Big)
\end{align*}
Hence, after simplifying and removing negative terms in the error part, we get
\begin{align*}
&\mathbb{E}(U^*AUBU^*AUBU^*AU)-\mathbb{E}(U^*AUBU^*AU)^{2}\\
=&\frac{B^2\Big(1+(1+4/N^2)a_3/N^2-6a_2/N^2\Big)
+B\Big(2(a_2-1)+4/N^2(a_2-a_3)\Big)
+\Big(a_3+a_2b_2+2-b_2-3a_2\Big)}{(1-1/N^2)(1-4/N^2)}\\
&\hspace{3cm}-\left(\frac{1}{1-1/N^2}B\left(1-\frac{a_2}{N^2}\right)+\frac{1}{1-1/N^2}\sigma_A^2\right)^2 \\
&=\frac{a_3-3a_2+2+\sigma_A^2(b_2-\sigma_A^2)+\delta_{N}}{(1-1/N^{2})^2(1-4/N^2)}=\frac{k_{3}(A)+\sigma_A^2(b_2-\sigma_A^2)+\delta_{N}}{(1-1/N^{2})^2(1-4/N^2)},
\end{align*}
where $k_{3}(A)$ denotes the third free cumulant of $\mu_A$ as defined in Section \ref{Section:free_probabilistic_background}, and with the bound
\begin{align*}
\delta_{N}\leq& \frac{1}{N^2}\left(4\sigma_A^4+\sigma_A^2+b_2+B(2a_2+6)\sigma_A^2+B^2(3+a_3)\right)+\frac{1}{N^4}(4B\tilde{\sigma}_A^2+3a_3B^2)\\
\leq& \frac{(9+2a_2+a_3)b_2}{N}+\frac{4a_2^2+a_2+b_2}{N^2}+\frac{(4a_2+3a_3)b_2}{N^3}\leq \frac{(10+4a_2+5a_3)b_2}{N},
\end{align*}
where we used the fact that $B$ and $B^2$ are smaller than $b_2N$.
\end{proof}

\begin{proposition}\label{minoration_omega_multiplicative}
Let $z\in\mathbb{C}$ with $\Im(z):=\eta$. Then, whenever
$$N^{2}\geq \frac{\vert z\vert}{\eta^3}C_{thres,A}(\eta),$$
with 
\begin{align*}
C_{thres,A}(\eta)=48b_2a_\infty^3&\left(1+\frac{m_{A\ast B}^N(1^{3},21^{2})}{\eta^2\sigma_B^2}\right)\cdot\left(1+\frac{m_2}{\eta}+\frac{\tilde{\sigma}_M^2}{\eta^2}\right)\\
&\hspace{2cm}\cdot\left(1+\frac{k_3(B)+\sigma_B^2(a_2-\sigma_B^2))+\frac{(10+4b_2+5b_3)a_2}{N}}{(1-N^{-2})^2(1-4N^{-2})a_{\infty}\eta}\right),
\end{align*}
then,
$$\Im \omega_{A}\geq \frac{2}{3}\eta \hspace{1cm}\text{ and } \hspace{1 cm}\Vert G_{A}(\omega_{A})\vert \leq \frac{3}{2\eta}.$$
\end{proposition}
\begin{proof}
By Lemma \ref{lower_bounds_multi}, 
$$A(\mathbb{E} MG_{M})^{-1}=\mathbb{E}[UB^{1/2}G_{M'}B^{1/2}U^*]^{-1}=-z+\tilde{A}+\frac{\sigma_A^2}{1-1/N^2}+\Upsilon$$
with $\tilde{A}\leq A$ and $\vert\Upsilon_{ii}\vert\leq \frac{\gamma'_{B}}{\eta}$, with $\gamma'_{B}$ as $\gamma'_{A}$ in Lemma \ref{lower_bounds_multi} with $A$ and $B$ switched. Then, by \eqref{expression_subor_other},
\begin{align*}
\omega_AA=&A^{2}-(A+\omega_A\mathbb{E}\Delta_{A})(A\mathbb{E}[MG_{M}]^{-1})\\
=&A^2+zA-A\left(\tilde{A}+\frac{\sigma_A^2}{1-1/N^2}\right)-A\Upsilon+z\omega_{A}\mathbb{E}\Delta_{A}-\omega_A\mathbb{E}\Delta_A\left(\tilde{A}+\frac{\sigma_A^2}{1-1/N^2}+\Upsilon\right).
\end{align*}
Hence, using the fact that $\Tr(A)=1$ and $\Tr(\mathbb{E}\Delta_{A})=0$, we get by taking the trace in the latter formula
\begin{equation}\label{expression_omega}
\omega_{A}=z+\Tr(A(A-\tilde{A})-\frac{\sigma_A^2}{1-1/N^2}-\Tr(A\Upsilon)-\omega_{A}\Tr(\mathbb{E}\Delta_A(\tilde{A}+\Upsilon)).
\end{equation}
Remark that $\Tr(A(A-\tilde{A}))=\Tr(A^{1/2}(A-\tilde{A})A^{1/2})$ and $\Tr(A\Upsilon)=\Tr(A^{1/2}\Upsilon A^{1/2})$. Hence, since $\tilde{A}$ is self-adjoint and $\Upsilon\in \mathcal{H}^{-}_{n}$, $\Tr(A(A-\tilde{A})-\Tr(A\Upsilon))\in\mathbb{C}^+$. Therefore,
$$\Im \omega_A\geq \Im z-\left\vert \omega_{A}\Tr(\mathbb{E}\Delta_A(\tilde{A}+\Upsilon))\right\vert.$$
On the other hand, by the definition of $\omega_A$ and $\Delta_A$, we have 
\begin{align*}
\omega_A\Tr(\mathbb{E}\Delta_A(\tilde{A}+\Upsilon))=&\frac{z\mathbb{E}f_{A}}{\mathbb{E}\tilde{m}_{M}}\frac{z}{\mathbb{E}f_{A}}\Tr\left(\mathbb{E}[(f_{A}-\mathbb{E}f_{A})(\tilde{A}+\Upsilon)G_{M}-(m_{M}-\mathbb{E}m_{M})(\tilde{A}+\Upsilon)AG_{M}]\right)\\
=&\frac{1}{\tilde{m}_{M}}\mathbb{E}\Big((zf_{A}-z\mathbb{E}f_{A})\cdot(zf_{\tilde{A}+\Upsilon}-z\mathbb{E}f_{\tilde{A}+\Upsilon})\\
&\hspace{3cm}+(zm_{M}-z\mathbb{E}m_{M})\cdot\left(zf_{A(\tilde{A}+\Upsilon)}-z\mathbb{E}f_{A(\tilde{A}+\Upsilon)}\right)\Big),
\end{align*}
where we write $f_{T}=\Tr(TG_{M})$ and omitted the dependence in $z$. Hence, by Cauchy-Schwartz inequality and Lemma \ref{concentration_result_multi_z},
\begin{align*}
\left\vert \omega_A\Tr(\Delta_A(\tilde{A}+\Upsilon))\right\vert\leq &\frac{1}{\vert \tilde{m}_{M}(z)\vert}\left(\sqrt{\Var(zf_{A})\Var\left(zf_{\tilde{A}+\Upsilon}\right)}+\sqrt{\Var(zm_{M})\Var\left(zf_{A(\tilde{A}+\Upsilon)}\right))}\right)\\
\leq& \frac{1}{\vert \tilde{m}_{M}(z)\vert}\frac{16\Vert A\Vert_{\infty}^2\Vert \tilde{A}+\Upsilon\Vert_{\infty}}{N^2\eta^2}\left(\Tr(B^2)+\frac{m_{A\ast B}(1^3,21^2)}{\eta^2}\right)\\
\leq&\frac{1}{\vert\tilde{m}_{M}(z)\vert}\frac{16a_{\infty}^3\left(1+\frac{\gamma'_{B}}{a_{\infty}\eta}\right)}{N^2\eta^2}\left(b_2+\frac{m_{A\ast B}(1^3,21^2)}{\eta^2}\right).
\end{align*}
Therefore, whenever 
$$N^{2}\geq \frac{1}{\vert\tilde{m}_{M}(z)\vert}\frac{48a_{\infty}^3\left(1+\frac{\gamma'_{B}}{a_{\infty}\eta}\right)}{N^2\eta^3}\left(b_2+\frac{m_{A\ast B}(1^3,21^2)}{\eta^2}\right),$$
for some $\alpha<1$, 
$$\Im \omega_{A}\geq 2\eta/3 \text{  and  } \Vert G_{A}(\omega_{A})\Vert \leq \frac{3}{2\eta}.$$
Since $\frac{1}{\tilde{m}_{M}(z)}=z-m_2+\tilde{\sigma}_M^2m_{\rho_{M}}(z)$,
$$\frac{1}{\vert\tilde{m}_{M}(z)\vert}\frac{48a_{\infty}^3\left(1+\frac{\gamma'_{B}}{a_{\infty}\eta}\right)}{\eta^3}\left(b_2+\frac{m_{A\ast B}(1^3,21^2)}{\eta^2}\right)\leq\frac{\vert z\vert }{\eta^3}C_{thres,A}(\eta),$$
with 
$$C_{thres,A}(\eta)=48b_2a_{\infty}^3\left(1+\frac{m_{A\ast B}(1^3,21^2)}{\eta^2b_2}\right)\cdot\left(1+\frac{\gamma'_{B}}{a_{\infty}\eta}\right)\cdot\left(1+\frac{m_2}{\eta}+\frac{\tilde{\sigma}_M^2}{\eta^2}\right).$$
\end{proof}
\begin{proof}[Proof of Proposition \ref{convergence_subord_multi}]
Suppose that $N\geq \frac{\vert z\vert }{\eta^3}C_{thres,A}(\eta)$. Then, by Proposition \ref{minoration_omega_multiplicative}, 
$$\Vert AG_{A}(\omega_{A})\Vert_{\infty}\leq \frac{3\Vert A\Vert_{\infty}}{2\eta}.$$
Hence, by Cauchy-Schwartz inequality and Lemma \ref{concentration_result_multi_z},
\begin{align*}
\left\vert\Tr(AG_{A}(\omega_{A})\mathbb{E}\Delta_{A})\right\vert\leq&\frac{1}{\left\vert z\mathbb{E}f_{A}(z)\right\vert}\sqrt{\Var( zf_{A})\Var(zf_{AG_{A}(\omega_{A})})}+\sqrt{\Var(zf_{A^2G_{A}(\omega_{A})})\Var(zm_{M})}\\
\leq&\frac{24a_{\infty}^3}{N^2\eta^3\vert z\mathbb{E}f_{A}\vert}\left(b_2+\frac{m_{A\ast B}(1^3,21^2)}{\eta^2}\right).
\end{align*}
Hence, by \eqref{subor_karg_multiplication} and the fact that $\mathbb{E}\Tr\Delta_{A}=0$,
\begin{align*}
\vert zm_{M}-\omega_{A}m_{A}(\omega_{A})\vert =&\left\vert \omega_{A}\Tr(G_{A}(\omega_{A})\mathbb{E}\Delta_{A})\right\vert\\
=& \left\vert  \Tr(AG_{A}(\omega_{A})\mathbb{E}\Delta_{A})\right\vert\\
\leq& \frac{24a_{\infty}^3}{N^2\eta^3\vert z\mathbb{E}f_{A}\vert}\left(b_2+\frac{m_{A\ast B}(1^3,21^2)}{\eta^2}\right).\\
\end{align*}
By Lemma \ref{lower_bounds_multi},
$$\vert \mathbb{E}(f_{A})^{-1}/z\vert\leq 1+\frac{a_2}{\eta}+\frac{\tilde{\sigma}_A^2+a_2\sigma_B^2}{\eta^2(1-1/N^2)},$$
which yields 
$$\vert zm_{M}(z)-\omega_{A}m_{A}(\omega_{A})\vert\leq \frac{C_{bound,A}}{N^2},$$
with 
$$C_{bound,A}=\frac{24a_{\infty}^3b_2}{\eta^3}\left( 1+\frac{a_2}{\eta}+\frac{\tilde{\sigma}_A^2+a_2\sigma_B^2}{\eta^2(1-1/N^2)}\right)\cdot\left(1+\frac{m_{A\ast B}(1^3,21^2)}{b_2\eta^2}\right).$$
Writing $\eta=\kappa\tilde{\sigma}_1$ in the latter equation yields the result.
\end{proof}
We next turn to the concentration bound for the subordination involving $\omega_B$. Although we use the same method as for $\omega_A$, the proof slightly differs to avoid a bound on the norm of $B$.
\begin{lemma}\label{minoration_omega_B_multi}
For $N^2\geq \frac{\vert z\vert C_{thres,B}(\eta)}{\eta^3}$, then
$\Im(\omega_B)\geq 2\eta/3$, and 
$$\Vert G_{B}(\omega_{B})\Vert\leq \frac{3}{2\eta},$$
with 
\begin{align*}
&C_{thres,B}(\eta)= 24a_{\infty}b_2\sqrt{1+\frac{m_{A\ast B}^N(1^3,21^2))}{b_2\eta^2}}\cdot \left(1+\frac{a_2}{\eta}+\frac{\tilde{\sigma}_A^2+a_2\sigma_B^2}{(1-N^{-2})\eta^2}\right)
\Bigg(\sqrt{1+\frac{m_{A\ast B}^N(1^3,21^2)}{b_2\eta^2}}\\&\hspace{4cm}+\frac{a_\infty^{3/2}\sqrt{b_4}}{\sqrt{2b_2}\eta}+(1+2\sqrt{b_2}a_{\infty}^{3/2}/\eta)\frac{k_3(A)+\sigma_A^2(b_2-\sigma_A^2)+\frac{b_2(10+4a_2+5a_3)}{N}}{(1-N^{-2})^2(1-4N^{-2})^2\eta\sqrt{b_2}}\Bigg).
\end{align*}
\end{lemma}
\begin{proof}
Taking the trace in \eqref{expression_subor_other_B} yields 
$$\omega_B=\Tr(B)-\Tr(B\mathbb{E}(\hat{M}G_{\hat{M}})^{-1})+\mathbb{E}\Tr(\Delta_BB\mathbb{E}(\hat{M}G_{\hat{M}})^{-1}).$$
By Lemma \ref{lower_bounds_multi}, $B\mathbb{E}(\hat{M}G_{\hat{M}})^{-1}=\mathbb{E}(U^*A^{1/2}G_MA^{1/2}U]^{-1}=z-\beta B+\frac{\Var(\mu_A)}{1-1/N^2}+\Upsilon$ with $0<\beta<1$ $\Vert$, $\Upsilon\Vert\leq\frac{\gamma'_{A}}{\eta}$ and using a similar reasoning as in Proposition \ref{minoration_omega_multiplicative} gives 
$$\Im \omega_B\geq z-\delta,$$
with $\vert\delta\vert\leq \vert \Tr((\beta B+\Upsilon)\Delta_{B})\vert$. Using the definition of $\Delta_{B}$ from Lemma \ref{Lem:subor_multi} and Cauchy-Schwartz inequality yields then
\begin{align*}
\vert \delta\vert\leq& \frac{1}{\vert\mathbb{E}f_{A}(z)\vert}\left(\sqrt{\Var(zf_{A})\Var(\tilde{f}_{\beta B+\Upsilon})}+\sqrt{\Var(\tilde{m}_{M})\Var(\Tr((\beta B+\Upsilon)U^{*}A^{1/2}G_{M}A^{1/2}U)}\right)\\
\leq&\frac{1}{\vert\mathbb{E}f_{A}(z)\vert}\Bigg(\sqrt{\Var(zf_{A})}\left(\sqrt{\Var(\tilde{f}_{B})}+\sqrt{\Var(\tilde{f}_\Upsilon)}\right)+\sqrt{\Var(\tilde{m}_{M})}\Big(\sqrt{\Var(\tilde{m}_{M})}\\&\hspace{2cm}+\sqrt{\Var(\Tr(U\Upsilon U^{*}A^{1/2}G_{M}A^{1/2}))}\Big)\Bigg),
\end{align*}
with $\tilde{f}_{T}=\Tr(TG_{M'})$. By Lemma \ref{concentration_result_multi_z}, 
$$\Var(zf_{A})\leq \frac{8a_{\infty}^3}{\eta^2N^2}\left(b_2+\frac{m_{A\ast B}^N(1^3,21^2)}{\eta^2}\right),$$
and
$$\Var(\tilde{m}_{M}(z))=\Var(zm_{z})\leq \frac{8a_{\infty}}{\eta^2N^2}\left(b_2+\frac{m_{A\ast B}^N(1^3,21^2)}{\eta^2}\right).$$
By the second part of Lemma \ref{concentration_result_multi} (switching $A$ and $B$) with $\alpha=4$ and $\beta=4$,
$$Var(\tilde{f}_{B})\leq \frac{4b_4a_{\infty}^2}{N^2\eta^4},$$
and by the same lemma with $A$ and $B$ switched, using $\alpha=2,\, \beta=\infty$,
$$\Var(\tilde{f}_{\Upsilon})\leq \frac{4\gamma'^{2}b_2a_{\infty}^2}{N^2\eta^6}.$$
Finally, by the second part of Lemma \ref{concentration_result_multi_B} with $\alpha=\infty$ and $\beta=2$ and by the bound on the infinite norm of $\Upsilon$,
\begin{align*}
\Var(\Tr(U\Upsilon U^{*}A^{1/2}G_{M}A^{1/2})\leq& \frac{8a_{\infty}}{\eta^2N^2}\left(\Vert \Upsilon\Vert^2\Tr(A)+\frac{\Vert A\Vert_{\infty}^2\Vert \Upsilon\Vert_{\infty}^2\Tr(AUB^2U^*)}{\eta^2}\right)\\
\leq&\frac{8\gamma'^2_{A}a_{\infty}}{\eta^4N^2}\left(1+\frac{a_{\infty}^2b_2}{\eta^2}\right).
\end{align*}
Hence putting all the previous bounds together yields
\begin{align*}
\vert \delta\vert\leq& \frac{8a_\infty\sqrt{b_2+\frac{m_{A\ast B}^N(1^3,21^2)}{\eta^2}}}{\vert\mathbb{E}f_{A}(z)\vert\eta^2N^2}\Bigg(a_\infty^{3/2}/\sqrt{2}\left(\sqrt{b_4}/\eta+\gamma'_{A}\sqrt{b_2}/\eta^2\right)\\
&\hspace{5cm}+\sqrt{b_2+\frac{m_{A\ast B}^N(1^3,21^2)}{\eta^2}}+\frac{\gamma'_{A}\sqrt{1+a_{\infty}^2b_2/\eta^2}}{\eta}\Bigg)\\
\leq& \frac{8b_2a_\infty\sqrt{1+\frac{m_{A\ast B}^N(1^3,21^2)}{b_2\eta^2}}}{\vert\mathbb{E}f_{A}(z)\vert\eta^2N^2}\Bigg(\sqrt{1+\frac{m_{A\ast B}^N(1^3,21^2)}{b_2\eta^2}}+\frac{a_\infty^{3/2}\sqrt{b_4}}{\sqrt{2b_2}\eta}+\frac{\gamma'_{A}(1+2\sqrt{b_2}a_{\infty}^{3/2}/\eta)}{\eta\sqrt{b_2}}\Bigg)
\end{align*}
By Lemma \ref{lower_bounds_multi}, $\frac{1}{\vert \mathbb{E}f_{A}\vert}\leq (\vert z\vert +\Tr(A^2)+\frac{\gamma}{\eta})$.  Hence, by Lemma \ref{lower_bounds_multi},
$$\vert \delta\vert\leq \frac{\vert z\vert C_{thres,B}(\eta)}{3\eta^2N^2},$$
with 
\begin{align*}
&C_{thres,B}(\eta)= 24a_{\infty}b_2\sqrt{1+\frac{m_{A\ast B}^N(1^3,21^2))}{b_2\eta^2}}\cdot \left(1+\frac{a_2}{\eta}+\frac{\tilde{\sigma}_A^2+a_2\sigma_B^2}{(1-N^{-2})\eta^2}\right)
\Bigg(\sqrt{1+\frac{m_{A\ast B}^N(1^3,21^2)}{b_2\eta^2}}\\&\hspace{4cm}+\frac{a_\infty^{3/2}\sqrt{b_4}}{\sqrt{2b_2}\eta}+(1+2\sqrt{b_2}a_{\infty}^{3/2}/\eta)\frac{k_3(A)+\sigma_A^2(b_2-\sigma_A^2)+\frac{b_2(10+4a_2+5a_3)}{N}}{(1-N^{-2})^2(1-4N^{-2})^2\eta\sqrt{b_2}}\Bigg).
\end{align*}
Then, when
$$N^2\geq \frac{C_{thres,B}(\eta)\vert z\vert }{\eta^3},$$
we have $\vert \delta\vert\leq \eta/3$, which yields $\Im \omega_B\geq \frac{2\eta}{3}$ and
$$\Vert G_{B}(\omega_B)\Vert_{\infty}\leq\frac{3}{2\eta}.$$
\end{proof}
\begin{proof}[Proof of Proposition \ref{bound_subord_B}.]
By \eqref{subor_karg_multiplication_B},
$$\tilde{m}_{M}(z)=\tilde{m}_{B}(\omega_B)+\frac{z}{\mathbb{E}f_{A}(z)}\mathbb{E}\Tr(BG_{B}(\omega_B)\Delta_B).$$
Let us bound the error term by first rewriting it as 
\begin{align*}
\frac{z}{\mathbb{E}f_{A}(z)}\Tr(BG_{B}(\omega_B)\Delta_B)=&\frac{z}{\omega_B\mathbb{E}f_{A}(z)}\Tr\left((B+B^2G_{B}(\omega_B))\Delta_{B}\right)\\
=&\frac{1}{\mathbb{E}zm_{M}(z)}\Big(\mathbb{E}((zf_{A}-z\mathbb{E}f_{A})(\tilde{f}_{B+B^2G_{B}(\omega_B)}-\mathbb{E}\tilde{f}_{B+B^2G_{B}(\omega_B)}))\\
&+\mathbb{E}\big((zm_{z}-z\mathbb{E}m_{z})(\Tr((B+B^2G_{B}(\omega_B)U^{*}A^{1/2}G_{M}A^{1/2}U)\\
&\hspace{4cm}-\mathbb{E}\Tr((B+B^2G_{B}(\omega_B)U^{*}A^{1/2}G_{M}A^{1/2}U\big)\Big),\\
\end{align*}
where we used the definition of $\omega_B(z)$ and wrote $\tilde{f}_{T}$ for $\Tr(TG_{M'})$. Thus, by Cauchy-Schwartz inequality,
\begin{align*}
&\vert \tilde{m}_{M}(z)-\tilde{m}_{B}(\omega_B)\vert\\
\leq& \frac{1}{\vert z\mathbb{E}m_{M}(z)\vert}\Big(\sqrt{\Var(zf_{A})\Var(\tilde{f}_{B})}+\sqrt{\Var(zf_{A})\Var(\tilde{f}_{B^2G_{B}(\omega_B)})}\\
&\hspace{2cm}+\sqrt{\Var(zm_{M}(z))\Var(\tilde{m}_{M}(z))}+\sqrt{\Var(zm_{M}(z))\Var(\Tr(B^2G_{B}(\omega_B)U^{*}A^{1/2}G_{M}A^{1/2}U))}\Big).
\end{align*}

By Lemma \ref{concentration_result_multi_z}, 
$$\Var(zf_{A})\leq \frac{8a_{\infty}^3}{\eta^2N^2}\left(b_2+\frac{m_{A\ast B}^N(1^3,21^2)}{\eta^2}\right),$$
and
$$\Var(\tilde{m}_{M}(z))=\Var(zm_{z})\leq \frac{8a_{\infty}}{\eta^2N^2}\left(b_2+\frac{m_{A\ast B}^N(1^3,21^2)}{\eta^2}\right).$$
By the first part of Lemma \ref{concentration_result_multi} with $A$ and $B$ switched and with $\alpha=\beta=4$, 
\begin{align*}
Var(\tilde{f}_{B})\leq &\frac{4b_4a_{\infty}^2}{N^2\eta^4},
\end{align*}
and by the first part of Lemma \ref{concentration_result_multi} with $A$ and $B$ switched and with $\alpha=6,\beta=3$,
\begin{align*}
\Var(\tilde{f}_{B^2G_{B}(\omega_{B})})\leq \frac{4\Tr(\vert B^2G_{B}(\omega_B)\vert^3)^{2/3}\Tr(B^6)^{1/3}\Vert A\Vert_{\infty}^2}{N^2\eta^4}\leq \frac{9b_6a^2_{\infty}}{N^2\eta^6},
\end{align*}
where we used the hypothesis on $N$ and Lemma \ref{minoration_omega_B_multi} to get $\Vert G_{B}(\omega_B)\Vert_{\infty}\leq \frac{3}{2\eta}$. Finally, by the second part of  Lemma \ref{concentration_result_multi_B}, with $\alpha=1/3$ and $\beta=1/6$, and using the fact that $\Vert G_{B}(\omega_B)\Vert_{\infty}\leq \frac{3}{2\eta}$,
\begin{align*}
\Var&(\Tr(A^{1/2}U( B^2G_{B}(\omega_B)) U^{*}A^{1/2}G_{M})\\
\leq &\frac{18\Vert A\Vert_{\infty}}{N^2\eta^4}\left(\mathbb{E}\Tr(A^{1/2}UB^4U^{*} A^{1/2})+\frac{(\mathbb{E}\Tr( (A^{1/2}UB^2U^{*} A^{1/2})^3))^{2/3}(\mathbb{E}\Tr((A^{1/2}UB^2U^{*}A^{1/2})^{3}))^{1/3}}{\eta^2}\right)\\
\leq &\frac{18\Vert A\Vert_{\infty}}{N^2\eta^4}\left(\mathbb{E}\Tr(A^{1/2}UB^4U^{*} A^{1/2})+\frac{\mathbb{E}\Tr( (AUB^2U^{*})^3)}{\eta^2}\right)\\
\leq &\frac{18a_{\infty}}{N^2\eta^4}\left(b_4+\frac{m_{A\ast B^2}^N(1^3,1^3)}{\eta^2}\right).
\end{align*}
Then, putting all latter bounds together yields
\begin{align*}
\vert \tilde{m}_{M}&(z)-\tilde{m}_{B}(\omega_B)\vert\leq \frac{1}{\vert z\mathbb{E}m_{M}(z)\vert}\Bigg(\frac{4\sqrt{2}a_{\infty}^{5/2}}{\eta^3N^2}\sqrt{b_2+\frac{m_{A\ast B}^N(1^3,21^2)}{\eta^2}}\left(\sqrt{b_4}+\frac{3\sqrt{b_6}}{2\eta}\right)\\
+&\frac{8a_{\infty}}{\eta^2N^2}\Bigg(b_2+\frac{m_{A\ast B}^N(1^3,21^2)}{\eta^2}+\frac{3}{2\eta}\sqrt{b_2+\frac{m_{A\ast B}^N(1^3,21^2)}{\eta^2}}\sqrt{b_4+\frac{m_{A\ast B}^N(1^3,2^3)}{\eta^2}}\Bigg)\Bigg)\\
\leq&\frac{4\sqrt{2}a_{\infty}\sqrt{b_2+\frac{m_{A\ast B}^N(1^3,21^2)}{\eta^2}}}{\vert z\mathbb{E}m_{M}(z)\vert\eta^2N^2}\Bigg(\frac{a_\infty^{3/2}}{\eta}\left(\sqrt{b_4}+\frac{3\sqrt{b_6}}{2\eta}\right)\\
&\hspace{4cm}+\sqrt{2}\sqrt{b_2+\frac{m_{A\ast B}^N(1^3,21^2)}{\eta^2}}+\frac{3}{\sqrt{2}\eta}\sqrt{b_4+\frac{m_{A\ast B}^N(1^3,2^3)}{\eta^2}}\Bigg).
\end{align*}
Since, by \eqref{F_mu_expression}, $\frac{-1}{\mathbb{E}m_{M}}=z-\mathbb{E}\Tr(M))+\left(\mathbb{E}(\Tr(M^2))-\mathbb{E}(\Tr M)^2\right)m_{\rho}(z)$ for some probability measure $\rho$, and by Lemma \ref{trace_Weingarten} $\mathbb{E}(\Tr M)=\Tr(A)\Tr(B)=1$ and $\mathbb{E}(\Tr M^2)\leq \frac{1}{1-N^{-2}}(a_2b_1^2+a_1^2b_2-a_1^2b_1^2)$,
$$\frac{1}{\vert z\mathbb{E}m_{M}\vert}\leq 1+\frac{1+\frac{\sigma_A^2+\sigma_B^2}{(1-N^{-2})\eta}}{\vert z\vert}\leq 1+\frac{1}{\eta}+\frac{\sigma_A^2+\sigma_B^2}{(1-N^{-2})\eta^2}.$$
Hence, 
$$\vert \tilde{m}_{M}(z)-\tilde{m}_{B}(\omega_B)\vert\leq\frac{C_{bound,B}}{N^2},$$
with 
\begin{align*}
C_{bound,B}&=\frac{4\sqrt{2}a_{\infty}b_2}{\eta^2}\left(1+\frac{1}{\eta}+\frac{\sigma_A^2+\sigma_B^2}{(1-N^{-2})\eta^2}\right)\cdot\sqrt{1+\frac{m_{A\ast B}^N(1^3,21^2)}{b_2\eta^2}}\\
&\cdot\Bigg(\frac{a_\infty^{3/2}}{\eta}\left(\sqrt{\frac{b_4}{b_2}}+\sqrt{\frac{9b_6}{4b_2\eta^2}}\right)+\sqrt{2}\sqrt{1+\frac{m_{A\ast B}^N(1^3,21^2)}{b_2\eta^2}}+\frac{3}{\sqrt{2b_2}\eta}\sqrt{b_4+\frac{m_{A\ast B}^N(1^3,2^3)}{\eta^2}}\Bigg),
\end{align*}
and writing $\eta=\kappa\tilde{\sigma}_1$ yields the the second part of the statement. The lower bound on the imaginary part of $\omega_B$ is directly given by Lemma \ref{minoration_omega_B_multi}. \end{proof}

\section{Stability results for the deconvolution}\label{Section:pointwise_concentration}

In this section, we apply the concentration results from the previous section to get the mean squared error of our estimator $\widehat{\mathcal{C}_{B}}[\eta]$. We need to take into account the error term from the fluctuations of $m_{H}$ or $m_{M}$ around their average and fluctuations from $\mu_A$ around $\mu_1$ (recall the definition of $\mu_1$ from Condition \ref{concentration_noise}). To this end, introduce in the additive case the (random) error terms
$$\delta_H(z)=m_{H}(\omega_3(z))-\mathbb{E}m_{H}(\omega_3(z)), \; \delta_A(z)=m_{A}(\omega_A\circ\omega_3(z))-m_{\mu_1}(\omega_A\circ\omega_3(z)),$$
where $\omega_3(z)$ is given by Theorem \ref{decon_additif_first_step} and $\omega_A$ by \eqref{matrix_subordination_definition}, and in the multiplicative case
$$\tilde{\delta}_M(z)=\tilde{m}_{M}(\omega_3(z))-\mathbb{E}\tilde{m}_{M}(\omega_3(z)), \; \tilde{\delta}_A(z)=\tilde{m}_{A}(\omega_A\circ\omega_3(z))-\tilde{m}_{\mu_1}(\omega_A\circ\omega_3(z)),$$
where $\omega_3$ is given in Theorem \ref{subordination_deconvolution_multi} and $\omega_A$ in \eqref{definition_subor_multi}. The dependence of the latter functions in $z$ will often be dropped in the sequel. 

Stability results in both the additive and multiplicative cases are obtained using the coercive property of the reciprocal Cauchy transform, which is summarized in the next lemma.
\begin{lemma}\label{coerci_Fmu}
Let $\mu$ be a probability measure with variance $\sigma^2$. For all $z,z'\in \mathbb{C}^+$,
$$F_{\mu}(z)-F_{\mu}(z')=(z-z')(1+\tau_{\mu}(z,z')),$$
with $\vert\tau_\mu(z,z')\vert\leq \frac{\sigma^2}{\Im z\Im z'}$.
\end{lemma}
\begin{proof}
By \eqref{F_mu_expression},
$$F_{\mu}(z)=z-\mu(1)+\sigma^2 m_{\rho}(z),$$
with $\rho$ a probability measure on $\mathbb{R}$. Then, for $z,z'\in \mathbb{C}_{\sigma}$,
$$F_{\mu}(z)-F_{\mu}(z')=z-z'+\sigma^2(m_{\rho}(z)-m_{\rho}(z')).$$
Moreover, 
\begin{align*}
 m_{\rho}(z)-m_{\rho}(z')=\int_{\mathbb{R}}\frac{1}{t-z}d\rho(t)-\int_{\mathbb{R}}\frac{1}{t-z'}d\rho(t)=&(z-z') \int_{\mathbb{R}}\frac{1}{(t-z)(t-z')}d\rho(t),
\end{align*}
which implies the first statement of the lemma. The second statement is given by the inequality 
$\left\vert\int_{\mathbb{R}}\frac{1}{(t-z)(t-z')}d\rho(t)\right\vert\leq \frac{1}{\Im z\Im z'}$.
\end{proof}
Following a similar pattern as for previous notations, we simply write $\tau_{X}$ instead of $\tau_{\mu_X}$ for $X$ self-adjoint in $\mathcal{M}_{N}(\mathbb{C})$.
\subsection{Additive case}

For $z\in \mathbb{C}_{2\sqrt{2}\sigma_1}$, let $(\omega_1,\omega_3)\in\mathbb{C}^+\times\mathbb{C}^+$ be the solution of the system
\begin{equation}\label{fixed_point_additif}
\left\lbrace\begin{matrix}
\omega_1+z=\omega_3+F_{H}(\omega_3)\\
\omega_1+z=\omega_3+F_{\mu_1}(\omega_1)
\end{matrix}\right.,
\end{equation}
which, by Theorem \ref{decon_additif_first_step}, exists and satisfies
$$\Im \omega_3\geq \frac{3\Im(z)}{4},\,\Im \omega_1\geq \frac{\eta}{2},$$
with $\eta=\Im z$. Let $\omega_A,\omega_B$ be the subordination functions introduced in \eqref{matrix_subordination_definition} for $\omega_3$.
 \begin{lemma}\label{expression_difference_additif}
 For $z\in\mathbb{C}^+$ such that $\Im z\geq 2\sqrt{2}\sigma_1$,
\begin{align*} 
 (m_{B}(z)-m_{H}(\omega_3))=&\frac{Lm_{B}(z)}{m_{\mu_1}(\omega_A)}(m_{A}(\omega_A)-\mathbb{E}m_{H}(\omega_3))+m_{B}(\omega_B)-\mathbb{E}m_{H}(\omega_3)\\&\hspace{2cm}-\frac{Lm_{B}(z)}{m_{\mu_1}(\omega_A)}\delta_A+\left(\frac{m_{B}(z)}{m_{H}(\omega_3)}L\tau_{\mu_1}(\omega_1,\omega_A)-1\right)\delta_H,
\end{align*}
 with 
 $$L=\left(1+\frac{m_{B}(\omega_B)-\mathbb{E}m_{H}(\omega_3)}{\mathbb{E}m_{H}(\omega_3)}\right)\frac{1+\tau_B(\omega_B,z)}{1+\tau_{\mu_1}(\omega_1,\omega_A)}.$$
 \end{lemma}
 \begin{proof}
 Note that 
 $$m_{B}(z)-m_{H}(\omega_3)=m_B(z)-m_{B}(\omega_B)+m_{B}(\omega_B)-\mathbb{E}m_{H}(\omega_3)+\mathbb{E}m_{H}(\omega_3)-m_{H}(\omega_3).$$
First,
\begin{align}
m_B(z)-m_{B}(\omega_B)=-\frac{1}{F_{B}(z)}+\frac{1}{F_{B}(\omega_B)}=&\left(F_{B}(z)-F_{B}(\omega_B)\right)m_{B}(z)m_{B}(\omega_B)\nonumber\\
=&(z-\omega_B)(1+\tau_{B}(\omega_B,z))m_{B}(z)m_{B}(\omega_B),\label{difference_mub_zomegaB}
\end{align}
where we used Lemma \ref{coerci_Fmu} in the last inequality. Then, using the relation satisfied by $\omega_B$ and $z$ yields
\begin{align*}
\omega_B-z=&\omega_3+F_{\bar{H}}(\omega_3)-\omega_A-\omega_3-F_{H}(\omega_3)+\omega_1\\
=&\omega_1-\omega_A+F_{\bar{H}}(\omega_3)-F_{H}(\omega_3),
\end{align*}
where $F_{\bar{H}}=\frac{-1}{\mathbb{E}m_{H}}$. Then, by Lemma \ref{coerci_Fmu} and the relation $F_{\mu_1}(\omega_1)=F_{H}(\omega_3)$, with $\tau_1=\tau_{\mu_1}(\omega_1,\omega_A)$,
\begin{align*}
\omega_B-z=&\frac{F_{\mu_1}(\omega_1)-F_{\mu_1}(\omega_A)}{1+\tau_{\mu_1}(\omega_1,\omega_A)}+F_{\bar{H}}(\omega_3)-F_{H}(\omega_3)\\
=&\frac{F_{H}(\omega_3)-F_{\bar{H}}(\omega_3)+F_{\bar{H}}(\omega_3)-F_{\mu_1}(\omega_A)}{1+\tau_1}+F_{\bar{H}}(\omega_3)-F_{H}(\omega_3)\\
=&(F_{\bar{H}}(\omega_3)-F_{H}(\omega_3))\frac{\tau_1}{1+\tau_1}+\frac{F_{\bar{H}}(\omega_3)-F_{\mu_1}(\omega_A)}{1+\tau_1}\\
=&\frac{F_{\bar{H}}(\omega_3)F_{H}(\omega_3)\tau_1}{1+\tau_1}(\mathbb{E}m_{H}(\omega_3)-m_{H}(\omega_3))+\frac{F_{\mu_1}(\omega_A)F_{\bar{H}}(\omega_3)}{1+\tau_1}(\mathbb{E}m_{H}(\omega_3)-m_{\mu_1}(\omega_A)).
\end{align*}
Write temporarily  $\epsilon_B=\frac{m_{B}(\omega_B)-\mathbb{E}m_{H}(\omega_3)}{\mathbb{E}m_{H}(\omega_3)}$, $\epsilon_{A}=\frac{m_{\mu_1}(\omega_A)-\mathbb{E}m_{H}(\omega_3)}{m_{\mu_1}(\omega_A)}$. Hence, putting the latter relation in \eqref{difference_mub_zomegaB} yields
$$m_B(z)-m_{B}(\omega_B)=m_{B}(z)\left( L\tau_1\frac{\delta_{H}(\omega_3)}{m_{H}(\omega_3)}+L\epsilon_A\right),$$
with 
\begin{equation}\label{eq:definition_L}
L=\frac{m_{B}(\omega_B)}{\mathbb{E}m_{H}(\omega_3)}\frac{1+\tau_2}{1+\tau_1}=(1+\epsilon_B)\frac{1+\tau_2}{1+\tau_1},
\end{equation}
where $\tau_2=\tau_{B}(\omega_B,z)$. Hence, using the first relation of the proof gives then 
\begin{align*}
(m_{B}(z)-m_{H}(\omega_3))=&\frac{Lm_{B}(z)}{m_{\mu_1}(\omega_A)}(m_{\mu_A}(\omega_A)-\mathbb{E}m_{H}(\omega_3))+m_{B}(\omega_B)-\mathbb{E}m_{H}(\omega_3)\\-&\frac{L m_{B}(z)}{m_{\mu_1}(\omega_A)}\delta_A+\left(\frac{m_{B}(z)}{m_{H}(\omega_3)}L\tau_1-1\right)\delta_H.
\end{align*}
\end{proof}

From  the latter lemma we express the distance between $m_{B}(z)$ and $m_{\mu_H}(\omega_3)$ in terms of the fluctuations $\delta_{H}$ and $\delta_A$. Recall that we set $\Im z=\eta$ and $\Im \omega_3=\xi\sigma_1$.
\begin{proposition}\label{stability_decon_additif}
Suppose that $N^2\geq \frac{\max(C_{thres,A}(\xi\sigma_1),C_{thres,B}(\xi\sigma_1))}{\xi^3\sigma_1^3}$. Then
\begin{align*}
\vert m_{B}(z)-m_{H}(\omega_3)\vert\leq & \frac{C_{1}(\eta/\sigma_1)}{\vert z\vert N^2}
+\frac{C_{2}(\eta/\sigma_1)}{\vert z\vert}\vert \omega_A\delta_A\vert+\frac{C_{3}(\eta/\sigma_1)}{\vert z\vert}\vert \omega_3\delta_H\vert,
\end{align*}
where $C_{1}(\eta/\sigma_1),C_{2}(\eta/\sigma_1)$ and $C_{3}(\eta/\sigma_1)$ are respectively given in \eqref{first_constant_additif}, \eqref{second_constant_additif} and \eqref{third_constant_additif}.
\end{proposition}
\begin{proof}
By Proposition \ref{convergence_subord_B}, for 
$N\geq \sqrt{\frac{\max(C_{thres,A}(\xi\sigma_1),C_{thres,B}(\xi\sigma_1)}{\xi^3\sigma_1^3}}$, with $C_{thres,A}(\xi\sigma_1)$ given in Proposition \ref{minoration_omega}, then $\Im \omega_A,\Im \omega_B\geq 2\xi\sigma_1/3$ and 
$$\vert \mathbb{E}m_{H}(\omega_3)-m_{\mu_1}(\omega_{A}(z))\vert\leq \frac{C_{bound,A}(\xi)}{\vert \omega_3\vert N^2},$$
 and
$$\vert \mathbb{E}m_{H}(\omega_3)-m_{B}(\omega_{B}(z))\vert\leq \frac{C_{bound,B}(\xi)}{\vert \omega_3\vert N^2},$$
with $C_{bound,A}(\xi)$ and $C_{bound,B}(\xi)$ given in Proposition \ref{convergence_subord_B}. Hence, in particular, by the definition of $L$ from \eqref{eq:definition_L}, we get
$$\vert L\vert\leq \left(1+\frac{C_{bound,B}(\xi)}{\mathbb{E}m_{H}(\omega_3)\vert \omega_3\vert\xi^2\sigma_1^2N^2}\right)\left\vert\frac{1+\tau_B(\omega_B,z)}{1+\tau_{\mu_1}(\omega_1,\omega_A)}\right\vert.$$
Moreover, by Proposition \ref{minoration_omega} and Theorem \ref{decon_additif_first_step}, $\Im \omega_A,\Im \omega_B\geq 2\Im\omega_3/3\geq \eta/2$, and $\Im \omega_1\geq \eta/2$, which yields
$$\vert \tau_{\mu_1}(\omega_1,\omega_A)\vert=\left\vert\int_{\mathbb{R}}\frac{\sigma_1^2d\rho(t)}{(\omega_1-t)(\omega_A-t)}\right\vert\leq \frac{4\sigma_1^2}{\eta^2}, \;\vert \tau_B(\omega_B,z)\vert=\left\vert \int_{\mathbb{R}}\frac{\sigma_B^2d\rho'(t)}{(\omega_B-t)(z-t)}\right\vert\leq\frac{2\sigma_B^2}{\eta^2}.$$
Hence, since by \eqref{lower_bound_expectation} we have $\left\vert\frac{1}{\omega_3\mathbb{E}m_{H}(\omega_3)}\right\vert\leq 1+\frac{a_2+b_2}{\Im\omega_3^2}$,
$$L\leq \left(1+\frac{C_{bound,B}(\xi)\left(1+\frac{a_2+b_2}{\xi^2\sigma_1^2}\right)}{N^2}\right)\frac{1+2\sigma_B^2/\eta^2}{1-4\sigma_1^2/\eta^2}:=K(N).$$
Therefore, by Lemma \ref{expression_difference_additif}
\begin{align*}
\vert m_{B}(z)-m_{H}(\omega_3)\vert\leq & \frac{C_{1}(\eta/\sigma_1)}{\vert z\vert N^2}
+\frac{C_{2}(\eta/\sigma_1)}{\vert z\vert}\vert \omega_A\delta_A\vert+\frac{C_{3}(\eta/\sigma_1)}{\vert z\vert}\vert\omega_3\delta_H\vert,
\end{align*}
with, recalling that $\Im \omega_A\geq \eta/2$, using that $\vert \omega_A-\omega_3\vert\leq \frac{\sigma_{B}^2}{\xi\sigma_1}+\xi\sigma_1/3$ by Proposition \ref{minoration_omega} and $\vert z-\omega_3\vert\leq 1+\frac{2\sigma_1^2}{\eta}$ by Theorem \ref{decon_additif_first_step},
\begin{align}
&C_{1}(\eta/\sigma_1)=K(N)\cdot\frac{\vert F_{\mu_1}(\omega_A)\vert }{\vert \omega_A\vert}\cdot\frac{\vert \omega_A\vert}{\vert \omega_3\vert}\cdot\vert \omega_3\vert\cdot \vert m_{A}(\omega_A)-\mathbb{E}m_{H}(\omega_3)\vert\cdot\left\vert m_{B}(z)z\right\vert\nonumber\\
&\hspace{8cm}+\left\vert\frac{z}{\omega_3}\right\vert\cdot\vert\omega_3\vert\cdot\left\vert m_{B}(\omega_B)-\mathbb{E}m_{H}(\omega_3)\right\vert\nonumber\\
\leq&\left(1+\frac{2\sigma_1^2}{\eta^2}\right)C_{bound,B}(\xi)+ \left(1+\frac{C_{bound,B}(\xi)\left(1+\frac{a_2+b_2}{\xi^2\sigma_1^2}\right)}{N^2}\right)\cdot\frac{1+2\sigma_B^2/\eta^2}{1-4\sigma_1^2/\eta^2}\cdot\left(1+\frac{4\sigma_1^2}{\eta^2}\right)\nonumber\\
&\hspace{7cm}\cdot\left(\frac{4}{3}+\frac{\sigma_B^2}{(\xi\sigma_1)^2}\right)C_{bound,A}(\xi)\left(1+\frac{\sigma_B}{\eta}\right),\label{first_constant_additif}
\end{align}
\begin{align}
C_{2}(\eta/\sigma_1)=&K(N)\frac{\vert F_{\mu_1}(\omega_A)\vert }{\vert \omega_A\vert}\left\vert zm_B(z)\right\vert\nonumber\\
\leq& \left(1+\frac{C_{bound,B}(\xi)\left(1+\frac{a_2+b_2}{\xi^2\sigma_1^2}\right)}{N^2}\right)\cdot\frac{1+2\sigma_B^2/\eta^2}{1-4\sigma_1^2/\eta^2}\cdot\left(1+\frac{4\sigma_1^2}{\eta^2}\right)\cdot\left(1+\frac{\sigma_B}{\eta}\right),\label{second_constant_additif}
\end{align}
and 
\begin{align*}
C_{3}(\eta/\sigma_1)=&\left\vert \tau_{\mu_1}(\omega_1,\omega_A)\frac{zm_{B}(z)}{\omega_3m_{H}(\omega_3)} L-\frac{z}{\omega_3}\right\vert.
\end{align*}
Using $z-\omega_3=h_{\mu_1}(\omega_1)$ to expand the right hand side of the latter equation gives then
\begin{align*}
\tau_{\mu_1}(\omega_1,\omega_A)\frac{zm_{B}(z)}{\omega_3m_{H}(\omega_3)} L-\frac{z}{\omega_3}=-1-\frac{h_{\mu_1}(\omega_1)}{\omega_3}+L\tau_{\mu_1}(\omega_1,\omega_A)(1+\tilde{m}_{B}(z))\left(1+\frac{\sigma_H^2}{\omega_3}m_{\rho}(\omega_3)\right),
\end{align*}
and finally
\begin{align}
C_{3}(\eta/\sigma_1)\leq& 1+\frac{8\sigma_1^2}{3\eta^2}\nonumber\\
+&\left(1+\frac{C_{bound,B}(\xi)\left(1+\frac{a_2+b_2}{\xi^2\sigma_1^2}\right)}{N^2}\right)\cdot\frac{1+2\sigma_B^2/\eta^2}{1-4\sigma_1^2/\eta^2}\cdot\frac{4\sigma_1^2}{\eta^2}\cdot \left(1+\frac{\sigma_B}{\eta}\right)\cdot \left(1+\frac{16\sigma_H^2}{9\eta^2}\right).\label{third_constant_additif}
\end{align}
\end{proof}
\subsection{Multiplicative case} 
We now turn to the multiplicative case, which follows a similar pattern. We first express the difference between $\tilde{m}_{B}(z)$ and $\tilde{m}_{M}(\omega_3)$.
\begin{lemma}\label{decomposition_error_multi}
Set $\epsilon_A=\tilde{m}_{A}(\omega_A)-\mathbb{E}\tilde{m}_{M}(\omega_3)$ and $\epsilon_{B}=\tilde{m}_{B}(\omega_B)-\mathbb{E}\tilde{m}_{M}(\omega_3)$. Then
\begin{align*}
\tilde{m}_{B}(z)-\tilde{m}_{M}(\omega_3)=&\tilde{m}_{B}(z)\tilde{F}_{B}(\omega_B)\epsilon_B-L\epsilon_A+L\tilde{\delta}_A
+\left[L'-1\right]\tilde{\delta}_M,
\end{align*}
with $L=\frac{z\tilde{m}_B(z)\tilde{F}_{\mu_1}(\omega_A)(1+\tau_{\tilde{B}}(\omega_B,z))}{\omega_A(1+\tau_{\tilde{\mu}_1}(\omega_1,\omega_A))}$ and $L'=\frac{z\tilde{m}_B(z)(1+\tau_{\tilde{B}}(\omega_B,z))(F_{\mu_1}(\omega_1)-\tilde{F}_{\mu_1}(\omega_1))}{\omega_A(1+\tau_{\mu_1}(\omega_1,\omega_A))}$.
\end{lemma}
\begin{proof}
We have
$$\tilde{m}_{B}(z)-\tilde{m}_{M}(\omega_3)=\tilde{m}_{B}(z)-\tilde{m}_{B}(\omega_B)+\tilde{m}_{B}(\omega_B)-\mathbb{E}\tilde{m}_{M}(\omega_3)-\tilde{\delta}_M,$$
and, setting $\epsilon_B=\tilde{m}_{B}(\omega_B)-\mathbb{E}\tilde{m}_{M}(\omega_3)$,
\begin{align}
\tilde{m}_{B}(z)-\tilde{m}_{B}(\omega_B)
=&\left(\tilde{F}_{B}(z)-\tilde{F}_{B}(\omega_B)\right)\mathbb{E}\tilde{m}_{M}(\omega_3)\tilde{m}_B(z)+\left(\tilde{F}_{B}(z)-\tilde{F}_{B}(\omega_B)\right)\epsilon_B\tilde{m}_B(z)\nonumber\\
=&\left(1+\tau_{\widetilde{\mu_{B}}}(\omega_B,z)\right)\mathbb{E}\tilde{m}_{M}(\omega_3)\tilde{m}_B(z)(z-\omega_B)+\left(\frac{\tilde{m}_{B}(z)}{\tilde{m}_{B}(\omega_B)}-1\right)\epsilon_B.\label{difference_omegaB_z_multi}
\end{align}
By Theorem \ref{subordination_deconvolution_multi}, $\omega_1z=\omega_3\hat{F}_{M}(\omega_3)$, and by \eqref{matrix_subordination_equation_multi}, $\omega_A\omega_B=\omega_3\hat{F}_{\bar{M}}(\omega_3)$, with $\hat{F}_{\bar{M}}$ denoting $1+\tilde{F}_{\bar{M}}$ and $\tilde{F}_{\bar{M}}=\frac{-1}{\mathbb{E}\tilde{m}_{M}}$. Hence,
\begin{align*}
\omega_B-z=&\omega_3\left(\frac{\hat{F}_{\bar{M}}(\omega_3)}{\omega_A}-\frac{\hat{F}_{M}(\omega_3)}{\omega_1}\right)\\
=&\omega_3\left(\frac{\hat{F}_{\bar{M}}(\omega_3)-\hat{F}_{M}(\omega_3)}{\omega_A}+\frac{\hat{F}_{M}(\omega_3)(\omega_1-\omega_A)}{\omega_1\omega_A}\right).
\end{align*}
Then, since
\begin{align*}
\omega_1-\omega_A=&\frac{1}{1+\tau_{\mu_1}(\omega_1,\omega_A)}( \tilde{F}_{\mu_1}(\omega_1)-\tilde{F}_{\mu_1}(\omega_A))\\
=&\frac{1}{1+\tau_{\tilde{\mu}_1}(\omega_1,\omega_A)}(\tilde{F}_{\mu_1}(\omega_1)-\tilde{F}_{\bar{M}}(\omega_3)+\tilde{F}_{\bar{M}}(\omega_3)-\tilde{F}_{\mu_1}(\omega_A)),
\end{align*}
we get, using again the relation $\omega_1z=\omega_3\hat{F}_{M}(\omega_3)$ and $\tilde{F}_{\mu_1}(\omega_1)=\tilde{F}_{M}(\omega_3)$,
\begin{align*}
\omega_B-z=&\omega_3\frac{\hat{F}_{\bar{M}}(\omega_3)-\hat{F}_{M}(\omega_3)}{\omega_A}+z\frac{\tilde{F}_{\mu_1}(\omega_1)-\tilde{F}_{\bar{M}}(\omega_3)+\tilde{F}_{\bar{M}}(\omega_3)-\tilde{F}_{\mu_1}(\omega_A)}{\omega_A(1+\tau_{\tilde{\mu}_1}(\omega_1,\omega_A))}\\
=&-\omega_{3}\tilde{F}_{M}(\omega_3)\tilde{F}_{\bar{M}}(\omega_3)\frac{\tilde{\delta}_M}{\omega_A}+z\tilde{F}_{M}(\omega_3)\tilde{F}_{\bar{M}}(\omega_3)\frac{\tilde{\delta}_{M}}{\omega_A(1+\tau_{\tilde{\mu}_1}(\omega_1,\omega_A))}\\
&\hspace{6cm}-z\tilde{F}_{\bar{M}}(\omega_3)\tilde{F}_{\mu_1}(\omega_A)\frac{\tilde{m}_{\mu_1}(\omega_A)-\tilde{m}_{\bar{M}}(\omega_3)}{\omega_A(1+\tau_{\tilde{\mu}_1}(\omega_1,\omega_A))}\\
=&\tilde{F}_{\bar{M}}(\omega_3)\tilde{F}_{M}(\omega_3)\frac{z-\omega_3}{\omega_A(1+\tau_{\tilde{\mu}_1}(\omega_1,\omega_A))}\tilde{\delta}_{M}(z)+z\frac{\tilde{F}_{\bar{M}}(\omega_3)\tilde{F}_{\mu_1}(\omega_A)}{\omega_A(1+\tau_{\tilde{\mu}_1}(\omega_1,\omega_A))}(\tilde{\delta}_A-\epsilon_A),
\end{align*}
with $\epsilon_A=\tilde{m}_{A}(\omega_A)-\mathbb{E}\tilde{m}_{M}(\omega_3)$. Putting the latter equality in \eqref{difference_omegaB_z_multi} yields then
\begin{align*}
\tilde{m}_{B}(z)-\tilde{m}_{B}(\omega_B)
=&(1+\tau_{\widetilde{\mu_{B}}}(\omega_B,z))\tilde{m}_B(z)\Bigg[\tilde{F}_{M}(\omega_3)\frac{\omega_3-z}{\omega_A(1+\tau_{\tilde{\mu}_1}(\omega_1,\omega_A))}\tilde{\delta}_{M}(z)\\&+z\frac{\tilde{F}_{\mu_1}(\omega_A)}{\omega_A(1+\tau_{\tilde{\mu}_1}(\omega_1,\omega_A))}(\tilde{\delta}_A-\epsilon_A)\Bigg]+\left(\frac{\tilde{m}_{B}(z)}{\tilde{m}_{B}(\omega_B)}-1\right)\epsilon_B.
\end{align*}
Since $\omega_3=-zh_{\mu_1}(\omega_1)$ (see Theorem \ref{subordination_deconvolution_multi}) and $\tilde{F}_{M}(\omega_3)=\tilde{F}_{\mu_1}(\omega_1)$, we can further simplify the above expression since
\begin{align*}
(\omega_3-z)\tilde{F}_{M}(\omega_3)=z(-h_{\mu_1}(\omega_1)-1)\tilde{F}_{\mu_1}(\omega_1)
=&z\left[-\left(\frac{-1}{m_{\mu_1}(\omega_1)}-\omega_1\right)\frac{-1}{1+\omega_1m_{\mu_1}(\omega_1)}-\tilde{F}_{\mu_1}(\omega_1)\right]\\
=&z\left[\frac{-1}{m_{\mu_1}(\omega_1)}-\tilde{F}_{\mu_1}(\omega_1)\right]=z(F_{\mu_1}(\omega_1)-\tilde{F}_{\mu_1}(\omega_1)),
\end{align*}
yielding
\begin{align*}
\tilde{m}_{B}(z)-\tilde{m}_{B}(\omega_B)
=&\frac{z\tilde{m}_B(z)(1+\tau_{\widetilde{\mu_{B}}}(\omega_B,z))}{\omega_A(1+\tau_{\tilde{\mu}_1}(\omega_1,\omega_A))}\Bigg[(F_{\mu_1}(\omega_1)-\tilde{F}_{\mu_1}(\omega_1))\tilde{\delta}_{M}(z)+\tilde{F}_{\mu_1}(\omega_A)(\tilde{\delta}_A-\epsilon_A)\Bigg]\\&\hspace{5cm}+\left(\frac{\tilde{m}_{B}(z)}{\tilde{m}_{B}(\omega_B)}-1\right)\epsilon_B.
\end{align*}
Hence,
\begin{align*}
&\tilde{m}_{B}(z)-\tilde{m}_{M}(\omega_3)=\frac{\tilde{m}_{B}(z)}{\tilde{m}_{B}(\omega_B)}\epsilon_B-\frac{z\tilde{m}_B(z)\tilde{F}_{\mu_1}(\omega_A)(1+\tau_{\widetilde{\mu_{B}}}(\omega_B,z))}{\omega_A(1+\tau_{\tilde{\mu}_1}(\omega_1,\omega_A))}\epsilon_A\\
+&\frac{z\tilde{m}_B(z)\tilde{F}_{\mu_1}(\omega_A)(1+\tau_{\widetilde{\mu_{B}}}(\omega_B,z))}{\omega_A(1+\tau_{\tilde{\mu}_1}(\omega_1,\omega_A))}\delta_A
+\left[\frac{z\tilde{m}_B(z)(1+\tau_{\widetilde{\mu_{B}}}(\omega_B,z))(F_{\mu_1}(\omega_1)-\tilde{F}_{\mu_1}(\omega_1))}{\omega_A(1+\tau_{\tilde{\mu}_1}(\omega_1,\omega_A))}-1\right]\tilde{\delta}_M.
\end{align*}
\end{proof}
Estimating the different contributions from latter lemma yields the following control on the deconvolution procedure in the multiplicative case.
\begin{proposition}\label{stability_decon_multi}
Let $z\in\mathbb{C}^+$ satisfy $\Im(z)=\kappa\tilde{\sigma}_1$ with $\kappa> g(\xi_0)$, and consider the solution $(\omega_1,\omega_3)\in\mathbb{C}^+\times \mathbb{C}_{\xi_0\tilde{\sigma}_1}$ of the system of equations
\begin{equation}\label{equation_deconvolution}
\begin{matrix}
\omega_{1}z=\omega_3\hat{F}_{\mu_M}(\omega_{3})\\
\omega_{1}z=\omega_3\hat{F}_{\mu_1}(\omega_{1})
\end{matrix},
\end{equation}
which exists by Theorem \ref{subordination_deconvolution_multi}. Then, writing $\xi=g^{-1}(\kappa)$, for 
$$N^2\geq\frac{\vert \omega_3\vert}{\xi^3\tilde{\sigma}_1^3}\max\Big(C_{thres,A}(\xi\tilde{\sigma}_1),C_{thres,B}(\xi\tilde{\sigma}_1)\Big),$$
we have 
\begin{align*}
\left\vert \tilde{m}_{M}(\omega_3)-\tilde{m}_{B}(z)\right\vert \leq&\frac{C_{1}(\kappa)}{N^2}+C_{2}(\kappa)\tilde{\delta}_{A}+C_{3}(\kappa)\tilde{\delta}_M,
\end{align*}
with $C_{1}(\kappa),\,C_{2}(\kappa),\, C_{3}(\kappa)$  respectively given in \eqref{first_constant_multi}, \eqref{second_constant_multi} and \eqref{third_constant_multi}.
\end{proposition}
\begin{proof}
We have to bound the different contributions from Lemma \ref{decomposition_error_multi}. Suppose that 
$$N^2\geq\frac{\vert \omega_3\vert}{\xi^3\tilde{\sigma}_1^3}\max\Big(C_{thres,A}(\xi\tilde{\sigma}_1),C_{thres,B}(\xi\tilde{\sigma}_1)\Big).$$
Then, since $\Im \omega_3\geq \xi\tilde{\sigma}_1$ by Lemma \ref{upper_bound_z_multi} and $C_{thres,A},\,C_{thres,B}$ are decreasing functions, $\Im\omega_A\geq 2\Im \omega_3/3$ by Proposition \ref{minoration_omega_multiplicative}. Hence, 
$$\left\vert\frac{\tilde{F}_{\mu_1}(\omega_A)}{\omega_A}\right\vert\leq 1+\frac{\mu_1(2)}{\Im \omega_A}+\frac{\tilde{\sigma}_1^2}{(\Im \omega_A)^2}\leq 1+\frac{3\mu_1(2)}{2\xi\tilde{\sigma_1}}+\frac{9}{4\xi^2}.$$
Moreover, $\Im \omega_1\geq k(\xi)\tilde{\sigma}_1$, thus $\vert \tau_{\tilde{\mu}_1}(\omega_1,\omega_A)\vert\leq \frac{\tilde{\sigma}_1^2}{\Im\omega_1\Im \omega_A}\leq \frac{3}{2\xi k(\xi)}$. Similarly, $\Im \omega_B\geq 2\Im \omega_3/3$ by Lemma \ref{minoration_omega_B_multi}, thus $\tau_{B}(z,\omega_B)\leq \frac{3\tilde{\sigma}_2^2}{2\eta\xi\tilde{\sigma_1}}$. Hence, since $z\tilde{m}_{B}(z)=-1+\int_{\mathbb{R}}\frac{t^2}{t-z}d\mu_B(t)$,
$$L\leq \left(1+\frac{b_2}{\eta}\right)\cdot\left(1+\frac{3\mu_1(2)}{2\xi\tilde{\sigma_1}}+\frac{9}{4\xi^2}\right)\cdot\frac{1+\frac{3\tilde{\sigma}_2^2}{2\eta\xi\tilde{\sigma_1}}}{1-\frac{3}{2\xi k(\xi)}},$$
and, using the fact that $F_{\mu_1}(\omega_1)-\tilde{F}_{\mu_1}(\omega_1)=\sigma_1^2+\sigma_1^2m_{\rho}(\omega_1)-\tilde{\sigma}_1^2m_{\rho'}(\omega_1)$,
$$L'\leq \frac{3}{2\xi\tilde{\sigma}_1}\cdot\left(1+\frac{b_2}{\eta}\right)\cdot\left(\sigma_1^2+\frac{\sigma_1^2}{k(\xi)\tilde{\sigma}_1}+\frac{\tilde{\sigma}_1^2}{k(\xi)\tilde{\sigma}_1}\right)\cdot\frac{1+\frac{3\tilde{\sigma}_2^2}{2\eta\xi\tilde{\sigma_1}}}{1-\frac{3}{2\xi k(\xi)}}.$$
Then, we have by \eqref{definition_subor_multi}
$$\frac{\omega_B}{z}=\frac{\omega_B}{\omega_3}\frac{\omega_3}{z}=-h_{\mu_1}(\omega_1)\frac{\mathbb{E}m_{M}(\omega_3)}{\mathbb{E}f_{A}(\omega_3)}.$$
Since $\omega_3\mathbb{E}m_{M}(\omega_3)=-1+\mathbb{E}\tilde{m}_{M}(\omega_3)$, by Lemma \ref{lower_bounds_multi}
$$\left\vert \frac{\mathbb{E}m_{M}(\omega_3)}{\mathbb{E}f_{A}}\right\vert=\left\vert \frac{\omega_3\mathbb{E}m_{M}(\omega_3)}{\omega_3\mathbb{E}f_{A}}\right\vert\leq \left(1+\frac{\sigma_M}{\xi\tilde{\sigma}_1}\right)\cdot\left(1+\frac{a_2}{\xi\tilde{\sigma}_1}+\frac{a_2\sigma_B^2+\tilde{\sigma}_A^2}{(1-N^{-2})\xi^2\tilde{\sigma}_1^2}\right),$$
which yields
$$\left\vert\frac{\omega_B}{z}\right\vert\leq\left(1+\frac{\sigma_1^2}{k(\xi)\tilde{\sigma}_1}\right)\cdot\left(1+\frac{\sigma_M}{\xi\tilde{\sigma}_1}\right)\cdot\left(1+\frac{a_2}{\xi\tilde{\sigma}_1}+\frac{a_2\sigma_B^2+\tilde{\sigma}_A^2}{(1-N^{-2})\xi^2\tilde{\sigma}_1^2}\right).$$
Hence,
\begin{align*}
\vert\tilde{m}_{B}(z)\tilde{F}_{B}(\omega_B)\vert=&\left\vert\frac{\omega_B}{z}\right\vert\cdot\left\vert z\tilde{m}_{B}(z)\right\vert\cdot\left\vert \frac{\tilde{F}_{B}(\omega_B)}{\omega_B}\right\vert\\
\leq& \left(1+\frac{\sigma_1^2}{k(\xi)\tilde{\sigma}_1}\right)\cdot\left(1+\frac{\sigma_M}{\xi\tilde{\sigma}_1}\right)\cdot\left(1+\frac{a_2}{\xi\tilde{\sigma}_1}+\frac{a_2\sigma_B^2+\tilde{\sigma}_A^2}{(1-N^{-2})\xi^2\tilde{\sigma}_1^2}\right)\\
&\hspace{5cm}\cdot\left(1+\frac{b_2}{\eta}\right)\cdot\left(1+\frac{3b_2}{2\xi\tilde{\sigma}_2}+\frac{9\tilde{\sigma}_B^2}{4\xi^2\tilde{\sigma}_1^2}\right).
\end{align*}
Putting all the above bounds together, and using Proposition \ref{convergence_subord_multi} and Proposition \ref{bound_subord_B} to get $\epsilon_A\leq \frac{C_{bound,A}(\xi)}{N^2}$ and $\epsilon_B\leq \frac{C_{bound,B}(\xi)}{N^2}$, we finally obtain 
$$\vert \tilde{m}_B(z)-\tilde{m}_{M}(\omega_3)\vert\leq \frac{C_{1}(\kappa)}{N^2}+C_{2}(\kappa)\tilde{\delta}_A+C_{3}(\kappa)\tilde{\delta}_M,$$
with, for $\xi=g^{-1}(\kappa)$,
\begin{align}\label{first_constant_multi}
&C_{1}(\kappa)=\left(1+\frac{b_2}{\kappa\tilde{\sigma}_1}\right)\cdot\Bigg[\left(1+\frac{3\mu_1(2)}{2\xi\tilde{\sigma_1}}+\frac{9}{4\xi^2}\right)\cdot\frac{1+\frac{3\tilde{\sigma}_2^2}{2\eta\xi\tilde{\sigma_1}}}{1-\frac{3}{2\xi k(\xi)}}\cdot C_{bound,A}(\xi)\nonumber\\
+&\left(1+\frac{\sigma_1^2}{k(\xi)\tilde{\sigma}_1}\right)\cdot\left(1+\frac{\sigma_M}{\xi\tilde{\sigma}_1}\right)\cdot\left(1+\frac{a_2}{\xi\tilde{\sigma}_1}+\frac{a_2\sigma_B^2+\tilde{\sigma}_A^2}{(1-N^{-2})\xi^2\tilde{\sigma}_1^2}\right)\cdot\left(1+\frac{3b_2}{2\xi\tilde{\sigma}_2}+\frac{9\tilde{\sigma}_B^2}{4\xi^2\tilde{\sigma}_1^2}\right)\cdot C_{bound,B}(\xi)\Bigg],
\end{align}
\begin{equation}\label{second_constant_multi}
C_{2}(\kappa)=\left(1+\frac{b_2}{\kappa\tilde{\sigma}_1}\right)\cdot\left(1+\frac{3\mu_1(2)}{2\xi\tilde{\sigma_1}}+\frac{9}{4\xi^2}\right)\cdot\frac{1+\frac{3\tilde{\sigma}_2^2}{2\eta\xi\tilde{\sigma_1}}}{1-\frac{3}{2\xi k(\xi)}},
\end{equation}
and
\begin{equation}\label{third_constant_multi}
C_{3}(\kappa)=1+\frac{3}{2\xi\tilde{\sigma}_1}\cdot\left(1+\frac{b_2}{\kappa\tilde{\sigma}_1}\right)\cdot\left(\sigma_1^2+\frac{\sigma_1^2}{k(\xi)\tilde{\sigma}_1}+\frac{\tilde{\sigma}_1^2}{k(\xi)\tilde{\sigma}_1}\right)\cdot\frac{1+\frac{3\tilde{\sigma}_2^2}{2\eta\xi\tilde{\sigma_1}}}{1-\frac{3}{2\xi k(\xi)}}
\end{equation}
\end{proof}

\subsection{$L^2$-estimates}\label{Section:L2stability}
Building on the previous stability results, we deduce the proofs of Theorem \ref{Theorem:concentration_cauchy_additif} and Theorem \ref{Theorem:concentration_cauchy_multiplicatif}. In this section, we fix a parameter $\eta>0$ which denotes the imaginary part of the line on which the fist part of the deconvolution process is achieved (see Section \ref{Section:statement_deconvolution} for an explanation of the method). Then, for each $t\in\mathbb{R}$, the deconvolution process associates to each sample of $H$ or $M$ an estimator $\widehat{m_{B,\eta}}(t):=\widehat{m_{B}}(t+i\eta)$ of $m_{B,\eta}(t):=m_{B}(t+i\eta)$ respectively given by $\widehat{m_{B,\eta}}(t)=m_{H}(\omega_3(t+i\eta))$ and $\widehat{m_{B,\eta}}(t)=\frac{\omega_3(t+i\eta)}{t+i\eta}m_{M}(\omega_3(t+i\eta))$, with $\omega_3$ the subordination function respectively given by Theorem \ref{decon_additif_first_step} and Theorem \ref{subordination_deconvolution_multi}.  

Note first that the function $f_{z}:t\rightarrow \frac{z}{t-z}$ is $C^1$ for $z\in\mathbb{C}^+$, and, viewed as a function on $\mathcal{H}_{N}(\mathbb{C})$, we have for $A\in\mathcal{H}_{N}(\mathbb{C})$
$$\nabla f_{z}(A)(X)=\Tr\left(z\frac{1}{A-z}X\frac{1}{A-z}\right)=\Tr\left(\frac{z}{(A-z)^2}X\right).$$
Hence, $\Vert \nabla f_{z}(A)\Vert_{2}= \frac{1}{N}\left\Vert \frac{z}{(A-z)^2}\right\Vert_{2}\leq \frac{1}{N}\left(\left\Vert\frac{1}{A-z}\right\Vert_{2}+\left\Vert\frac{A}{(A-z)^2}\right\Vert_{2}\right)$ and thus, with the second hypothesis of Condition \ref{concentration_noise},  $\mathbb{E}\Vert \nabla f_{z}(A)\Vert_{2}^2\leq\frac{1}{N}\left(\frac{1}{\eta}+\frac{(1+c/N)\sqrt{\mu_1(2)}}{\eta^2}\right)^2$, where $\eta=\Im z$. This implies by the third hypothesis of Condition \ref{concentration_noise}
\begin{equation}\label{concentration_Cauchy_A}
\sqrt{\mathbb{E}\vert \omega_A\delta_A\vert^2}\leq \sqrt{\mathbb{E}\vert\tilde{\delta}_{A}\vert^2}\leq \frac{C_A\left(1+\frac{(1+c/N)\sqrt{\mu_1(2)}}{\Im \omega_A}\right)}{\Im \omega_A N}.
\end{equation}

Using the latter inequality, we deduce the following estimate in the additive case. 

\begin{proposition}\label{chebychev_bound_additif}
Suppose that $\eta>2\sqrt{2}\sigma_1$ and $N^2\geq\frac{4^3\max(C_{thres,A}(3\eta/4),C_{thres,B}(3\eta/4))}{3^3\eta^3}$. Then,
$$\mathbb{E}(\Vert \widehat{m_{B,\eta}}-m_{B,\eta}\Vert_{L^2}^2)\leq \frac{\pi}{\eta N^2}\left(\frac{4C_{A}C_{2}(\eta/\sigma_1)}{\eta^2}+\frac{8\sqrt{2}C_{3}(\eta/\sigma_1)}{3\eta}\sqrt{\sigma_A^2+4^2\frac{\sigma_A^2\sigma_B^2+a_4}{3^2\eta^2}}+\frac{C_{1}(\eta/\sigma_1)}{N}\right)^2.$$
\end{proposition}
\begin{proof}
Write temporarily $\omega_3(t+i\eta)=\omega_3$ and $\omega_A=\omega_A(\omega_3)$. By Theorem \ref{decon_additif_first_step}, we know that $\Im \omega_3\geq3\eta/4$. Hence, by Proposition \ref{stability_decon_additif}, for $z=t+i\eta$ with $\eta>2\sqrt{2}\sigma_1$ and $N^2\geq\frac{4^3\max(C_{thres,A}(3\eta/4),C_{thres,B}(3\eta/4))}{3^3\eta^3}$,
\begin{align*}
\vert m_{B}(z)-m_{H}(\omega_3)\vert\leq & \frac{C_{1}(\eta/\sigma_1)}{\vert z\vert N^2}
+\frac{C_{2}(\eta/\sigma_1)}{\vert z\vert}\vert \omega_A\delta_A\vert+\frac{C_{3}(\eta/\sigma_1)}{\vert z\vert}\vert \omega_3\delta_H\vert,
\end{align*}
with $C_{1}(\eta/\sigma_1),C_{2}(\eta/\sigma_1),C_{3}(\eta/\sigma_1)$ given in Proposition \ref{stability_decon_additif} for $\xi=\Im \omega_3$. Hence,
\begin{align*}
\mathbb{E}(\vert m_{B}(z)-m_{H}(\omega_3)\vert^2)\leq& \frac{1}{\vert z\vert^2}\left[\frac{C_{1}(\eta/\sigma_1)}{N^2}+C_{2}(\eta/\sigma_1)\sqrt{\mathbb{E}\left(\vert \omega_A\delta_A\vert^2\right)}+ C_{3}(\eta/\sigma_1)\sqrt{\mathbb{E}\left(\vert\omega_3\delta_H\vert^2\right)}\right]^2.
\end{align*}
First, by \eqref{concentration_Cauchy_A}, we have 
$$\sqrt{\mathbb{E}(\vert \omega_A\delta_A(\omega_A)\vert)}\leq \frac{C_A\left(1+\frac{(1+c/N)\sqrt{\mu_1(2)}}{\Im \omega_A}\right)}{\Im \omega_A N}\leq \frac{3C_A\left(1+\frac{3(1+c/N)\sqrt{\mu_1(2)}}{2\Im \omega_3}\right)}{2\Im \omega_3 N},$$ 
where the lower bound on $\Im \omega_A$ comes from Proposition \ref{convergence_subord_B}. Then, by the lower bound on $\Im \omega_3$ from Theorem \ref{decon_additif_first_step},
$$\sqrt{\mathbb{E}(\left\vert \omega_A\delta_A\right\vert^2)}\leq \frac{2C_A\left(1+\frac{2(1+c/N)\sqrt{\mu_1(2)}}{\eta}\right)}{\eta N}.$$
Finally, by Lemma \ref{concentration_result_additive_z} and the hypotheses $\Tr(A)=0$ and $\Tr(B)=0$,
\begin{align*}
\mathbb{E}\left(\vert \omega_3\delta_H(\omega_3)\vert^2\right)\leq & \frac{8}{N^2(\Im\omega_3)^2}\left(\sigma_A^2+\frac{\sigma_A^2\sigma_B^2+a_4}{(\Im\omega_3)^2}\right)\leq\frac{2^7}{3^2N^2\eta^2}\left(\sigma_A^2+4^2\frac{\sigma_A^2\sigma_B^2+a_4}{3^2\eta^2}\right).
\end{align*}
Hence, 
\begin{align*}
&\left[\frac{C_{1}(\eta/\sigma_1)}{N^2}+C_{2}(\eta/\sigma_1)\sqrt{\mathbb{E}\left(\vert \omega_A\delta_A\vert^2\right)}+ C_{3}(\eta/\sigma_1)\sqrt{\mathbb{E}\left(\vert\omega_3\delta_H\vert^2\right)}\right]^2\\
&\leq \left(\frac{C_{1}(\eta/\sigma_1)}{N^2}+\frac{2C_2(\eta/\sigma_1)C_A\left(1+\frac{2(1+c/N)\sqrt{\mu_1(2)}}{\eta}\right)}{\eta N}+\frac{8\sqrt{2}C_{3}(\eta/\sigma_1)}{3\eta N}\sqrt{\sigma_A^2+4^2\frac{\sigma_A^2\sigma_B^2+a_4}{3^2\eta^2}}\right)^2\\
&\leq \frac{1}{N^2}\left(\frac{2C_2(\eta/\sigma_1)C_A\left(1+\frac{2(1+c/N)\sqrt{\mu_1(2)}}{\eta}\right)}{\eta}+\frac{8\sqrt{2}C_{3}(\eta/\sigma_1)}{3\eta}\sqrt{\sigma_A^2+4^2\frac{\sigma_A^2\sigma_B^2+a_4}{3^2\eta^2}}+\frac{C_{1}(\eta/\sigma_1)}{N}\right)^2.
\end{align*}
Since, $\int_{\mathbb{R}}\frac{dt}{\vert t+i\eta\vert^2}=\frac{\pi}{\eta}$, the latter inequality yields
\begin{align*}
&\mathbb{E}(\vert \widehat{m_{B,\eta}}-m_{B,\eta}\vert_{L^2}^2)\\
&\leq \frac{\pi}{\eta N^2}\left(\frac{2C_2(\eta/\sigma_1)C_A\left(1+\frac{2(1+c/N)\sqrt{\mu_1(2)}}{\eta}\right)}{\eta}+\frac{8\sqrt{2}C_{3}(\eta/\sigma_1)}{3\eta}\sqrt{\sigma_A^2+4^2\frac{\sigma_A^2\sigma_B^2+a_4}{3^2\eta^2}}+\frac{C_{1}(\eta/\sigma_1)}{N}\right)^2.
\end{align*}
\end{proof}
\begin{proof}[Proof of Theorem \ref{Theorem:concentration_cauchy_additif}] Specifying the latter proposition for $\eta=2\sqrt{2}\sigma_1$ and taking the imaginary part imply statement of Theorem \ref{Theorem:concentration_cauchy_additif}.
\end{proof}

We get a similar result for the multiplicative case.
\begin{proposition}\label{chebychev_bound_multiplicative}
Suppose that $\eta=\kappa\tilde{\sigma}_1$ with $\kappa>g(\xi_0)\tilde{\sigma}_1$ and write $\xi=g^{-1}(\kappa)$. Suppose that $N^2\geq\frac{2\eta\max(C_{thres,A}(\xi\tilde{\sigma}_1),C_{thres,B}(\xi\tilde{\sigma}_1)}{\xi^3\tilde{\sigma}_1^3}\left(1+\frac{\sigma_1^2}{k(\xi)\tilde{\sigma}_1}\right)$, and set 
$$t_{N}=\frac{\sqrt{3}\xi^3\tilde{\sigma}_1^3N^2}{2\max(C_{thres,A}(\xi\tilde{\sigma}_1),C_{thres,B}(\xi\tilde{\sigma}_1))\left(1+\frac{\sigma_1^2}{k(\xi)\tilde{\sigma}_1}\right)}.$$ 
Then,
$$\mathbb{E}(\Vert \widehat{m_{B,\eta}}-m_{B,\eta}\Vert_{L^2([-t_n,t_n])}^2)\leq \frac{K_{1}}{N^2}+\frac{K_{2}}{N^3}+\frac{K_{3}}{N^4},$$
with
$$K_{1}(\eta)=\frac{2\pi}{\kappa\tilde{\sigma}_1}\left(\frac{3^4 C_{2}(g^{-1}(\kappa)))^2C_{A}}{2^4g^{-1}(g^{-1}(\kappa))^4\tilde{\sigma_1}^4}+\frac{\Delta(\kappa)C_{3}((g^{-1}(\kappa)))^2}{ g^{-1}(\kappa)^2\tilde{\sigma}_1^2}\right),$$
with $\Delta(\kappa)$ is given in \eqref{definition_delta_kappa},
$$K_{2}(\eta)=\frac{2\pi C_{1}((g^{-1}(\kappa)))}{\eta}\left(\frac{9 C_{A}C_{2}((g^{-1}(\kappa)))}{4g^{-1}(\kappa)^2\tilde{\sigma_1}^2}+\frac{\sqrt{\Delta(\kappa)}C_{3}((g^{-1}(\kappa))) }{  g^{-1}(\kappa)\tilde{\sigma}_1}\right),$$
and
$$K_{3}(\eta)= \frac{\pi C_{1}((g^{-1}(\kappa)))^2}{\eta}.$$
\end{proposition}

\begin{proof}
The proof is similar to the additive case, but we have to take into account the fact that the bound we got in Proposition \ref{stability_decon_multi} only holds on a sub-interval of $\mathbb{R}$.  Indeed, by this Proposition, for $z=t+i\kappa\tilde{\sigma}_1$ with $\kappa>g(\xi_0)$ and when, with $\xi=g^{-1}(\kappa)$,
$$N^2\geq\frac{\vert \omega_3\vert}{\xi^3\tilde{\sigma}_1^3}\max\Big(C_{thres,A}(\xi\tilde{\sigma}_1),C_{thres,B}(\xi\tilde{\sigma}_1)\Big),$$
we have 
\begin{align*}
\left\vert \tilde{m}_{M}(\omega_3)-\tilde{m}_{B}(z)\right\vert\leq \frac{C_{1}(\kappa)}{N^2}+C_{2}(\kappa)\tilde{\delta}_{A}+C_{3}(\kappa)\tilde{\delta}_M
\end{align*}
with $C_{1}(\kappa),\,C_{2}(\kappa),\,C_{3}(\kappa)$ given in Proposition \ref{stability_decon_multi}. Hence, 
Since $\omega_3(z)=-h_{\mu_1}(\omega_1)z$ and $\Im \omega_1(z)\geq k(\Im \omega_3/\tilde{\sigma}_1)\tilde{\sigma}_1\geq k(\xi)\tilde{\sigma}_1$, the condition on $N$ is fulfilled when
\begin{align*}
\vert z\vert=\sqrt{t^2+\eta^2} \leq&\frac{\xi^3\tilde{\sigma}_1^3N^2}{\max\Big(C_{thres,A}(\xi\tilde{\sigma}_1),C_{thres,B}(\xi\tilde{\sigma}_1)\Big)\vert h_{\mu_1}(\omega_1)\vert}\\
\leq& \frac{\xi^3\tilde{\sigma}_1^3N^2}{\max\Big(C_{thres,A}(\xi\tilde{\sigma}_1),C_{thres,B}(\xi\tilde{\sigma}_1)\Big)\left(1+\frac{\sigma_1^2}{k(\xi)\tilde{\sigma}_1}\right)}.
\end{align*}
By the hypothesis on $N$ from the statement of the proposition, this is satisfied always satisfied when $t\leq \sqrt{3}\eta$. When $t\geq \sqrt{3}\eta$, this is then satisfied when
$$t\leq \frac{\sqrt{3}\xi^3\tilde{\sigma}_1^3N^2}{2\max(C_{thres,A}(\xi\tilde{\sigma}_1),C_{thres,B}(\xi\tilde{\sigma}_1))\left(1+\frac{\sigma_1^2}{k(\xi)\tilde{\sigma}_1}\right)}.$$ 
Set $t_{N}=\frac{\sqrt{3}\xi^3\tilde{\sigma}_1^3N^2}{2\max\Big(C_{thres,A}(\xi\tilde{\sigma}_1),C_{thres,B}(\xi\tilde{\sigma}_1)\Big)\left(1+\frac{\sigma_1^2}{k(\xi)\tilde{\sigma}_1}\right)}$. 
Then, writing $z=t+i\eta$,
$$\vert \widehat{m_{B,\eta}}(t)-m_{B,\eta}(t)\vert=\left\vert \frac{\omega_3(z)}{z}m_{M}(\omega_3(z))-m_{B}(z)\right\vert=\frac{1}{\vert z\vert}\vert\tilde{m}_{M}(\omega_3(z))-\tilde{m}_{B}(z)\vert.$$
Hence, by the hypothesis on $N$, the definition of $t_{N}$ and Proposition \ref{stability_decon_multi},
\begin{align*}
\mathbb{E}(\Vert \widehat{m_{B,\eta}}-m_{B,\eta}\Vert_{L^2([-t_n,t_n])}^2)\leq&\int_{\mathbb{R}}\frac{1}{\vert z\vert^2}\mathbb{E}(\vert \tilde{m}_{M}(\omega_3(t+i\eta))-\tilde{m}_{B}(z)\vert^2)dt\\
\leq& \int_{\mathbb{R}}\frac{1}{\vert z\vert^2}\left( \frac{C_{1}(\kappa)}{N^2}+C_{2}(\kappa)\sqrt{\mathbb{E}\left(\vert\tilde{\delta}_A\vert^2\right)}+C_{3}(\kappa)\sqrt{\mathbb{E}\left(\vert\tilde{\delta}_{M}\vert^2\right)}\right)^2dt,
\end{align*}
with $C_{1}(\kappa),\,C_{2}(\kappa),\,C_{3}(\kappa)$ respectively given in \eqref{first_constant_multi}, \eqref{second_constant_multi} and \eqref{third_constant_multi}. By Lemma \ref{concentration_result_multi_simple},
\begin{align*}
\mathbb{E}(\vert \tilde{\delta}_M\vert_{2}^2)\leq\frac{\Delta(\kappa)}{g^{-1}(\kappa)^2\tilde{\sigma}_1^2N^2},
\end{align*}
with 
\begin{equation}\label{definition_delta_kappa}
\Delta(\kappa)=8\Bigg(\sqrt{a_2(b^0_4+\sigma_A^2\sigma_B^4)}
+\frac{a_\infty}{g^{-1}(\kappa)^2\tilde{\sigma}_1^2}\left(m_{A\ast B}(1^3,21^2)-2m_{A\ast B}(1^3,1^3)+m_{A\ast B}(21,1^2)\right)\Bigg)
\end{equation}
Since $\Im \omega_A\geq 2\Im \omega_3/3\geq 2g^{-1}(\kappa)/3\tilde{\sigma_1}$ by Proposition \ref{minoration_omega_multiplicative}, \eqref{concentration_Cauchy_A} yields
\begin{align*}
\sqrt{\mathbb{E}(\vert \tilde{\delta}_A(\omega_A)\vert^2)}\leq \frac{3C_A\left(1+\frac{3(1+c/N)\sqrt{\mu_1(2)}}{2g^{-1}(\kappa)\tilde{\sigma_1}}\right)}{2g^{-1}(\kappa)\tilde{\sigma_1} N}.
\end{align*}
Putting all the above bound together and using that $\int_{\mathbb{R}}\frac{dt}{\vert z\vert^2}=\frac{\pi}{\kappa\tilde{\sigma}_1}$ yields
\begin{align*}
\mathbb{E}(\vert \widehat{m_{B,\eta}}-m_{B,\eta}\vert&_{L^2([-t_n,t_n])}^2)\\
\leq& \frac{\pi}{\kappa\tilde{\sigma}_1N^2}\left(\frac{3C_2(\kappa)C_A\left(1+\frac{3(1+c/N)\sqrt{\mu_1(2)}}{2g^{-1}(\kappa)\tilde{\sigma_1}}\right)}{2g^{-1}(\kappa)\tilde{\sigma_1} }+\frac{C_3(\kappa)\sqrt{\Delta(\kappa)}}{g^{-1}(\kappa)\tilde{\sigma}_1}+\frac{C_1(\kappa)}{N}\right)^2.
\end{align*}
\end{proof}
It remains to estimate the contribution of $m_{B,\eta}$ on $\mathbb{R}\setminus [-t_N,t_N]$ to the $L^2$-norm of $m_{B,\eta}$. Remark that we are only interested in the imaginary part of this function to build the estimator $\widehat{C_{B}}[\eta]$. Hence, we get the following estimates.
\begin{lemma}\label{negligible_remaining}
Suppose that $N^2\geq\frac{2\eta\max(C_{thres,A}(\xi\tilde{\sigma}_1),C_{thres,B}(\xi\tilde{\sigma}_1)}{\xi^3\tilde{\sigma}_1^3}\left(1+\frac{\sigma_1^2}{k(\xi)\tilde{\sigma}_1}\right)$. Then,
$$\Vert \Im m_{B,\eta}\Vert_{L^2(\mathbb{R}\setminus [-t_N,t_N])}^2\leq \frac{2^4\max(C_{thres,A}(\xi\tilde{\sigma}_1),C_{thres,B}(\xi\tilde{\sigma}_1))^3\left(1+\frac{1}{k\circ g^{-1}(\kappa)}\right)^3}{N^6\sqrt{3}(\xi\tilde{\sigma}_1)^9}.$$
\end{lemma}
\begin{proof}
Note first that for $\mu$ a probability measure with second moment,
$$m_\mu(z)=-\frac{1}{z}+\frac{1}{z^2}\left(-\mu(1)+\int_{\mathbb{R}}\frac{t^2}{t-z}d\mu(t)\right).$$
Hence, for $z$ such that $z=t+i\eta$,
$$\vert\Im m_{\mu}(z)\vert\leq \vert \Im(z^{-1})\vert +\frac{\mu(1)+\frac{\mu(2)}{\eta}}{\vert z\vert^2}\leq\frac{1}{\vert z\vert^2}\left(\eta+\mu(1)+\frac{\mu(2)}{\eta}\right).$$
Thus,
$$\int_{t_N}^{+\infty}\vert\Im m_{B}(t+i\eta)\vert^2dt\leq \left(\eta+1+\frac{b_2}{\eta}\right)^2\int_{t_N}^{\infty}\frac{dt}{(t^2+\eta^2)^2}\leq \frac{3\left(\eta+1+\frac{b_2}{\eta}\right)^2}{t_N^3},$$
and using the definition of $t_N$ yields
$$\Vert \Im m_{B,\eta}\Vert^2_{L^2(\mathbb{R}\setminus [-t_N,t_N])}\leq \frac{6\cdot2^3\max(C_{thres,A}(\xi\tilde{\sigma}_1),C_{thres,B}(\xi\tilde{\sigma}_1))^3\left(1-\frac{1}{k\circ g^{-1}(\kappa)}\right)^3\left(\eta+1+\frac{b_2}{\eta}\right)^2}{N^63^{3/2}(\xi\tilde{\sigma}_1)^9}.$$
\end{proof}
We can now prove Theorem \ref{Theorem:concentration_cauchy_multiplicatif}.
\begin{proof}
Set $\eta=\kappa\tilde{\sigma}_1$ with $\kappa>g(\xi_0)$. Then,
\begin{align*}
&\mathbb{E}(\Vert \widehat{\mathcal{C}}_{B}(\eta)-\mathcal{C}_{B}(\eta)\Vert_{L^2}^2)\\
=&\frac{1}{\pi^2}\int_{\mathbb{R}\setminus[-t_{N},t_{N}]}\mathbb{E}\vert \Im m_{B,\eta}(t+i\eta)\vert^2dt+\frac{1}{\pi^2}\int_{-t_{N}}^{t_{N}}\mathbb{E}\vert \Im \widehat{m_{B,\eta}}(t+i\eta)-\Im m_{B,\eta}(t+i\eta)\vert^2dt\\
\leq&\frac{1}{\pi^2}\int_{\mathbb{R}\setminus[-t_{N},t_{N}]}\vert \Im m_{B,\eta}(t+i\eta)\vert^2dt+\frac{1}{\pi^2}\int_{-t_{N}}^{t_{N}}\mathbb{E}\vert \widehat{m_{B,\eta}}(t+i\eta)-m_{B,\eta}(t+i\eta)\vert^2dt.
\end{align*}
On the one hand, Lemma \ref{negligible_remaining} yields
$$\int_{\mathbb{R}\setminus[-t_{N},t_{N}]}\vert \Im m_{B,\eta}(t+i\eta)\vert^2dt\leq \frac{\pi^2C_{4}(\kappa)}{N^6},$$
with 
\begin{equation}\label{definition_threshold_bound}
C_{4}(\kappa)=\frac{2^4\max(C_{thres,A}(\xi\tilde{\sigma}_1),C_{thres,B}(\xi\tilde{\sigma}_1))^3\left(1+\frac{1}{\pi^2k\circ g^{-1}(\kappa)}\right)^3}{\pi^2\sqrt{3}(\xi\tilde{\sigma}_1)^9}.
\end{equation}
On the other hand, by Proposition \ref{chebychev_bound_multiplicative},
\begin{align*}
\mathbb{E}(\vert \widehat{m_{B,\eta}}-m_{B,\eta}\vert&_{L^2([-t_n,t_n])}^2)\\
\leq& \frac{1}{\kappa\pi\tilde{\sigma}_1N^2}\left(\frac{3C_2(\kappa)C_A\left(1+\frac{3(1+c/N)\sqrt{\mu_1(2)}}{2g^{-1}(\kappa)\tilde{\sigma_1}}\right)}{2g^{-1}(\kappa)\tilde{\sigma_1} }+\frac{C_3(\kappa)\sqrt{\Delta(\kappa)}}{g^{-1}(\kappa)\tilde{\sigma}_1}+\frac{C_1(\kappa)}{N}\right)^2,
\end{align*}
with $C_{1}(\kappa), \,C_{2}(\kappa)$ and $C_{3}(\kappa)$ given in Proposition \ref{chebychev_bound_multiplicative}. The statement of the theorem is deduced from the two latter bounds.
\end{proof}

\appendix
\addcontentsline{toc}{section}{Appendices}

\section{Subordination in the multiplicative case}\label{Appendix:proof_multiplicative_subordination}
The goal of this first appendix is to prove Theorem \ref{subordination_deconvolution_multi}, which we recall here.
\begin{theorem*}
There exist two analytic functions $\omega_1,\omega_3:\mathbb{C}_{g(\xi_0)\tilde{\sigma}_1}\rightarrow\mathbb{C}^+$ such that 
$$z\omega_1(z) =\omega_3(z)\frac{\omega_3(z)m_{M}(\omega_3(z))}{1+\omega_3(z)m_{M}(\omega_3(z))}=\omega_3(z)\frac{\omega_1(z)m_{\mu_1}(\omega_1(z))}{1+\omega_1(z)m_{\mu_1}(\omega_1(z))}$$ 
for all $z\in \mathbb{C}_{g(\xi_0)\tilde{\sigma}_1}$. Moreover, setting $K_{z}(w)=-h_{\mu_1}\left(w^2\frac{ m_{M}(w)}{1+w m_{M}(w)}/z\right)z$ for $z \in\mathbb{C}_{g(\xi_0)\tilde{\sigma_1}}$ and $w\in \mathbb{C}^+$, then
\begin{enumerate}
\item if $\Re z< -K_0$ with $K_0$ given in Lemma \ref{convergence_negative_real}, then 
$$\omega_3(z)=\lim_{n\rightarrow \infty} K_{z}^{\circ n}(z),$$
\item if $z\in \mathbb{C}_{g(\xi_0)\tilde{\sigma}_1}$, then for all $z'\in \mathbb{C}_{g(\xi_0)\tilde{\sigma}_1}\cap B(z,R(g^{-1}(\Im z)))$, with $R(g^{-1}(\Im(z)))>0$ given in \eqref{definition_R},
$$\omega_3(z')=\lim_{n\rightarrow \infty} K_{z'}^{\circ n}(\omega_3(z)).$$
\end{enumerate}
\end{theorem*}

In the following lemma, recall that $k$ is the function defined on $[2,+\infty[$ by $k(t)=\frac{t+\sqrt{t^2-4}}{2}$.
\begin{lemma}\label{control_inverse_Fmu}
Let $\mu$ be a probability measure with finite variance $\sigma^2$. If $w\in \mathbb{C}^+$ is such that $\Im \omega>2\sigma$, then there exists $z\in\mathbb{C}^+$ with $\Im z>k(\Im \omega/\sigma)\sigma$ such that $F_{\mu}(z)=\omega$.
\end{lemma}
\begin{proof}
By \cite[Lemma 24]{MiSp}, the inverse $F_{\mu}^{<-1>}$ of $F_\mu$ is well-defined on $\mathbb{C}_{2\sigma}$ and takes values in $\mathbb{C}_{\sigma}$. Hence, if $w\in\mathbb{C}^+$ is such that $\Im w>2\sigma$, there exists $z\in \mathbb{C}_{\sigma}$ such that $F_{\mu}(z)=w$. By \eqref{F_mu_expression}, $\vert F_{\mu}(z)-z+\mu(1)\vert\leq \frac{\sigma^2}{\Im (z)}$, which yields
$$\Im \omega-\Im z\leq \frac{\sigma^2}{\Im(z)}.$$
Hence, dividing the latter inequality by $\sigma$ and setting $t=\Im\omega/\sigma$, $\xi=\Im(z)/\sigma$, we have 
$$t-\xi\leq \frac{1}{\xi},$$ 
or $\xi^2-t\xi+1\geq 0$. Since $t>2$ and $\xi>1$, this implies that
$\xi\geq k(t)$ with $k(t)=\frac{t+\sqrt{t^2-4}}{2}$, or equivalently
$$\Im z>k(\Im \omega/\sigma)\sigma.$$
\end{proof}

For $z\in\mathbb{C}$, set
$$\Phi_{z}(\omega_1,\omega_3)=\begin{pmatrix}
\omega_1z-\omega_3\hat{F}_{\mu_1}(\omega_1)\\
\omega_1z-\omega_3\hat{F}_{M}(\omega_3),
\end{pmatrix}$$
where $\hat{F}_{\mu}(w)=1+F_{\tilde{\mu}}(w)=\frac{wm_\mu(w)}{1+wm_\mu(w)}$ is defined in Section \ref{Section :Probability measure with positive support}, and remark that $\Phi_{z}(\omega_1,\omega_3)=0$ precisely when $(\omega_1,\omega_3)$ satisfies the first relations of Theorem \ref{subordination_deconvolution_multi}. Recall that we assume $\mu_1(1)=\mu_M(1)=1$, and we write $\tilde{\sigma}_{i}^2=\Var(\tilde{\mu}_i)=\mu_i(3)-\mu_i(2)^2$ for $i=1,M$. We first have the following relations between $\Im z$ and $\Im \omega_3$ when $\Phi_{z}(\omega_1,\omega_3)=0$. 
\begin{lemma}\label{upper_bound_z_multi}
If $\Im \omega_3>2\tilde{\sigma}_1$, there exist $z\in\mathbb{C},\,\omega_1\in \mathbb{C}^+$ such that $\Phi_{z}(\omega_1,\omega_3)=0$. Moreover, if we write $\Im z=k_{z}\tilde{\sigma}_1$ and $\Im \omega_3=k_3\tilde{\sigma}_{1}$, we have 
$$k_{z}\leq k_3+\frac{1}{k(k_3)}+\frac{1}{k(k_3)}\left(\frac{1}{k(k_3)}+\frac{\vert\sigma_M^2-\sigma_1^2\vert}{k(k_3)\tilde{\sigma_1}}+\frac{\tilde{\sigma}_M^2}{k_3\tilde{\sigma}_1^2}\right)\left(\frac{\sigma_1^2}{\tilde{\sigma}_1}+\frac{1}{k(k_3)}\right):=g(k_3).$$ 
\end{lemma}

\begin{proof}
Suppose that $\Im \omega_3>2\tilde{\sigma_1}$. Then, $\Im \hat{F}_{M}(\omega_3)\geq \Im \omega_3>2\tilde{\sigma}_1$ by \eqref{F_mu_increasing_imaginary}, and thus by Lemma \ref{control_inverse_Fmu} there exists $\omega_1$ such that $\hat{F}_{\mu_1}(\omega_1)=\hat{F}_{M}(\omega_3)$ and $\Im \omega_1\geq k(\Im \hat{F}_{M}(\omega_3)/\tilde{\sigma}_1)\tilde{\sigma}_1$. Since the function $k$ is increasing, we have in particular $\Im \omega_1\geq k(k_{3})\tilde{\sigma}_1$. Since $\hat{F}_{\mu_1}(\omega_1)=\hat{F}_{M}(\omega_3)$, we have by using \eqref{F_mu_expression}
\begin{align}
\vert \omega_1-\omega_3\vert\leq& \vert \omega_1-\hat{F}_{\mu_1}(\omega_1)-\omega_3+\hat{F}_{M}(\omega_3)\vert\nonumber\\
\leq&\vert \sigma_1^2-\sigma_M^2+\tilde{\sigma}_M^2m_{\rho_3}(\omega_3)-\tilde{\sigma}_1^2m_{\rho_1}(\omega_1)\vert\nonumber\\
\leq &\left(\frac{1}{k(k_3)}+\frac{\tilde{\sigma}_M^2}{k_3\tilde{\sigma}_1^2}\right)\tilde{\sigma}_1+\vert\sigma_M^2-\sigma_1^2\vert,\label{distance_omega1_omega3}
\end{align}
Setting $z=\frac{\omega_3}{\omega_1}\hat{F}_{M}(\omega_3)$ yields then 
$$\Phi_{z}(\omega_1,\omega_3)=0.$$
Writing $\hat{F}_{\mu_1}(\omega_1)=\omega_1-\Var(\mu_1)+\tilde{\sigma}_1^2m_{\rho_{1}}(\omega_1)$ gives also
\begin{align*}
z=\omega_3\frac{\hat{F}_{\mu_1}(\omega_1)}{\omega_1}=&\omega_3-\frac{\omega_3}{\omega_1}\left(\Var(\mu_1)-\tilde{\sigma}_1^2m_{\rho_{1}}(\omega_1)\right)\\
=&\omega_3-\left(1+\frac{\omega_3-\omega_1}{\omega_1}\right)\left(\Var(\mu_1)-\tilde{\sigma}_{1}^2m_{\rho_{1}}(\omega_1)\right)
 .
\end{align*}
Hence, since $\Var(\mu_1)$ is real,
$$\Im z\leq \Im \omega_3+\frac{\tilde{\sigma}_1^2}{\Im(\omega_1)}+\frac{1}{\Im \omega_1}\left(\frac{\tilde{\sigma}_1}{k(k_3)}+\frac{\tilde{\sigma}_M^2}{k_3\tilde{\sigma}_1}+\vert\sigma_M^2-\sigma_1^2\vert\right)\left(\Var(\mu_1)+\frac{\tilde{\sigma}_{1}^2}{\Im\omega_1}\right).$$
Using that  $\Im\omega_1\geq k(k_3)\tilde{\sigma}_1$ implies then
$$\Im z\leq k_3\tilde{\sigma}_1+\frac{\tilde{\sigma}_1}{k(k_3)}+\frac{1}{k(k_3)}\left(\frac{1}{k(k_3)}+\frac{\vert\sigma_M^2-\sigma_1^2\vert}{k(k_3)\tilde{\sigma_1}}+\frac{\tilde{\sigma}_M^2}{k_3\tilde{\sigma}_1^2}\right)\left(\frac{\sigma_1^2}{\tilde{\sigma}_1}+\frac{1}{k(k_3)}\right)\tilde{\sigma}_1.$$
The inequality of the statement is then obtained after dividing by $\tilde{\sigma}_1$. 
\end{proof} 
In the sequel, define $H_{3}(w)=w\hat{F}_{M}(w)$ and $K_{z}(w)=-h_{\mu_1}(H_{3}(w)/z)z$ for $w\in\mathbb{C}^+$. Define also two functions $\theta,L :\mathbb{R}_{>0}\rightarrow \mathbb{R}$ by 
\begin{equation}\label{eq:expression_theta}
\theta(u)=6\left(1+\frac{\sigma_1^2}{k(u)\tilde{\sigma}_1}\right)\cdot\left(1+\frac{\sigma_M^2}{u\tilde{\sigma}_1}+\frac{4\tilde{\sigma}_M^2}{u^2\tilde{\sigma}_1^2}\right),
\end{equation}
and
\begin{align}
L(u)= &32\left(\frac{\sigma_1^2}{(u^2-4)\tilde{\sigma}_1^2}+\frac{2(\mu_1(3)-2\mu_1(2)+1)}{(u^2-4)^{3/2}\tilde{\sigma}_1^3}\right)\cdot \left(1+\frac{\sigma_M^2}{u\tilde{\sigma}_1}+\frac{4\tilde{\sigma}_M^2+\sigma_M^4}{u^2\tilde{\sigma}_1^2}\right)^2\\
&\hspace{5cm}+\frac{8\sigma_1^2}{(u^2-4)\tilde{\sigma_1}^2}\cdot\left(1 +8\frac{m_4-2m_3m_2+m_2^2}{u^3\tilde{\sigma}_1^3}\right).\label{eq:expression_L}
\end{align}
The expression of the two latter functions is not important regarding the statement of Theorem \ref{subordination_deconvolution_multi}, but they play a role in the concrete implementation of the deconvolution procedure. In the following lemma, recall the definition of $t$ from \eqref{definition_t}.
\begin{lemma}\label{definition_domain_tildeh}
Suppose that $\Phi_{z}(\omega_1,\omega_3)=0$ with $k_3:=\Im \omega_3/\tilde{\sigma}_1> 2$. Then, $K_{z}(\omega_3)=\omega_3$,
$$\vert K_{z}'(\omega_3))\vert\leq  t(k_3),$$
and if $\vert w-\omega_3\vert \leq k_3\tilde{\sigma}_1/\theta(k_3)$, then $K_{z}(w)$ is well-defined and satisfies
$$\vert K''(w)\vert\leq L(k_3).$$
\end{lemma}
\begin{proof}
Note first that since $\Phi_z(\omega_1,\omega_3)=0$, $\omega_1=\frac{\omega_3\hat{F}_{M}(\omega_3)}{z}=H_{3}(\omega_3)/z$. Hence, using again the relation $\Phi_z(\omega_1,\omega_3)=0$ together with \eqref{relation_F_hat} yields $K_{z}(\omega_3)=\omega_3$. Moreover, for $w\in\mathbb{C}^+$ such that $w':=H_{3}(w)/z\in \mathbb{C}^+$,
$$K_z'(w)=-h'_{\mu_1}(H_{3}(w)/z)H'_{3}(w)=-\frac{zw'h_{\mu_1}(w')h'_{\mu_1}(w')}{wh_{\mu_1}(w')}\frac{H'_{3}(w)}{\hat{F}_{M}(w)}.$$
Since $H'_{3}(w)=w\hat{F}'_{M}(w)+\hat{F}_{M}(w)$,
\begin{equation}\label{bound_kz'}K_z'(w)=-\frac{zw'h_{\mu_1}(w')h'_{\mu_1}(w')}{wh_{\mu_1}(w')}\left(1+\frac{w\hat{F}'_{M}(w)}{\hat{F}_{M}(w)}\right)=-\frac{zh_{\mu_1}(w')}{w}u_{1}(2-u_{3}),
\end{equation}
with $u_1=\frac{w'h'_{\mu_1}(w')}{h_{\mu_1}(w')}$ and $u_3=1-\frac{w\hat{F}_{M}'(w)}{\hat{F}_{M}(w)}$. Remark that \eqref{relation_F_hat} implies then
$$1-\frac{w}{\hat{F}_{M}(w)}\hat{F}'_{M}(w)=1+h_{M}(w)\left(-\frac{1}{h_{M}(w)}+\frac{wh_{M}'(w)}{h_{M}(w)^2}\right)=\frac{wh_{M}'(w)}{h_{M}(w)}.$$
Moreover, by \eqref{nevanlinna_LH}, for $\mu$ a probability measure supported on $\mathbb{R}^+$ with $\mu(1)=1$ and $u\in\mathbb{C}^+$,
$$\frac{uh'_{\mu}(u)}{h_{\mu}(u)}=uL_{\mu}'(u)=\Var(\mu)(um_{\rho_{L}}'(u))=\Var(\mu)\left(-m_{\rho_{L}}(u)+\int_{\mathbb{R}}\frac{t}{(u-t)^2}d\rho_{L}(t)\right),$$
with $\rho_L,\, L_\mu$ given in Section \ref{Section :Probability measure with positive support}. This implies, using the formula $\rho_{L}(1)=\frac{\Var(\tilde{\mu})+\Var(\mu)/2}{\Var(\mu)}$ given before \eqref{nevanlinna_LH},
\begin{align*}
\left\vert \frac{uh'_{\mu}(u)}{h_{\mu}(u)}\right\vert\leq \Var(\mu)\left(\frac{1}{\Im u}+\int_{\mathbb{R}}\left\vert\frac{t}{(u-t)^2}\right\vert d\rho_{L}(t)\right)\leq& \Var(\mu)\left(\frac{1}{\Im u}+\frac{1}{(\Im u)^2}\int_{\mathbb{R}}\vert t\vert d\rho_{L}(t)\right)\\
\leq& \frac{\Var(\mu)}{\Im u}+\frac{\Var(\tilde{\mu})+ \Var(\mu)^2/2}{\Im(u)^2}.
\end{align*}
Hence, applying this bound to $u_1$ and $u_3$ in \eqref{bound_kz'} gives 
\begin{equation}\label{eq:first_bound_Kz}
\vert K'_{z}(w)\vert\leq \left\vert\frac{zh_{\mu_1}(w')}{w}\right\vert \left(\frac{\sigma_1^2}{\Im w'}+\frac{\tilde{\sigma}_1^2+\sigma_1^4/2}{\Im(w')^2}\right)\left(2+\frac{\sigma_M^2}{\Im w}+\frac{\tilde{\sigma}_M^2+\sigma_M^4/2}{\Im(w)^2}\right).
\end{equation}
Remark that for $w=\omega_3$, then $w'=\omega_1$ and $\frac{zh_{\mu_1}(\omega_1)}{\omega_3}=-1$. Since $\Im \omega_3=k_3\tilde{\sigma}_1$ and  $\Im w'\geq k(k_3)\tilde{\sigma}_1$ by Lemma \ref{control_inverse_Fmu}, we thus obtain
$$\vert K'_{z}(\omega_3)\vert\leq  \left(\frac{\sigma_1^2}{k(k_3)\tilde{\sigma}_1}+\frac{\tilde{\sigma}_1^2+\sigma_1^4/2}{k(k_3)^2\tilde{\sigma}_1^2}\right)\left(2+\frac{\sigma_M^2}{k_3\tilde{\sigma}_1}+\frac{\tilde{\sigma}_M^2+\sigma_M^4/2}{k_3^2\tilde{\sigma}_1^2}\right)= t(k_3).$$
The goal of the proof is now to bound $K_z''$ in a neighborhood of $\omega_3$.  First, by \eqref{F_mu_expression} applied to $\hat{F}_{M}$, 
$\hat{F}_{M}(w)=1+\tilde{F}_{M}(w)=w-\sigma_M^2+\tilde{\sigma}_M^2m_{\tilde{\rho}}(w)$ for some probability measure $\tilde{\rho}$. Hence,
\begin{align}
H_3'(w)=&\left(\hat{F}_{M}(w)+w\hat{F}'_{M}(w)\right)\nonumber\\
=&w\left(2+\frac{-\sigma_M^2+\tilde{\sigma}_M^2m_{\tilde{\rho}}(w)}{w}+\tilde{\sigma}_M^2\int_{R}\frac{1}{(t-w)^2}d\tilde{\rho}(t)\right).\label{derivative_H3}
\end{align}
Then, the equality $\omega_3=h_{\mu_1}(\omega_1)z$ yields
\begin{equation}\label{ratio_nomr_w_z}
\left\vert\frac{w}{z}\right\vert=\left\vert \frac{w}{\omega_3}\right\vert\cdot\left\vert \frac{\omega_3}{z}\right\vert=\left\vert \frac{w}{\omega_3}\right\vert\cdot\left\vert h_{\mu_1}(\omega_1)\right\vert\leq \frac{3}{2}\left(1+\frac{\sigma_1^2}{k(k_3)\tilde{\sigma}_1}\right),
\end{equation}
for $w$ such that $\vert w-\omega_3\vert\leq \frac{\Im \omega_3}{2}$, where we used the definition of $h_{\mu_1}$ from Section \ref{Section:free_probabilistic_background} on the last inequality. This implies
\begin{align}
\left\vert \frac{1}{z}H_3'(w)\right\vert\leq& \frac{3}{2}\left(1+\frac{\sigma_1^2}{k(k_3)\tilde{\sigma}_1}\right)\cdot\left(2+\frac{\sigma_M^2}{\Im w}+\frac{2\tilde{\sigma}_M^2}{(\Im w)^2}\right)\nonumber\\
&\leq 3\left(1+\frac{\sigma_1^2}{k(k_3)\tilde{\sigma}_1}\right)\cdot\left(1+\frac{\sigma_M^2}{k_3\tilde{\sigma}_1}+\frac{4\tilde{\sigma}_M^2}{k_3^2\tilde{\sigma}_1^2}\right)=\theta(k_3)/2,\label{eq:bound_derivartive_H3}
\end{align}
when we assume $\vert w-\omega_3\vert\leq \frac{k_3\tilde{\sigma}_1}{2}$. Hence, for $\vert w-\omega_3\vert\leq \frac{ k_3\tilde{\sigma}_1}{\theta(k_3)}$, using $\theta(k_3)>6$ yields first $\frac{k_3\tilde{\sigma}_1}{\theta(k_3)}\leq \frac{k_3\tilde{\sigma}_1}{2}$, and then we get
\begin{equation}\label{lower_bound_imaginary_H3}
\Im w'=\Im\frac{H_3(w)}{z}\geq \Im\frac{H_{3}(\omega_3)}{z}-\theta(k_3)/2\frac{ k_3\tilde{\sigma}_1}{\theta(k_3)}\geq (k(k_3)-k_3/2)\tilde{\sigma}_1>\frac{\sqrt{k_3^2-4}}{2}\tilde{\sigma}_1,
\end{equation}
so that $h_{\mu_1}(H_3(w)/z)$ is well-defined. Then,
$$K_{z}''(w)=-H'_{3}(w)^2/zh''_{\mu_1}(w')-H''_{3}(w)h'_{\mu_1}(w')=-\frac{H'_{3}(w)^2}{H_3(w)}w'h''_{\mu_1}(w')-H''_{3}(w)h'_{\mu_1}(w').$$
On the first hand, by \eqref{mean_rho}
$$\left\vert w'h''_{\mu_1}(w')\right\vert=\left\vert \sigma_1^2\int_{\mathbb{R}}\frac{-2w'}{(t-w')^3}d\rho_1(t)\right\vert\leq 2\left(\frac{\sigma_1^2}{\Im w'^2}+\frac{\mu_1(3)-2\mu_1(1)\mu_1(2)+\mu_1(1)^3}{\Im w'^3}\right),$$
and 
\begin{align*}
\left\vert \frac{H'_{3}(w)^2}{H_3(w)}\right\vert =&\left\vert \hat{F}_{M}(w)/w+\hat{F}'_{M}(w)\right\vert\cdot\left\vert 1+w\frac{\hat{F}'_{M}(w)}{\hat{F}_{M}(w)}\right\vert\\
=& \left\vert 2-\frac{\sigma_M^2}{w}+\tilde{\sigma}_M\left(\frac{m_{\tilde{\rho}}(w)}{w}+m'_{\tilde{\rho}}(w)\right)\right\vert\cdot\left\vert 2-u_3\right\vert\\
\leq&\left(2+\frac{\sigma_M^2}{\Im w}+\frac{2\tilde{\sigma}_M^2}{(\Im w)^2}\right)\cdot\left(2+\frac{\sigma_M^2}{\Im w}+\frac{\tilde{\sigma}_M^2+\sigma_M^4/2}{\Im(w)^2}\right)\leq \left(2+\frac{\sigma_M^2}{\Im w}+\frac{2\tilde{\sigma}_M^2+\sigma_M^4/2}{\Im(w)^2}\right)^2,
\end{align*}
which yields, together with the hypothesis $\vert w-\omega_3\vert\leq\frac{k_3\tilde{\sigma}_1}{\theta(k_3)}$ and the lower bound on $\Im w'$ obtained in \eqref{lower_bound_imaginary_H3},
$$\left\vert \frac{H'_{3}(w)^2}{H_3(w)}w'h''_{\mu_1}(w')\right\vert\leq  32\left(\frac{\sigma_1^2}{(k_3^2-4)\tilde{\sigma}_1^2}+\frac{2(\mu_1(3)-2\mu_1(2)+1)}{(k_3^2-4)^{3/2}\tilde{\sigma}_1^3}\right)\cdot \left(1+\frac{\sigma_M^2}{k_3\tilde{\sigma}_1}+\frac{4\tilde{\sigma}_M^2+\sigma_M^4}{k_3^2\tilde{\sigma}_1^2}\right)^2.$$
On the other hand, when $\vert w-\omega_3\vert\leq \frac{k_3\tilde{\sigma}_1}{2}$,
\begin{align*}
\left\vert H''_{3}(w)\right\vert =\left\vert 2\hat{F}_{M}'(w)+w\hat{F}''_{M}(w)\right\vert=&\left\vert 2+\tilde{\sigma}_M^2\left(2\int_{\mathbb{R}}\frac{1}{(w-t)^2}d\tilde{\rho}(t)-\int_{\mathbb{R}}\frac{2w}{(w-t)^3}d\tilde{\rho}(t)\right)\right\vert\\
\leq &2\left(1 +8\frac{m_4-2m_3m_2+m_2^2}{k_3^3\tilde{\sigma}_1^3}\right),
\end{align*}
and 
$$\left\vert h'_{\mu_1}(w')\right\vert=\left\vert \sigma_1^2\int_{\mathbb{R}}\frac{1}{(w'-t)^2}d\rho(t)\right\vert\leq \frac{4\sigma_1^2}{(k_3^2-4)\tilde{\sigma_1}^2},$$
which gives 
$$\left\vert H''_{3}(w)h'_{\mu_1}(w')\right\vert\leq  \frac{8\sigma_1^2}{(k_3^2-4)\tilde{\sigma_1}^2}\cdot\left(1 +8\frac{m_4-2m_3m_2+m_2^2}{k_3^3\tilde{\sigma}_1^3}\right).$$
Finally, for $w\in \mathbb{C}^+$ such that $\vert w-\omega_3\vert\leq \frac{k_3\tilde{\sigma}_1}{\theta(k_3)}$,
\begin{align*}
\vert K_{z}''(w)\vert\leq &  32\left(\frac{\sigma_1^2}{(k_3^2-4)\tilde{\sigma}_1^2}+\frac{2(\mu_1(3)-2\mu_1(2)+1)}{(k_3^2-4)^{3/2}\tilde{\sigma}_1^3}\right)\cdot \left(1+\frac{\sigma_M^2}{k_3\tilde{\sigma}_1}+\frac{4\tilde{\sigma}_M^2+\sigma_M^4}{k_3^2\tilde{\sigma}_1^2}\right)^2\\
&\hspace{2cm}+\frac{8\sigma_1^2}{(k_3^2-4)\tilde{\sigma_1}^2}\cdot\left(1 +8\frac{m_4-2m_3m_2+m_2^2}{k_3^3\tilde{\sigma}_1^3}\right)=L(k_3).
\end{align*}
\end{proof}
From the latter lemma, it is clear by the implicit function theorem that $(\omega_1(z),\omega_3(z))$, solution of $\Phi_{z}(\omega_1(z),\omega_3(z))=0$, can be extended around some point $z_0\in\mathbb{C}^+$ as long as $t(\Im\omega_3(z_0)/\tilde{\sigma}_1)<1$. Hence, as in Section \ref{Section:statement_deconvolution}, let us introduce $\xi_0=\inf(\xi\geq \xi_g, t(\xi)<1)$, where $\xi_g=\argmin_{[2,\infty[} g$. We describe in the following lemma how to concretely extend $\omega_3$ around some point $z_0$ satisfying $\Im \omega_3(z_0)/\tilde{\sigma}_1>\xi_0$. 
\begin{lemma}\label{extension_omega_3}
Suppose that $z_0\in\mathbb{C}^+$ is such that there exist $\omega_3\in \mathbb{C}^{+}$ with $\Im(\omega_3)/\tilde{\sigma}_1:=k_3>\xi_0$ and $K_{z_0}(\omega_3)=0$. Then, for all $z\in B(z_0,R(\xi))$ with 
\begin{equation}\label{definition_R}
R(k_3)=\frac{(1-t(k_3))\min\left(\frac{1-t(k_3)}{2L(k_3)},\frac{k_3\tilde{\sigma_1}}{4\theta(k_3)}\right)}{2\left(1+\frac{2\sigma_1^2}{\sqrt{k_3^2-4}\tilde{\sigma}_1}+\frac{\mu_1(3)-2\mu_1(2)+1}{(k_3^2-4)\tilde{\sigma}_1^2}\right)},
\end{equation}
with $\theta(k_3),\,L(k_3) $ respectively defined in \eqref{eq:expression_theta} and \eqref{eq:expression_L}, there exist $\omega_1(z),\omega_3(z)$ such that 
$$\Phi_{z}(\omega_1(z),\omega_3(z))=0,$$
and $\omega_3(z)\in B\left(\omega_3,\frac{k_3 \tilde{\sigma}_1}{4\theta(k_3)}\right)$. Moreover, the function $z\mapsto(\omega_1(z),\omega_3(z))$ is analytic, and for $z\in B(z_0,R(k_3))$,
$$K_{z}^{\circ n}(\omega_3)\xrightarrow[n\rightarrow\infty]{} \omega_3(z).$$
\end{lemma}
\begin{proof}
Set $r_{0}=\min\left(\frac{1-t(k_3)}{2L(k_3)},\frac{k_3\tilde{\sigma_1}}{4\theta(k_3)}\right)$. Then, from the bounds on $K_{z_0}(\omega_3)$ and on $K_{z_0}''(w)$ for $w\in\mathbb{C}^+$ such that $\vert w-\omega_3\vert\leq \frac{k_3\tilde{\sigma}_1}{\theta(k_3)}$ given in Lemma \ref{definition_domain_tildeh},
$$\vert K'_{z_0}(w)\vert\leq \vert K'_{z_0}(\omega_3)\vert+L(k_3)\vert w-\omega_3\vert\leq t(k_3)+\frac{1-t(k_3)}{2}\leq \frac{1+t(k_3)}{2}$$
for $w\in B(\omega_3,r_0)$. Since $t$ is decreasing, the hypothesis $k_3>\xi_0$ and the definition of $\xi_0$ yield that $\frac{1+t(k_3)}{2}<1$. Hence, $K_{z_0}$ is a contraction on $B(\omega_3,r_0)$ and, since $K_{z_0}(\omega_3)=\omega_3$,
\begin{equation}\label{inclusion_ball}
d(K_{z_0}(B(\omega_3,r_0),\partial B(\omega_3,r_0))>\frac{1-t(k_3)}{2}r_0.
\end{equation} 

Let us study the derivative of $K_z(w)$ with respect to $z$. First, since $K_z(w)=-zh_{\mu_1}(H_3(w)/z)$ and $h_{\mu_1}(w)=-1+\sigma_1^2\int_{\mathbb{R}}\frac{1}{t-w}d\rho(t)$,
\begin{align*}
\frac{\partial}{\partial z}K_{z}(w)=&-h_{\mu_1}(H_{3}(w)/z)+\frac{H_{3}(w)}{z}h_{\mu_1}'(H_3(w)/z)\\
=&1+\sigma_1^2\int_{\mathbb{R}}\frac{1}{H_{3}(w)/z-t}d\rho(t)+\frac{H_{3}(w)}{z}\int_{\mathbb{R}}\frac{\sigma_1^2}{(H_3(w)/z-t)^2}d\rho(t)\\
=&1+\sigma_1^2\int_{\mathbb{R}}\frac{2}{H_{3}(w)/z-t}d\rho(t)+\int_{\mathbb{R}}\frac{\sigma_1^2t}{(H_3(w)/z-t)^2}d\rho(t).
\end{align*}
On the other hand,
\begin{equation}\label{eq:difference_H3_z0}
\frac{H_3(w)}{z}-\frac{H_3(\omega_3)}{z_0}=\frac{z_0}{z}\frac{H_3(w)-H_3(\omega_3)}{z_0}+\frac{H_3(\omega_3)}{z_0^2}\frac{z_0}{z}(z_0-z).
\end{equation}
Assuming $\vert z-z_0\vert\leq \frac{\Im z_0}{2}$ and using \eqref{eq:bound_derivartive_H3} give then
\begin{equation}\label{eq:bound_difference_H3_z0}
\left \vert \frac{z_0}{z}\frac{H_3(w)-H_3(\omega_3)}{z_0}+\frac{H_3(\omega_3)}{z_0^2}\frac{z_0}{z}(z-z_0)\right\vert\leq 2\frac{\theta(k_3)}{2}\vert w-\omega_3\vert+2\left\vert\frac{H_3(\omega_3)}{z_0^2}\right\vert \cdot\vert z-z_0\vert.
\end{equation}
Since $K_{z_0}(\omega_3)=\omega_3$, $\Phi_{z_0}(\omega_1,\omega_3)=0$ with $\omega_1=H_3(\omega_3)/z$, which implies $\omega_3=-zh_{\mu_1}(\omega_1)$. Given that $H_3(\omega_3)=\omega_3\hat{F}_{M}(\omega_3)$ and $\hat{F}_{M}=\omega_3-\sigma_M^2+\tilde{\sigma}_Mm_{\tilde{\rho}}(\omega_3)$, we thus have
\begin{align*}
\left\vert \frac{H_{3}(\omega_3)}{z_0^2}\right\vert=&\left\vert h_{\mu_1}(\omega_1)^2\left(1-\frac{\sigma_M^2-\tilde{\sigma}_M^2m_{\tilde{\rho}}(\omega_3)}{\omega_3}\right)\right\vert\\
\leq& \left(1+\frac{\sigma_1^2}{k(k_3)\tilde{\sigma}_1}\right)^2\cdot\left(1+\frac{\sigma_M^2}{k_3\tilde{\sigma}_1}+\frac{\tilde{\sigma}_M^2}{k_3^2\tilde{\sigma}_1^2}\right).
\end{align*}
Hence, since $\vert w-\omega_3\vert\leq \frac{k_3\tilde{\sigma}_1}{4\theta(k_3)}$, for $z\in\mathbb{C}^+$ such that $\vert z-z_0\vert\leq \frac{k_3\tilde{\sigma}_1}{4\left(1+\frac{\sigma_1^2}{k(k_3)\tilde{\sigma}_1}\right)^2\cdot\left(1+\frac{\sigma_M^2}{k_3\tilde{\sigma}_1}+\frac{\tilde{\sigma}_M^2}{k_3^2\tilde{\sigma}_1^2}\right)}$, \eqref{eq:difference_H3_z0} together with \eqref{eq:bound_difference_H3_z0} yield that
$$\Im \frac{H_3(w)}{z}>k(k_3)\tilde{\sigma}_1-\frac{k_3}{2}\tilde{\sigma}_1>\frac{\sqrt{k_3^2-4}}{2}\tilde{\sigma}_1.$$
Therefore, for such $z$,
\begin{equation}\label{eq:bound_dz_Kz}
\left\vert \frac{\partial}{\partial z}K_z(w)\right\vert\leq 1+\frac{4\sigma_1^2}{\sqrt{k_3^2-4}\tilde{\sigma}_1}+\frac{4\sigma_1^2\rho(1)}{(k_3^2-4)\tilde{\sigma}_1^2}\leq 1+\frac{4\sigma_1^2}{\sqrt{k_3^2-4}\tilde{\sigma}_1}+\frac{4(\mu_1(3)-2\mu_1(2)+1)}{(k_3^2-4)\tilde{\sigma}_1^2}.
\end{equation}
Since $r_0\leq \frac{k_3\tilde{\sigma}_1}{4\theta(k_3)}$, the expression of $\theta$ implies that for
\begin{equation}
\vert z-z_0\vert<\frac{(1-t(k_3))r_0}{2\left(1+\frac{4\sigma_1^2}{\sqrt{k_3^2-4}\tilde{\sigma}_1}+\frac{4(\mu_1(3)-2\mu_1(2)+1)}{(k_3^2-4)\tilde{\sigma}_1^2}\right)}:=R(k_3),
\end{equation}
then we also have $\vert z-z_0\vert \leq\frac{k_3\tilde{\sigma}_1}{4\left(1+\frac{\sigma_1^2}{k(k_3)\tilde{\sigma}_1}\right)^2\cdot\left(1+\frac{\sigma_M^2}{k_3\tilde{\sigma}_1}+\frac{\tilde{\sigma}_M^2}{k_3^2\tilde{\sigma}_1^2}\right)}$, so that by \eqref{inclusion_ball} and \eqref{eq:bound_dz_Kz},
$$K_{z}(B(\omega_3,r_0))\subset B(\omega_3,r_0),$$
with a strict inclusion. Hence, by Denjoy-Wolf theorem, there exists $\omega_3(z)\in B(\omega_3,r_0)$ such that $K_{z}(\omega_3(z))=\omega_3(z)$, and 
$$K_{z}^{\circ n}(\omega_3)\xrightarrow[n\rightarrow\infty]{} \omega_3(z).$$
The analyticity of the function $z\mapsto \omega_3(z)$ is deduced by the implicit function theorem and the above bounds on $K'_{z}$. 
\end{proof}
An important property of the radius $R(\xi)$ is to be increasing in $\xi$, which reflects the fact that the subordination equation is more stable as the imaginary part of $z$ grows.
\begin{lemma}\label{increasing_R}
The function $\xi\mapsto R(\xi)$ is increasing from $[\xi_0,+\infty[$ to $[0,\infty[$.
\end{lemma}
\begin{proof}
By \eqref{definition_t}, \eqref{eq:expression_theta},  \eqref{eq:expression_L} and the fact that $\xi\mapsto k(\xi)$ is increasing on $[2,\infty[$, the functions $t(\xi),\,\theta(\xi)$ and $L(\xi)$ are decreasing functions of $\xi$. The result is then implied by the expression of $R$ in \eqref{definition_R}.
\end{proof}
We establish now a result similar to the one of \cite[Proposition 3.4]{ATV}, with slightly different hypothesis.
\begin{lemma}\label{convergence_negative_real}
Suppose that $z\in \mathbb{C}^+$ is such that $d(z,[0,+\infty[)> K_0$, with $K_0$ being the positive root of 
$$K^2/3-\sigma_1^2\frac{27\sigma_M^2+2K/3}{4}\left(1+\frac{27(a_3-2a_2a_1+a_1^3)(\sigma_M^2+2K/3)}{4K^2}\right)=0.$$
Then, $K_{z}^{\circ n}(z)$ converges to a solution $\omega_3(z)$ of the equation $K_{z}(w)=w$ as $n$ goes to infinity, and $\Im\omega_3(z)>2\tilde{\sigma}_1$.
\end{lemma}
\begin{proof}
The proof of this lemma is similar to the one of \cite[Proposition 3.4]{ATV}.
\end{proof}

We can now prove Theorem \ref{subordination_deconvolution_multi}. Recall from Section \ref{Section:statement_deconvolution} that $\xi_{0}$ is the unique positive root in $]\xi_g,+\infty[$ of the the equation
$$\xi_0=\inf\left(\xi\geq \xi_{g},\left(\frac{\sigma_1^2}{\tilde{\sigma}_1k(\xi)}+\frac{\tilde{\sigma}_1^2+\sigma_1^4/2}{k(\xi)^2\tilde{\sigma}_1^2}\right)\left(2+\frac{\sigma_M^2}{\xi\tilde{\sigma}_1}+\frac{\tilde{\sigma}_M^2+\sigma_M^4/2}{\xi^2\tilde{\sigma}_1^2}\right)<1\right),$$
and set $K=g(\xi_{0})$, where $g$ is defined in \eqref{definition_g}. Note that the latter definition yields (with $\tilde{\sigma}_M>\tilde{\sigma_1})$ that $k(\xi_{0})^2\geq (2+1/\xi_0^2)\geq\frac{9}{4}$ which then implies $\xi_{0}\geq \frac{3}{2}+\frac{2}{3}=\frac{13}{6}$.

\begin{proof}[Proof of Theorem \ref{subordination_deconvolution_multi}]
Let us fix $\eta> K$, and write $z_t=t+i\eta\tilde{\sigma}_1$ for $t\in\mathbb{R}$. Since $K=g(\xi_0)$ with $\xi_0>\xi_g$, $\xi=g^{-1}(\eta)$ is well-defined and $\xi>\xi_0$. We write 
$$I=\left\{t\in \mathbb{R}, \exists \omega_3(z_t)\in \mathbb{C}_{g^{-1}(\eta)\tilde{\sigma}_1}, K_{z_t}(\omega_3(z_t))=\omega_3(z_t) \right\}.$$
Let us show that $I=\mathbb{R}$. By Lemma \ref{convergence_negative_real}, if $t<-K_0$, then $K_{z_t}^{\circ n}(z_t)$ converges to a fixed point $\omega_3(z_t)$ of $K_{z_t}$ as $n$ goes to infinity and $\Im\omega_3(z_t)>2\tilde{\sigma}_1$.
Hence, writing $\Im\omega_3(z_t)=k_3\tilde{\sigma}_1$, by Lemma \ref{upper_bound_z_multi}, $\eta\leq g(k_3)$ and since $g$ is increasing on $[\xi_g,+\infty[$, $k_{3}\geq g^{-1}(\eta)$. Hence, there exist $K'$, such that $]-\infty,K']\subset I$, and $I$ is non void.

If $t\in I$, then there exists $\omega_3(z_t)$ such that $K_{z}(\omega_3(z_t))=\omega_3(z_t)$ and $\Im(\omega_3(z_t))\geq g^{-1}(\eta)\tilde{\sigma}_1> \xi_0\tilde{\sigma}_1$. Hence, by Lemma \ref{extension_omega_3}, for all $z'\in B(z_t,R(\Im(\omega_3(z_t))/\tilde{\sigma}_1))$, where $ R$ is defined in \eqref{definition_R}, there exists $\omega_3(z')$ such that $K_{z}(\omega_3(z'))=\omega_3(z')$. By Lemma \ref{increasing_R}, $R(\xi)$ is increasing in $\xi$, and $\Im \omega_3(z_t)\geq g^{-1}(\eta)\tilde{\sigma}_1$, thus $B(z_t,R(g^{-1}(\eta)))\subset B(z_t,R(\Im(\omega_3(z_t))/\tilde{\sigma}_1))$. Hence, considering $B(z_t,R(g^{-1}(\eta)))\cap \mathbb{R}+i\eta$ yields an open interval $I_t\subset \mathbb{R}$ such that for all $t'\in I_t$, there exists $\omega_3(z_{t'})\in B\left(\omega_3(z_t),\frac{\Im \omega_3(z_t)}{4\theta(k_3)}\right)$ fixed point of $K_{z_t'}$, and 
$$\omega_3(z_{t'})=\lim_{n\rightarrow\infty} K_{z_{t'}}^{\circ n}(\omega_3(z_t)).$$
Remark that \eqref{eq:expression_theta} yields $\frac{k_3 \tilde{\sigma}_1}{4\theta(k_3)}\leq \frac{k_3}{24}\tilde{\sigma}_1$, implying that $\Im \omega_3(z_{t'})\geq \frac{23}{24}\Im \omega_3(z_t)$. Since $\Im(\omega_3(z_t))>\xi_0\tilde{\sigma}_1>\frac{13}{6}\tilde{\sigma}_1 $, this implies that $\Im(\omega_3(z_t))>\frac{23\cdot 13}{24\cdot 6}\tilde{\sigma_1}>2\tilde{\sigma}_1$. Hence, by Lemma \ref{upper_bound_z_multi}, $\Im\omega_3(z_{t'})/\tilde{\sigma}_1\geq g^{-1}(\Im z_{t'}/\tilde{\sigma}_1)\geq g^{-1}(\eta)$. Hence, $I_{t}\subset I$ and thus $[t,t+R(g^{-1}(\eta))]\subset I$. The interval $I$ contains some interval $]-\infty,K]$ and for all $t\in I$, $[t,t+R(g^{-1}(\eta))]\subset I$, thus $I=\mathbb{R}$.

By the previous argument, $\omega_3(z)$ is defined on $\mathbb{C}_{K\tilde{\sigma_1}}$. Using then Lemma \ref{extension_omega_3} yields the local analyticity and the convergence result of the lemma. Finally, setting $\omega_1(z)=H_3(\omega_3(z))/z$ gives then a couple of analytic functions $(\omega_1(z),\omega_3(z))$ solution of $\Phi_z(\omega_1(z),\omega_3(z))=0$ for $z\in\mathbb{C}_{g(\xi_0)\tilde{\sigma}_1}$, which implies the first part of the theorem.
\end{proof}
\section{Intregration on the unitary group and Weingarten calculus}\label{Appendix:weingarten}
We prove here the integration formulas on the unitary group which are used in the manuscript. The goal is to integrate polynomials in the entries of a random unitary matrix with respect to the Haar measure. We only state the results for polynomials up to order six, which are the useful ones for our problems, and the tedious computations of this section are done using the very efficient software \cite{FuKoNe}. The fundamental ingredient of the proofs is the Weingarten calculus developed by  Collins and Sniady \cite{Col,ColSni}. In the following theorem, $U=(u_{ij})_{1\leq i,j\leq N}$ is a Haar unitary matrix.
\begin{theorem}[Weingarten calculus, \cite{Col}]\label{Weingarten_theorem}
Let $\vec{i},\vec{i}',\vec{j},\vec{j'}\in \mathbb{N}^r$ with $r\geq 1$. Then,
$$\int_{U_{N}}u_{i_{1}j_{1}}\ldots u_{i_rj_r}\bar{u}_{i'_{1}j'_{1}}\ldots \bar{u}_{i'_rj'_r}=\sum_{\substack{\sigma,\tau\in S_r\\ i\circ \sigma=i', j\circ\tau=\tau'}}W_{N,r}(\sigma\tau^{-1}),$$
where $S_{r}$ denotes the symmetric group of size $r$ and $W_{N,r}:S_{r}\rightarrow \mathbb{Q}$ is the Weingarten function whose values at $\sigma$ only depends on the cycle structure of the permutation. Moreover,
\begin{align*}
&W_{N,1}(\Id)=\frac{1}{N},\\
 &W_{N,2}(1^2)=\frac{1}{N^2(1-N^{-2})},\;W_{N,2}(2)=\frac{-1}{N^3(1-N^{-2})}\\
&W_{N,3}(1^3)=\frac{1-2N^{-2}}{N^3(1-N^{-2})(1-4N^{-2})},\; W_{N,3}(21)=\frac{-1}{N^4(1-N^{-2})(1-4N^{-2)}},\\
&\hspace{7cm} W_{N,3}(3)=\frac{2}{N^5(1-N^{-2})(1-4N^{-2})},
\end{align*}
where $(3^a2^b1^c)$ denotes a permutation with $a$ cycles of length $3$, $b$ cycles of length $2$ and $c$ cycles of length $1$.
\end{theorem}
Using the latter theorem, we prove the following asymptotic formulas for products of matrices $A$ and $UBU^{*}$.
\begin{lemma}\label{Weingarten_calculus}
Let $A,B\in \mathcal{M}_N(\mathbb{C})$ and $U\in U_{n}$ Haar unitary, and suppose that $A,B$ are diagonal. Then, $\mathbb{E}[UBU^{*}A]=\Tr(B)A$,
$$(1-1/N^2)\mathbb{E}(UBU^*AUBU^*)=\left(\Tr(A)\Tr(B^2)-\Tr(A)\Tr(B)^2+A\left(\Tr(B)^2-\frac{1}{N^2}\Tr(B^2)\right)\right),$$
and when $\Tr(B)=1$,
\begin{align*}
(1-1/N^2)(1-4/N^2)\mathbb{E}\left[UBU^{*}AUBU^{*}AUBU^{*}\right]
=&A^2\Big(1+(1+4/N^2)\Tr(B^3)/N^2-6/N^2\Tr(B^2)\Big)\\
+&A\Big(2(\Tr(B^2)-1)+4/N^2(\Tr(B^2)-\Tr(B^3))\Big)\\
+&\Big(\Tr(B^3)+\Tr(B^2)\Tr(A^2)+2-\Tr(A^2)-3\Tr(B^2)\Big).\\
\end{align*}
\end{lemma}
\begin{proof}
We only explain the proof of the second equality, since the proofs of the first and the third ones use similar pattern. Note first that $\mathbb{E}(UBU^*AUBU^*)$ commutes with $A$, and thus is diagonal when $A$ has distinct diagonal entries. By a continuity argument, $\mathbb{E}(UBU^*AUBU^*)$ is thus diagonal. Write $U=(u_{ij})_{1\leq i,j\leq N}$ and expand $ \mathbb{E}(UBU^*AUBU^*)_{ii}$ as 
\begin{align*}
\mathbb{E}(UBU^*AUBU^*)_{ii}=&\sum_{k,j,s=1}^N \mathbb{E}(u_{ik}B_{kk}\bar{u}_{jk}A_{jj}u_{js}B_{ss}\bar{u}_{is})\\
=&\sum_{k,j,s=1}^N B_{kk}A_{jj}B_{ss}\mathbb{E}(u_{ik}\bar{u}_{jk}u_{js}\bar{u}_{is}).
\end{align*}
Let $1\leq i,j\leq N$ and $1\leq k,s\leq N$. Then, by Theorem \ref{Weingarten_theorem} and summing on permutations of $S_2$,
$$\mathbb{E}(u_{ik}u_{js}\bar{u}_{is}\bar{u}_{jk})=\left\lbrace\begin{matrix}
-\frac{1}{N(N^2-1)}& \text{ if }& i\not=j, k\not=s\\
\frac{1}{N(N+1)}&\text{if}& i=j, k\not =s\text{ or } i\not=j, k=s\\
\frac{2}{N(N+1)}&\text{if}& i=j, k=s\\
\end{matrix}\right.$$
Hence, using the latter formula yields
\begin{align*}
\mathbb{E}(UBU^*AUBU^*)_{ii}=&\sum_{j\not =i}A_{jj}\left[\sum_{k\not =s}-\frac{1}{N(N^2-1)}B_{kk}B_{ss}+\sum_{k=1}^n \frac{1}{N(N+1)}B_{kk}^2\right]\\
&\quad +A_{ii}\left[\sum_{k\not =s}\frac{1}{N(N+1)}B_{kk}B_{ss}+\sum_{k=1}^n
\frac{2}{N(N+1)}B_{kk}^2\right]\\
=&(\Tr(A)-A_{ii}/N)\left[-\frac{1}{1-1/N^2}\Tr(B)^2+\Tr(B^2)\left(\frac{1}{1+1/N}+\frac{1}{N-1/N}\right)\right]\\
&\quad+A_{ii}\left[\frac{1}{1+1/N}\Tr(B)^2+\frac{1}{N+1}\Tr(B^2)\right]\\
=&\frac{1}{1-1/N^2}\left[\Tr(A)\Tr(B^2)-\Tr(A)\Tr(B)^2+A_{ii}\left(\Tr(B)^2-\frac{1}{N^2}\Tr(B^2)\right)\right]
\end{align*}
A similar computation yields the third equality. We used \cite{FuKoNe} to achieve the computation in the latter case.\end{proof}
Lemma \ref{Weingarten_calculus} directly yields formulas for expectation of trace of products. For two finite integer sequences  $s,s'$ of length $r\geq 1$, set
$$m_{A\ast B}(s,s')=\mathbb{E}\Tr(A^{s_1}UB^{s'_1}U^{*}\ldots A^{s_r}UB^{s'_r}U^{*}).$$
\begin{lemma}\label{trace_Weingarten}
Suppose that $A,B\in \mathcal{M}_{N}(\mathbb{C})$. Then, 
$$m_{A\ast B}(1,1)=\Tr(A)\Tr(B),$$
$$m_{A\ast B}(1^2,1^2)=\frac{1}{1-N^{-2}}\left[\Tr(A^2)\Tr(B)^2+\Tr(A)^2\Tr(B^2)-\Tr(A)^2\Tr(B)^2-\frac{1}{N^2}\Tr(A^2)\Tr(B^2)\right],$$
\begin{align*}
m_{A\ast B}(21,1^2)=\frac{1}{1-N^{-2}}\Big[\Tr(A^3)\Tr(B)^2+\Tr(A)\Tr(A^2)\Tr(B^2)&-\Tr(A)\Tr(A^2)\Tr(B)^2\\
&-\frac{1}{N^2}\Tr(A^3)\Tr(B^2)\Big],
\end{align*}
and when $\Tr(B)=1$,
\begin{align*}
m_{A\ast B}(1^3,1^3)=\frac{1}{(1-1/N^2)(1-4/N^2)}&\big(\Tr(B^3)+3\Tr(B^2)^2)\Var(\mu_A)\\
&+(\Tr(A^3)-3\Tr(A^2)+2\Tr(A)^3)+\tilde{\epsilon}_{N}\big),
\end{align*}
with 
\begin{align*}
\tilde{\epsilon}_{N}=\frac{6}{N^2}(\Tr(A^2)\Tr(B^2)-\Tr(B^2)\Tr(A^3)-\Tr(A^2)\Tr(B^3))+\frac{4}{N^4}\Tr(A^3)\Tr(B^3),
\end{align*}
and
\begin{align*}
m_{A\ast B}(1^3,21^2)=\frac{1}{(1-1/N^2)(1-4/N^2)}&\big(\Tr(B^4)+(\Tr(B^2)^2+2\Tr(B^3))\Var(\mu_A)\\
&+\Tr(B^2)(\Tr(A^3)-3\Tr(A^2)+2\Tr(A)^3)+\tilde{\epsilon}_{N}\big),
\end{align*}
with 
\begin{align*}
\tilde{\epsilon}_{N}=\frac{1}{N^2}\big[\Tr(A^3)(\Tr(B^4)-2\Tr(B^2)^2&-4\Tr(B^3))+\Tr(A^2)(2\Tr(B^2)^2\\
&-6\Tr(B^4)+4\Tr(B^2))\big]+\frac{1}{N^4}\Tr(A^3)\Tr(B^4).
\end{align*}
\end{lemma}
\section{Analysis on the unitary group}\label{Appendix:concentration_unitary_group}
We provide here concentration inequalities on the unitary group which imply all our concentration results concerning the Stieltjes transform. Proofs are adapted from Kargin's approach in \cite{Kar2} to get bounds only depending on first moments of the matrices involved.
\subsection{Poincaré inequality and concentrations results}
Several concentrations inequalities exist on the unitary group \cite{AnGuZe, BaEm}. In this paper, we only use Poincaré inequality, which has the fundamental property of having an error term which is averaged on the unitary group. Poincaré inequalities exist on every compact Riemaniann manifolds without boundary, for which the Laplacian operator has a discrete spectrum.
\begin{theorem}[Poincaré inequality]
Suppose that $M$ is a compact manifold without boundary and with volume form $\mu$, and let $\lambda_1>0$ be the first non-zero eigenvalue of the Laplacian on $M$. Then, for all $f\in C^{2}(M)$ such that $\int_{M}fd\mu=0$,
$$\int_{M}\vert f\vert^2d\mu\leq \frac{1}{\lambda_1}\int_{M}\Vert \nabla f\Vert^2d\mu.$$ 
\end{theorem}
Proof of this theorem is a direct consequence of the integration by part formula on $M$. In the case of the unitary group $U_N$ the spectrum of the Laplacian can be explicitly computed using the representation theory of the group (see \cite{Hum}), and the first eigenvalue of the Laplacian is simply equal to $N$. Hence, we deduce from Poincaré inequality the following concentration inequality for the unitary group.
\begin{corollary}[Poincaré inequality on $U_N$]\label{poincare_inequality}
For all $f\in C^{2}(U_{N})$ such that $\int_{U_{N}}fd\mu=0$, where $\mu$ denotes the Haar measure on $U_{N}$,
$$\int_{U_{N}}\vert f\vert^2d\mu\leq \frac{1}{N}\int_{U_{N}}\Vert \nabla f\Vert^2d\mu.$$ 
\end{corollary}

In the sequel, the functions $f$ we will studied are traces of matrices involved the various resolvents of the manuscript. We will use several times the generalized matrix H\"{o}lder inequality for Schatten p-norms. Recall that the Schatten $p$-norm of a matrix $X\in \mathcal{M}_{N}(\mathbb{C})$ is defined by
$$\Vert X\Vert_{p}=[N\Tr((X^*X)^{p/2})]^{1/p}.$$
Then, if $X_{1},\ldots X_{k}\in \mathcal{M}_{N}(\mathbb{C})$ and $\alpha_1,\ldots,\alpha_k\in [1,+\infty]$, then
\begin{equation}\label{holder_inequality}
\Vert X_{1}\ldots X_{k}\Vert_{r}\leq \prod_{i=1}^k\Vert X_{i}\Vert_{\alpha_i},
\end{equation}
where $\frac{1}{r}=\sum_{i=1}^k\frac{1}{\alpha_i}$. Remark that the matrix Holder is not a trivial consequence of the usual H\"{o}lder inequality, and its proof is quite involved (see \cite[7.3]{Ser}).

\subsection{Application to the additive convolutions}
For $H=UAU^*+B$, $z\in\mathbb{C}^+$ and $T\in \mathcal{M}_{N}(\mathbb{C})$, set $G_H=(H-z)^{-1}$ and define the function $f_{T}(z)=\Tr(T(H-z)^{-1})=\Tr(TG_H)$. In the following lemmas, we use the convention $\Tr(\vert T\vert^\infty)^{1/\infty}=\Vert T\Vert_{\infty}$ for $T\in\mathcal{M}_{N}(\mathbb{C})$.
\begin{lemma}\label{concentration_result_additive}
For $z\in\mathbb{C}^+$ with $\eta=\Im(z)$ and for $T\in \mathcal{M}_N(\mathbb{C})$,
$$\mathbb{E}\left(\vert f_{T}(z)-\mathbb{E}(f_{T}(z))\vert^{2}\right)\leq \frac{4\Tr(A^\alpha)^{2/\alpha}\Tr(\vert T\vert^\beta)^{2/\beta}}{\eta^{4}N^{2}},$$
where $\frac{1}{\alpha}+\frac{1}{\beta}=\frac{1}{2}$ with $\alpha,\beta\in[2,\infty]$.
\end{lemma}
\begin{proof}
By \eqref{poincare_inequality}, for any function $f$ with zero mean which is $C^2$ on $U_N$, $\mathbb{E}(\vert f\vert ^2)\leq \frac{1}{N}\mathbb{E}(\Vert \nabla f\Vert ^2)$. Let us apply this to the map $f_{T}$. Since $d_{X}(X-z)^{-1}=(X-z)^{-1}X(X-z)^{-1}$, applying the chain rule for $f_T$ at $U\in U_N$ yields for $X$ anti-Hermitian 
$$ \nabla_{U}f_T(X)=\Tr(TG_H[X,\tilde{A}]G_H)=\Tr([\tilde{A},G_HTG_H]X),$$
where $\tilde{A}=UAU^*$. Hence,
$$\Vert \nabla_{U}f_T\Vert_{2}=\frac{1}{N}\Vert [\tilde{A},G_HTG_H]\Vert_{2}\leq \frac{2}{N\eta^2}\Vert A\Vert_{\alpha}\Vert T\Vert_{\beta}$$
with $\frac{1}{\alpha}+\frac{1}{\beta}=\frac{1}{2}$, where we applied matrix Hölder inequality in the last inequality . Therefore, 
\begin{align*}
\mathbb{E}\Vert \nabla_{U}f_T\Vert_{2}^2\leq& \frac{4}{N^2\eta^4}\Vert A\Vert_{\alpha}^2\Vert T\Vert_{\beta}^2\leq\frac{4\Tr(A^\alpha)^{2/\alpha}\Tr(\vert T\vert^\beta)^{2/\beta}}{N\eta^4},
\end{align*}
so that \eqref{poincare_inequality} yields
$$\Var(f_T)\leq \frac{4\Tr(A^\alpha)^{2/\alpha}\Tr(\vert T\vert^\beta)^{2/\beta}}{N^2\eta^4}.$$
\end{proof}
\begin{lemma}\label{concentration_result_additive_z}
For $z\in\mathbb{C}^+$ with $\eta=\Im(z)$ and $\Tr(B)=0$,
$$\Var(zm_{H})\leq \frac{8}{N^2\eta^2}\left(\Tr(A^2)+\frac{\Tr(B^2)\Tr(A^2)+\Tr(A^4)}{\eta^2}\right),$$
and for $T\in \mathcal{M}_N(\mathbb{C})$,
\begin{align*}
&\Var(zf_{T})\\
\leq &\frac{12}{N^2\eta^2}\Bigg(\Tr(\vert T\vert^2)\Tr(A^2)+\frac{\mathbb{E}\left(\Tr((B\tilde{A}^2B)^{\alpha_1/2})\right)^{2/\alpha_1}\Tr( \vert T\vert^{\beta_1})^{2/\beta_1}+\Tr(A^{2\alpha_2})^{2/\alpha_2}\Tr(\vert T\vert^{\beta_2})^{2/\beta_2})}{\eta^2}\Bigg)
\end{align*}
for any $\alpha_1,\beta_1,\alpha_2,\beta_2\in[2,\infty]$ satisfying 
$$\frac{1}{\alpha_1}+\frac{1}{\beta_1}=\frac{1}{\alpha_2}+\frac{1}{\beta_2}=\frac{1}{2}.$$ 
\end{lemma}
\begin{proof}
Let us first prove the second statement. As in the latter lemma, taking the derivative of $zf_{T}$ at $U\in U_N$ yields for $X$ anti-Hermitian 
\begin{align*}
\nabla_{U}(zf_{T})(X)=&z\Tr(TG_H[X,\tilde{A}]G_H)\\
=&\Tr([\tilde{A},zG_HTG_H]X)\\
=&\Tr\left(\left[-\tilde{A}TG_H+G_HT\tilde{A}+\tilde{A}(B+\tilde{A})G_HTG_H-G_HTG_H(B+\tilde{A})\tilde{A}\right]X\right),
\end{align*}
where $\tilde{A}=UAU^*$ and we used the equality $zG_H=-1+HG_H$. Hence,
\begin{align*}
\Vert \nabla_{U}zf_T\Vert^{2}\leq&\frac{1}{N^2}\left(2\Vert T\tilde{A}\Vert_{2}+2\Vert  \tilde{A}BG_HTG_H\Vert_{2}+2\Vert \tilde{A}^2G_HTG_H\Vert_{2}\right)^2\\
\leq & \frac{12}{N^2}(\Vert T\tilde{A}\Vert_{2}^2+\Vert  \tilde{A}BG_HTG_H\Vert_{2}^2+\Vert \tilde{A}^2G_HTG_H\Vert_{2}^2).
\end{align*}
First, $\mathbb{E}(\Vert T\tilde{A}\Vert_{2}^2)=N\mathbb{E}(\Tr(TT^*\tilde{A}^2))=N\Tr(TT^*)\Tr(A^2)$ by Lemma \ref{trace_Weingarten}. Then, we apply the matrix Hölder inequality \eqref{holder_inequality} and then the usual H\"{o}lder inequality to get
\begin{align*}
\mathbb{E}(\Vert\tilde{A}BG_HTG_H\Vert^2)\leq \frac{1}{\eta^4}\mathbb{E}(\Vert\tilde{A}B\Vert_{\alpha_1}^{2}\Vert T\Vert_{\beta_1}^{2})\leq &\frac{1}{\eta^4}\mathbb{E}(\Vert\tilde{A}B\Vert_{\alpha_1}^{\alpha_1})^{\frac{2}{\alpha_1}}\Vert T\Vert_{\beta_1}^{2}\\
\leq& \frac{N}{\eta^4}\mathbb{E}\left(\Tr((B\tilde{A}^2B)^{\alpha_1/2})\right)^{2/\alpha_1}\Tr( \vert T\vert^{\beta_1})^{2/\beta_1},
\end{align*}
and 
$$\mathbb{E}(\Vert\tilde{A}^2G_HTG_H\Vert^2)\leq \frac{\Vert A^2\Vert_{\alpha_2}^{2}\Vert T\Vert_{\beta_2}^{2}}{\eta^4}\leq \frac{N}{\eta^4}\Tr(A^{2\alpha_2})^{2/\alpha_2}\Tr(\vert T\vert^{\beta_2})^{2/\beta_2})$$
for any $\alpha_1,\beta_1,\alpha_2,\beta_2\in [2,\infty]$ such that $\frac{1}{\alpha_1}+\frac{1}{\beta_1}=\frac{1}{\alpha_2}+\frac{1}{\beta_2}=\frac{1}{2}$. Hence, using Poincaré inequality yields
\begin{align*}
&\Var(zf_{T})\\
\leq &\frac{12}{N^2\eta^2}\Bigg(\Tr(\vert T\vert^2)\Tr(A^2)+\frac{\mathbb{E}\left(\Tr((B\tilde{A}^2B)^{\alpha_1/2})\right)^{2/\alpha_1}\Tr( \vert T\vert^{\beta_1})^{2/\beta_1}+\Tr(A^{2\alpha_2})^{2/\alpha_2}\Tr(\vert T\vert^{\beta_2})^{2/\beta_2})}{\eta^2}\Bigg)
\end{align*}
for such $\alpha_1,\beta_1,\alpha_2,\beta_2$. The proof of the first inequality is similar, since 
\begin{align*}
\nabla_{U}(zm_H)(X)=z\Tr(G_H[X,\tilde{A}]G_H)=&z\Tr([\tilde{A},G_H^2]X)\\
=&-\Tr([\tilde{A},G_H]X)+\Tr((\tilde{A}(B+\tilde{A})G_H^2-G_H^2(B+\tilde{A})\tilde{A})X),
\end{align*}
which yields 
$$\mathbb{E}\Vert \nabla_{U}zm_H\Vert^2\leq \frac{8}{N^2}\left(\frac{\mathbb{E}\Vert \tilde{A}\Vert_2^2}{\eta^2}+\frac{\mathbb{E}\Vert (B+\tilde{A})\tilde{A}\Vert^2_{2}}{\eta^4}\right).$$
First $\Vert \tilde{A}\Vert_2^2=N\Tr(A^2)$, and then
\begin{align*}
\mathbb{E}\Vert (B+\tilde{A})\tilde{A}\Vert_{2}^2=N\mathbb{E}\left[\Tr\left((B+\tilde{A})\tilde{A}^2(B+\tilde{A})\right)\right]=&N\mathbb{E}\left[\Tr(B^2\tilde{A}^2)+\Tr(\tilde{A}^4)+2\Tr(B\tilde{A}^3)\right]\\
=&N\left(\Tr(A^2)\Tr(B^2)+\Tr(A^4)\right),
\end{align*}
where we used Lemma \ref{trace_Weingarten} and $\Tr(B)=0$ on the last equality. The result is then deduced using Poincaré inequality.
\end{proof}
We give a similar result when the matrix $T$ of the latter lemma also depends on $UAU^*$.
\begin{lemma}\label{concentration_result_additive_tilde}
Let $z\in\mathbb{C}^+$ and for $T\in \mathcal{M}_N(\mathbb{C})$ set $\tilde{f}_{T}=\Tr(TUAU^*G_{H})$. Then,
\begin{align*}
\mathbb{E}\left(\vert \tilde{f}_{T}(z)-\mathbb{E}(\tilde{f}_{T}(z))\vert^{2}\right)&\leq \frac{4}{N^2\eta^4}\Big(\eta^2\left(\Tr(\vert T\vert^2)\Tr(A^2)+\sqrt{\Tr(\vert T\vert^4)\Tr(A^4)}\right)\\
&+2\sqrt{\Tr(A^4)m_{\vert T\vert^2\ast A^2}(1^2,1^2)}\Big),
\end{align*}
with the formula for $m_{\vert T\vert^2\ast A^2}(1^2,1^2)$ given in Lemma \ref{trace_Weingarten}.
\end{lemma}
\begin{proof}
Consider the map $\tilde{f}_{T}: U\mapsto \Tr(TUAU^{*}G_{H})$. Then, writing $\tilde{A}=UAU^*$,
$$\nabla_{U}\tilde{f}_{T}(X)=\Tr(T[X,\tilde{A}]G_{H}+T\tilde{A}G_{H}[X,\tilde{A}]G_{H})=\Tr([\tilde{A},G_{H}T]X)+\Tr([\tilde{A},G_{H}T\tilde{A}G_{H}]X).$$
Hence, by H\"{o}lder inequality,
\begin{align*}
\Vert \nabla_{U}\tilde{f}_{T}\Vert_{2}^2\leq&\frac{1}{N^2}(\Vert \tilde{A}G_{H}T\Vert_2+\Vert G_{H}T\tilde{A}\Vert_2+\Vert \tilde{A}G_{H}T\tilde{A}G_{H}\Vert_2+\Vert G_HT\tilde{A}G_{H}\tilde{A})\Vert_2)^2\\
\leq&\frac{4}{\eta^2N^2}\Vert T\Vert^2_{4}\Vert A\Vert_{4}^2+\frac{4}{\eta^2N^2}\Vert T\tilde{A}\Vert_{2}^2+\frac{8}{\eta^4N^2}\Vert T\tilde{A}\Vert_4^2\Vert A\Vert_{4}^2.
\end{align*} 
Integrating on the unitary group yields then
$$\mathbb{E}\Vert \nabla_{U}\tilde{f}_{T}\Vert_{2}^2\leq \frac{4\sqrt{\Tr(\vert T\vert^4)\Tr(A^4)}+4\Tr(\vert T\vert^2)\Tr(A^2)}{N\eta^2}+\frac{8\mathbb{E}\left[\Tr(\vert T\tilde{A}\vert^4)\right]^{1/2}\Tr(A^4)^{1/2}}{N\eta^4}.$$
Remark that $\mathbb{E}\left[\Tr(\vert T\tilde{A}\vert^4)\right]=\mathbb{E}\left[\Tr(T\tilde{A}^2T^*T\tilde{A}^2T^*)\right]=m_{\vert T\vert^2\ast A^2}(1^2,1^2)$, whose formula is given by Lemma \ref{trace_Weingarten}. The results then follows by Poincaré inequality.
\end{proof}

\subsection{Application to the multiplicative convolution}
We now state the concentration results for the multiplicative case. As in the additive case, for $M=A^{1/2}UBU^*A^{1/2}$ we write $f_T(z)=\Tr(TG_{M}(z))$, with $G_M(z)=(M-z)^{-1}$.
\begin{lemma}\label{concentration_result_multi}
For $z\in\mathbb{C}^+$ with $\eta=\Im(z)$ and for $T\in \mathcal{M}_N(\mathbb{C}),$
$$\mathbb{E}\left(\vert f_{T}(z)-\mathbb{E}(f_{T}(z))\vert^{2}\right)\leq \frac{4}{N^2\eta^4}\min\left(K\Vert T\Vert_{\infty}^2,\Vert B\Vert_{\infty}^2\Vert T\Vert_{\beta}^2\Vert A\Vert_{\alpha}^2\right).$$
with $K=\min\left(Tr(B^2)\Vert A\Vert_{\infty},\sqrt{\Tr(A^2)m_{A\ast B^2}(1^2,1^2)}\right)$, and $\alpha,\beta>0$ such that $\frac{1}{\alpha}+\frac{1}{\beta}=\frac{1}{2}$.
\end{lemma}
\begin{proof}
Like in the previous lemmas, the aim is to bound the derivative of the map $f_{T}:U\mapsto \Tr(T(z-A^{1/2}UBU^*A^{1/2}))$ (we drop the dependence in $z$ for $f_T$). Using the chain rule, we get 
$$\nabla_{U}f_{T}(X)=\Tr(TG_MA^{1/2}[X,\tilde{B}]A^{1/2}G_M)=\Tr([\tilde{B},A^{1/2}G_MTG_MA^{1/2}]X),$$
and with $\tilde{B}=UBU^*$. Hence, for all $U\in U_{N}$,
\begin{align*}
\Vert \nabla_{U}f_{T}\Vert_{2}\leq \frac{1}{N}\Vert [ \tilde{B},A^{1/2}G_MTG_MA^{1/2}]\Vert_{2}\leq \frac{2}{N}\Vert \tilde{B}A^{1/2}G_MTG_MA^{1/2}\Vert_{2},
\end{align*}
and we deduce that 
$$\mathbb{E}\left( \Vert \nabla_{U}f_{T}\Vert_{2}^2\right)\leq \frac{4}{N}\mathbb{E}\left(\Tr(A^{1/2}\tilde{B}^2A^{1/2}G_MTG_MAG_M^{*}T^*G_M^*)\right).$$
Then, either
$$\mathbb{E}\left( \Vert \nabla_{U}f_{T}\Vert_{2}^2\right)\leq 4\frac{\Vert A\Vert_{\infty}\Vert T\Vert_{\infty}^2}{N\eta^4}\mathbb{E}(\Tr(A\tilde{B}^2))\leq 4\frac{\Tr(B^2)\Vert A\Vert_{\infty}\Vert T\Vert_{\infty}^2}{N\eta^4},$$
where we used Lemma \ref{trace_Weingarten} and $\Tr(A)=1$ on the last inequality, or by applying the matrix Hölder's inequality,
$$\mathbb{E}\left( \Vert \nabla_{U}f_{T}\Vert_{2}^2\right)\leq 4\frac{\Vert B\Vert_{\infty}^2\Vert T\Vert_{\beta}^2\Vert A^{1/2}\Vert_{\alpha}^4}{N\eta^4}\leq 4\frac{\Vert B\Vert_{\infty}^2\Vert T\Vert_{\beta}^2\Vert A\Vert_{\alpha/2}^2}{N\eta^4},$$
for $\alpha,\beta>0$ such that $\frac{2}{\alpha}+\frac{1}{\beta}=\frac{1}{2}$. To get a bound in terms of moments of $A$, we used Cauchy-Schwartz inequality to get 
\begin{align*}
\mathbb{E}\left( \Vert \nabla_{U}f_{T}\Vert_{2}^2\right)\leq& 4\frac{\Vert T\Vert_{\infty}^2}{N\eta^4}\mathbb{E}\left(\sqrt{\Tr(A\tilde{B}^2A\tilde{B}^2)}\sqrt{\Tr(A^2)}\right)\\
\leq &\frac{4\Vert T\Vert_{\infty}^2}{N\eta^4}\sqrt{\mathbb{E}(\Tr(A\tilde{B}^2A\tilde{B}^2))}\sqrt{\Tr(A^2)}\\
\leq&\frac{4\Vert T\Vert_{\infty}^2\sqrt{\Tr(A^2)}}{N\eta^4}\sqrt{m_{A\ast B^2}(1^2,1^2)}.\\
\end{align*}
Using Poincaré inequality on the unitary group concludes the proof. 
\end{proof}

\begin{lemma}\label{concentration_result_multi_z}
For $z\in\mathbb{C}^+$ with $\eta=\Im(z)$ and for $T\in \mathcal{M}_N(\mathbb{C})$,
$$\mathbb{E}\left(\vert zf_{T}(z)-\mathbb{E}(zf_{T}(z))\vert^{2}\right)
\leq \frac{8\Vert T\Vert_{\infty}^2\Vert A\Vert_{\infty}}{\eta^2N^2}(\Tr(B^2)+m_{A\ast B}(1^3,21^2)/\eta^2).$$
\end{lemma}
\begin{proof}
As in the previous lemma, we have
$$\nabla_{U}f_{T}(X)=\Tr(TG_MA^{1/2}[X,\tilde{B}]A^{1/2}G_M)=\Tr([\tilde{B},A^{1/2}G_MTG_MA^{1/2}]X),$$
with $\tilde{B}=UBU^*$. Moreover, for all $U\in U_{N}$,
\begin{align*}
z\tilde{B}A^{1/2}G_MTG_MA^{1/2}=&\tilde{B}A^{1/2}(-1+A^{1/2}\tilde{B}A^{1/2}G_M)TG_MA^{1/2}\\
=&-\tilde{B}A^{1/2}TG_MA^{1/2}+\tilde{B}A\tilde{B}A^{1/2}G_MTG_MA^{1/2},
\end{align*}
and likewise
$$zA^{1/2}G_MTG_MA^{1/2}\tilde{B}=-A^{1/2}G_MTA^{1/2}\tilde{B}+A^{1/2}G_MTG_MA^{1/2}\tilde{B}A\tilde{B}.$$
Hence,
$$\mathbb{E}(\Vert z\nabla_{U}f_{T}(X)\Vert^2)\leq \frac{8}{N}\left(\frac{\Vert T\Vert_{\infty}^2\Vert A\Vert_{\infty}}{\eta^2}\mathbb{E}\left(\Tr(A\tilde{B}^2)\right)+\frac{\Vert A\Vert_{\infty}\Vert T\Vert_{\infty}^2\mathbb{E}\Tr(A\tilde{B}A\tilde{B}^2A\tilde{B})}{\eta^4}\right).$$
By Lemma \ref{trace_Weingarten} and $\Tr(A)=1$, $\mathbb{E}\Tr(A\tilde{B}^2)=\Tr(A)\Tr(B^2)=\Tr(B^2)$, and by Lemma \ref{trace_Weingarten} we also have 
$\mathbb{E}\Tr(A\tilde{B}A\tilde{B}^2A\tilde{B})=m_{A\ast B}(1^3,21^2)$. Poincaré inequality on the unitary group concludes then the proof of the lemma.
\end{proof}
In the simpler case where $T=\Id$ we can get a better bound. This improvement is important, since this gives the main contribution of our concentration bounds as $N$ goes to infinity.
\begin{lemma}\label{concentration_result_multi_simple}
For $M=A^{1/2}UBU^*A^{1/2}$ and $z\in\mathbb{C}^+$ with $\eta=\Im(z)$,
\begin{align*}
\mathbb{E}\left(\tilde{m}_{M}(z)-\mathbb{E}(\tilde{m}_{M}(z))\vert^{2}\right)\leq&\frac{8}{N^2}\Bigg(\frac{\sqrt{\Tr(A^2)m_{A\ast B_0^2}(1^2,1^2)}}{\eta^2}\\
&+\frac{\Vert A\Vert_{\infty}}{\eta^4}\left(m_{A\ast B}(1^3,21^2)-2m_{A\ast B}(1^3,1^3)+m_{A\ast B}(21,1^2)\right)\Bigg),
\end{align*}
where $B_0=B-\Tr(B)=B-\Id$.
\end{lemma}
\begin{proof}
We have
$$\nabla_{U}m_{M}(X)=\Tr(G_MA^{1/2}[X,\tilde{B}]A^{1/2}G_M)=\Tr([\tilde{B},A^{1/2}G_M^2A^{1/2}]X),$$
with $\tilde{B}=UBU^*$. Since $\Id$ commutes with $A^{1/2}G_M^2A^{1/2}$, we can replace $B$ by $B_{0}=B-\Tr(B)$ in the latter equality. Moreover, for all $U\in U_{N}$,
\begin{align*}
z\tilde{B_{0}}A^{1/2}G_M^2A^{1/2}=&\tilde{B}_{0}A^{1/2}(-1+A^{1/2}\tilde{B}A^{1/2}G_M)G_MA^{1/2}\\
=&-\tilde{B}_{0}A^{1/2}G_MA^{1/2}+\tilde{B}_{0}A\tilde{B}A^{1/2}G_M^2A^{1/2},
\end{align*}
and likewise
$$zA^{1/2}G_M^2A^{1/2}\tilde{B}_0=-A^{1/2}G_MA^{1/2}\tilde{B}_{0}+A^{1/2}G_M^2A^{1/2}\tilde{B}A\tilde{B}_{0}.$$
Hence,
$$\Vert z\nabla_{U}m_{M}(X)\Vert_{2}^2\leq \frac{8}{N^2}\left(\Vert A^{1/2}G_MA^{1/2}\tilde{B}_{0}\Vert_{2}^2+\Vert \tilde{B}_{0}A\tilde{B}A^{1/2}G_M^2A^{1/2}\Vert_2^2\right).$$
By the matrix Holder inequality \eqref{holder_inequality} with $\alpha=\beta=4$,
$$\Vert A^{1/2}G_MA^{1/2}\tilde{B}_{0}\Vert_{2}^2\leq \frac{N}{\eta^2}\sqrt{\Tr(A^2)\Tr(A\tilde{B}_0^2A\tilde{B}_{0}^2)},$$
and, using $\Tr(B)=1$,
$$\Vert \tilde{B}_{0}A\tilde{B}A^{1/2}G_M^2A^{1/2}\Vert_2^2\leq \frac{N\Vert A\Vert_\infty}{\eta^4}\Tr(A\tilde{B}A\tilde{B}_{0}^2A\tilde{B})\leq \frac{N\Vert A\Vert_\infty}{\eta^4}\Tr(A\tilde{B}A(B-\Id)^2A\tilde{B}).$$
Hence, after integration on the unitary group, and using the classical Holder inequality,
\begin{align*}
\mathbb{E}\Vert z\nabla_{U}m_{M}\Vert^2\leq& \frac{8}{N}\Bigg(\frac{\sqrt{\Tr(A^2)m_{A\ast (B-1)^2}(1^2,1^2)}}{\eta^2}\\
&\hspace{2cm}+\frac{\Vert A\Vert_{\infty}}{\eta^4}\left(m_{A\ast B}(1^3,21^2)-2m_{A\ast B}(1^3,1^3)+m_{A\ast B}(21,1^2)\right)\Bigg)
\end{align*}
Using Poincaré inequality on the unitary group and using that $\tilde{m}_M(z)=1+zm_M(z)$ concludes then the proof of the lemma.
\end{proof}
\begin{lemma}\label{concentration_result_multi_B}
For $z\in\mathbb{C}^+$ with $\eta=\Im(z)$ and for $T\in \mathcal{M}_N(\mathbb{C})$ normal, then, writing $\tilde{T}=UTU^{*}$,
\begin{align*}
\mathbb{E}(\vert \Tr(A^{1/2}&\tilde{T}A^{1/2}G_{M})-\mathbb{E}\Tr(A^{1/2}U\tilde{T}U^{*}A^{1/2}G_{M})\vert^2)\\
\leq& \frac{8\Vert A\Vert_{\infty}}{N^2\eta^2}\left(\mathbb{E}\Tr(A^{1/2}\vert \tilde{T}^2\vert A^{1/2})+\frac{(\mathbb{E}\Tr( (A^{1/2}\vert\tilde{T}\vert A^{1/2})^\alpha))^{2/\alpha}(\mathbb{E}\Tr((A^{1/2}UB^2U^{*}A^{1/2})^{\beta/2}))^{2/\beta}}{\eta^2}\right)
\end{align*}
for all $\alpha,\beta>1$ satisfying $\frac{1}{\alpha}+\frac{1}{\beta}=\frac{1}{2}$.
\end{lemma}
\begin{proof}
The first part of the lemma is a direct adaptation of the proof of Lemma \ref{concentration_result_multi} with the Hölder inequality 
$$\Vert A^{1/2}G_{M} TG_{M}A^{1/2}\tilde{B}\Vert_{2}\leq \Vert T\Vert_{\alpha}\Vert A^{1/2}\tilde{B}\Vert_{\beta}\Vert A\Vert_{\infty}^{1/2}$$ 
for all $\alpha,\beta>1$ satisfying $\frac{1}{\alpha}+\frac{1}{\beta}=\frac{1}{2}$. In view of applying the same method for the second part, we compute the derivative of the map $f_T:U\mapsto \Tr(A^{1/2}UTU^{*}A^{1/2}G_{M})$, which gives
\begin{align*}
\nabla_{U}f_T(X)=&\Tr([X,\tilde{T}]A^{1/2}G_{M}A^{1/2})+\Tr(A^{1/2}\tilde{T}A^{1/2}G_{M}A^{1/2}[X,\tilde{B}]A^{1/2}G_{M})\\
=&\Tr\left(([\tilde{T},A^{1/2}G_{M}A^{1/2}]+[\tilde{B},A^{1/2}G_{M}A^{1/2}\tilde{T}A^{1/2}G_{M}A^{1/2}])X\right).
\end{align*}
Hence,
$$N\Vert \nabla_{U}f\Vert_{2}\leq 2\Vert A^{1/2}G_{M}A^{1/2}\tilde{T}\Vert_{2}+2\Vert A^{1/2}G_{M}A^{1/2}\tilde{T}A^{1/2}G_{M}A^{1/2}\tilde{B}\Vert_{2}.$$
Using Holder inequality yields then
$$N\Vert \nabla_{U}f\Vert_{2}\leq \frac{2\Vert A^{1/2}\Vert_{\infty}\Vert A^{1/2}\tilde{T}\Vert_{2}}{\eta}+\frac{2\Vert A^{1/2}\Vert_{\infty}\Vert A^{1/2}\tilde{T}A^{1/2}\Vert_{\alpha}\Vert A^{1/2}\tilde{B}\Vert_{\beta}}{\eta^2},$$
for any $\alpha,\beta>1$ such that $\frac{1}{\alpha}+\frac{1}{\beta}=\frac{1}{2}$. Hence,
\begin{align*}
\Vert \nabla_{U}f\Vert_{2}^{2}\leq&\frac{8\Vert A\Vert_{\infty}}{N^2\eta^2} \left(\Vert A^{1/2}\tilde{T}\Vert_{2}^2+\frac{\Vert A^{1/2}\tilde{T}A^{1/2}\Vert_{\alpha}^{2}\Vert A^{1/2}\tilde{B}\Vert_{\beta}^{2}}{\eta^2}\right)\\
\leq&\frac{8\Vert A\Vert_{\infty}}{N\eta^2}\left(\Tr(A^{1/2}\tilde{T}^2A^{1/2})+\frac{\Tr( (A^{1/2}\vert \tilde{T}\vert A^{1/2})^\alpha)^{2/\alpha}\Tr((A^{1/2}B^2A^{1/2})^{\beta/2})^{2/\beta}}{\eta^2}\right),
\end{align*}
where we used that $\Vert ATA\Vert_{\alpha}\leq \Vert A\vert T\vert A\Vert_{\alpha}$ when $T$ is normal. Integrating on $U_{N}$, applying H\"older inequality on the last term of the latter sum and using Poincaré inequality yield then the result.
\end{proof}

\section{List of constants}\label{Appendix:constants}
We provide here a list of the constants involved in the main results together with their expressions. Recall the notations from Section \ref{Section:notation} and Appendix \ref{Appendix:weingarten} for notations involving moments of spectral distributions.

\subsection{Deconvolution procedure in the multiplicative case :}
\begin{flushleft}
$\bullet g(\xi)=\xi+\frac{1}{k(\xi)}\left(1+\left(\frac{1}{k(\xi)}+\frac{\vert\sigma_M^2-\sigma_1^2\vert}{k(\xi)\tilde{\sigma_1}}+\frac{\tilde{\sigma}_M^2}{\tilde{\sigma}_1^2\xi}\right)\left(\frac{\sigma_1^2}{\tilde{\sigma}_1}+\frac{1}{k(\xi)}\right)\right),$
\end{flushleft}
\begin{flushleft}
$\bullet t(\xi)=\left(\frac{\sigma_1^2}{k(\xi)\tilde{\sigma}_1}+\frac{\tilde{\sigma}_1^2+\sigma_1^4/2}{(k(\xi)\tilde{\sigma}_1)^2}\right)\left(2+\frac{\sigma_M^2}{\xi\tilde{\sigma}_1}+\frac{\tilde{\sigma}_M^2+\sigma_M^4/2}{\xi^2\tilde{\sigma}_1^2}\right),$
\end{flushleft}
\begin{flushleft}
$\bullet\theta(u)=6\left(1+\frac{\sigma_1^2}{k(u)\tilde{\sigma}_1}\right)\cdot\left(1+\frac{\sigma_M^2}{u\tilde{\sigma}_1}+\frac{4\tilde{\sigma}_M^2}{u^2\tilde{\sigma}_1^2}\right),$
\end{flushleft}
\begin{flalign*}
\bullet L(u)= &32\left(\frac{\sigma_1^2}{(u^2-4)\tilde{\sigma}_1^2}+\frac{2(\mu_1(3)-2\mu_1(2)+1)}{(u^2-4)^{3/2}\tilde{\sigma}_1^3}\right)\cdot \left(1+\frac{\sigma_M^2}{u\tilde{\sigma}_1}+\frac{4\tilde{\sigma}_M^2+\sigma_M^4}{u^2\tilde{\sigma}_1^2}\right)^2&\\
&\hspace{5cm}+\frac{8\sigma_1^2}{(u^2-4)\tilde{\sigma_1}^2}\cdot\left(1 +8\frac{m_4-2m_3m_2+m_2^2}{u^3\tilde{\sigma}_1^3}\right),
\end{flalign*}
\begin{flushleft}
$\bullet R(k)=\frac{(1-t(k))\min\left(\frac{1-t(k)}{2L(k)},\frac{k\tilde{\sigma_1}}{4\theta(k)}\right)}{2\left(1+\frac{2\sigma_1^2}{\sqrt{k^2-4}\tilde{\sigma}_1}+\frac{\mu_1(3)-2\mu_1(2)+1}{(k^2-4)\tilde{\sigma}_1^2}\right)}.$
\end{flushleft}

\subsection{Concentration inequality in the additive case :}

\begin{flalign*}
\bullet&C_{thres,A}(\eta)=&\\
&\frac{12\sigma_B^2\sigma_A}{\eta^3}\left(1+\frac{\sigma_A^2+\sigma_B^2}{\eta^2}\right)\Bigg(\sqrt{2\left(1+\frac{\sigma_A^2+\sigma_B^2\theta_B}{\eta^2}\right)\cdot\left(1+\sqrt{\theta_A\theta_B}+\frac{2\sqrt{m_{A^2\ast B^2}(1^2,1^2)\theta_A}}{\sigma_B^2\eta^2}\right)}\\
&\hspace{3cm}+\sqrt{3\frac{\sqrt{\theta_B\theta_A}\sigma_A^2}{\eta^2}\left(1+\frac{m_{A^2\ast B^2}(1^2,1^2)^{1/2}a_4^{1/2}+b_6^{2/3}a_6^{1/3}}{\sigma_A^2\sigma_B^2\eta^2}\right)}+2\frac{\theta_B^{1/4}\sigma_B^3\theta_A^{1/4}}{\eta^3}\Bigg),
\end{flalign*}
\begin{flalign*}
\bullet&C_{thres,B}(\eta)=&\\
&\frac{12\sigma_A^2\sigma_B}{\eta^3}\left(1+\frac{\sigma_B^2+\sigma_A^2}{\eta^2}\right)\Bigg(\sqrt{2\left(1+\frac{\sigma_B^2+\sigma_A^2\theta_A}{\eta^2}\right)\cdot\left(1+\sqrt{\theta_A\theta_B}+\frac{2\sqrt{m_{A^2\ast B^2}(1^2,1^2)\theta_B}}{\sigma_A^2\eta^2}\right)}\\
&\hspace{3cm}+\sqrt{3\frac{\sqrt{\theta_B\theta_A}\sigma_B^2}{\eta^2}\left(1+\frac{m_{A^2\ast B^2}(1^2,1^2)^{1/2}b_4^{1/2}+a_6^{2/3}b_6^{1/3}}{\sigma_A^2\sigma_B^2\eta^2}\right)}+2\frac{\theta_A^{1/4}\sigma_A^3\theta_B^{1/4}}{\eta^3}\Bigg),
\end{flalign*}
\begin{flalign*}
\bullet& C_{bound,A}(\kappa)=&\\
&\frac{12\sqrt{6}\sigma_B^2\sigma_A}{\kappa^3\sigma_1^3}\left(1+\frac{\sigma_A^2+\sigma_B^2}{\kappa^2\sigma_1^2}\right)\sqrt{1+\frac{\sigma_A^2+\theta_B\sigma_B^2}{\kappa^2\sigma_1^2}}\sqrt{1+\frac{m_{A^2\ast B^2}(1^2,1^2)^{1/2}a_4^{1/2}+b_6^{2/3}a_6^{1/3}}{a_2b_2\kappa^2\sigma_1^2}},
\end{flalign*}
\begin{flalign*}
\bullet& C_{bound,B}(\kappa)=&\\
&\frac{12\sqrt{6}\sigma_A^2\sigma_B}{\kappa^3\sigma_1^3}\left(1+\frac{\sigma_B^2+\sigma_A^2}{\kappa^2\sigma_1^2}\right)\sqrt{1+\frac{\sigma_B^2+\theta_A\sigma_A^2}{\kappa^2\sigma_1^2}}\sqrt{1+\frac{m_{B^2\ast A^2}(1^2,1^2)^{1/2}b_4^{1/2}+a_6^{2/3}b_6^{1/3}}{a_2b_2\kappa^2\sigma_1^2}},
\end{flalign*}
\begin{flalign*}
&\bullet C_{1}(\kappa)=&\\
&\left(1+\frac{2}{\kappa^2}\right)C_{bound,B}(3\kappa/4)+ \left(1+\frac{C_{bound,B}(3\kappa/4)\left(1+\frac{16(a_2+b_2)}{9\kappa^2\sigma_1^2}\right)}{N^2}\right)\cdot\frac{1+2\sigma_B^2/(\kappa\sigma_1)^2)}{1-4/\kappa^2}\nonumber\\
&\hspace{6cm}\cdot\left(1+\frac{4}{\kappa^2}\right)\cdot\left(\frac{4}{3}+\frac{16\sigma_B^2}{9(\kappa\sigma_1)^2}\right)C_{bound,A}(3\kappa/4)\left(1+\frac{\sigma_B}{\kappa\sigma_1}\right),
\end{flalign*}
\begin{flalign*}
\bullet C_{2}(\kappa)=&&\\
& \left(1+\frac{C_{bound,B}(3\kappa/4)\left(1+\frac{16(a_2+b_2)}{9\kappa^2\sigma_1^2}\right)}{N^2}\right)\cdot\frac{1+2\sigma_B^2/(\kappa\sigma_1)^2}{1-4/\kappa^2}\cdot\left(1+\frac{4}{\kappa^2}\right)\cdot\left(1+\frac{\sigma_B}{\kappa\sigma_1}\right),&
\end{flalign*}
\begin{flalign*}
\bullet C_{3}& (\kappa)= 1+\frac{8}{3\kappa^2}&\\
+&\left(1+\frac{C_{bound,B}(\xi\sigma_1)\left(1+\frac{16(a_2+b_2)}{9\kappa^2\sigma_1^2}\right)}{N^2}\right)\cdot\frac{1+2\sigma_B^2/(\kappa\sigma_1)^2}{1-/\kappa^2}\cdot\frac{4}{\kappa^2}\cdot \left(1+\frac{\sigma_B}{\kappa\sigma_1}\right)\cdot \left(1+\frac{16\sigma_H^2}{9\kappa^2\sigma_1^2}\right),
\end{flalign*}
\begin{flalign*}
\bullet&MSE:=\mathbb{E}\left(\Vert\widehat{\mathcal{C}_{B}}-\mathcal{C}_{B}\Vert_{L^2}^2\right)&\\
\leq& \frac{1}{2\sqrt{2}\pi\sigma_1 N^2}\left(\frac{C_2(2\sqrt{2})C_A\left(1+\frac{(1+c/N)\sqrt{\mu_1(2)}}{\sqrt{2}\sigma_1}\right)}{\sqrt{2}\sigma_1}+\frac{4C_{3}(2\sqrt{2})}{3\sigma_1}\sqrt{\sigma_A^2+2\frac{\sigma_A^2\sigma_B^2+a_4}{3^2\sigma_1^2}}+\frac{C_{1}(2\sqrt{2})}{N}\right)^2.&
\end{flalign*}

\subsection{Concentration inequality in the multiplicative case :}
Recall that $k_3(X)=x_3-3x_2^2+2x_1^3$ for $X\in\mathcal{H}_{N}(\mathbb{C})$.
\begin{flalign*}
\bullet C_{thres,A}(\eta)=48b_2a_\infty^3&\left(1+\frac{m_{A\ast B}^N(1^{3},21^{2})}{\eta^2\sigma_B^2}\right)\cdot\left(1+\frac{m_2}{\eta}+\frac{\tilde{\sigma}_M^2}{\eta^2}\right)&\\
&\hspace{4cm}\cdot\left(1+\frac{k_3(B)+\sigma_B^2(a_2-\sigma_B^2))+\frac{(10+4b_2+5b_3)a_2}{N}}{(1-N^{-2})^2(1-4N^{-2})a_{\infty}\eta}\right),
\end{flalign*}
\begin{flalign*}
&\bullet C_{thres,B}(\eta)= 24a_{\infty}b_2\sqrt{1+\frac{m_{A\ast B}^N(1^3,21^2))}{b_2\eta^2}}\cdot \left(1+\frac{a_2}{\eta}+\frac{\tilde{\sigma}_A^2+a_2\sigma_B^2}{(1-N^{-2})\eta^2}\right)
\Bigg(\sqrt{1+\frac{m_{A\ast B}^N(1^3,21^2)}{b_2\eta^2}}&\\&\hspace{4cm}+\frac{a_\infty^{3/2}\sqrt{b_4}}{\sqrt{2b_2}\eta}+(1+2\sqrt{b_2}a_{\infty}^{3/2}/\eta)\frac{k_3(A)+\sigma_A^2(b_2-\sigma_A^2)+\frac{b_2(10+4a_2+5a_3)}{N}}{(1-N^{-2})^2(1-4N^{-2})^2\eta\sqrt{b_2}}\Bigg),
\end{flalign*}
\begin{flushleft}
$\bullet C_{bound,A}(\xi)=24\frac{a_{\infty}^3b_2}{\xi^3\tilde{\sigma}_1^3}\left(1+\frac{m_{A\ast B}^N(1^{3},21^{2})}{\xi^2\tilde{\sigma}_1^2b_2}\right)\cdot\left(1+\frac{a_2}{\xi\tilde{\sigma}_1}+\frac{a_2\sigma_B^2+\tilde{\sigma}_1^2}{(1-N^{-2})\xi^2\tilde{\sigma}_1^2}\right),$
\end{flushleft}
\begin{flalign*}
\bullet&C_{bound,B}(\xi)=\frac{4\sqrt{2}a_{\infty}b_2}{\xi^2\tilde{\sigma}_1^2}\left(1+\frac{1}{\xi\tilde{\sigma}_1}+\frac{\sigma_A^2+\sigma_B^2}{(1-N^{-2})\xi^2\tilde{\sigma}_1^2}\right)\cdot\sqrt{1+\frac{m_{A\ast B}^N(1^3,21^2)}{b_2\eta^2}}&\\
&\hspace{1cm}\cdot\Bigg(\frac{a_\infty^{3/2}}{\xi\tilde{\sigma}_1}\left(\sqrt{\frac{b_4}{b_2}}+\sqrt{\frac{9b_6}{4b_2\xi^2\tilde{\sigma}_1^2}}\right)+\sqrt{2}\sqrt{1+\frac{m_{A\ast B}^N(1^3,21^2)}{b_2\xi^2\tilde{\sigma}_1^2}}+\frac{3}{\sqrt{2b_2}\xi\tilde{\sigma}_1}\sqrt{b_4+\frac{m_{A\ast B}^N(1^3,2^3)}{\xi^2\tilde{\sigma}_1^2}}\Bigg),
\end{flalign*}
\begin{flalign*}
&\bullet C_{1}(\kappa)=\left(1+\frac{b_2}{\kappa\tilde{\sigma}_1}\right)\cdot\Bigg[\left(1+\frac{3\mu_1(2)}{2\xi\tilde{\sigma_1}}+\frac{9}{4\xi^2}\right)\cdot\frac{1+\frac{3\tilde{\sigma}_2^2}{2\eta\xi\tilde{\sigma_1}}}{1-\frac{3}{2\xi k(\xi)}}\cdot C_{bound,A}(\xi)&\\
+&\left(1+\frac{\sigma_1^2}{k(\xi)\tilde{\sigma}_1}\right)\cdot\left(1+\frac{\sigma_M}{\xi\tilde{\sigma}_1}\right)\cdot\left(1+\frac{a_2}{\xi\tilde{\sigma}_1}+\frac{a_2\sigma_B^2+\tilde{\sigma}_A^2}{(1-N^{-2})\xi^2\tilde{\sigma}_1^2}\right)\cdot\left(1+\frac{3b_2}{2\xi\tilde{\sigma}_2}+\frac{9\tilde{\sigma}_B^2}{4\xi^2\tilde{\sigma}_1^2}\right)\cdot C_{bound,B}(\xi)\Bigg],&
\end{flalign*}
\begin{flushleft}
$\bullet C_{2}(\kappa)=\left(1+\frac{b_2}{\kappa\tilde{\sigma}_1}\right)\cdot\left(1+\frac{3\mu_1(2)}{2\xi\tilde{\sigma_1}}+\frac{9}{4\xi^2}\right)\cdot\frac{1+\frac{3\tilde{\sigma}_2^2}{2\eta\xi\tilde{\sigma_1}}}{1-\frac{3}{2\xi k(\xi)}},$
\end{flushleft}
\begin{flushleft}
$\bullet C_{3}(\kappa)=1+\frac{3}{2\xi\tilde{\sigma}_1}\cdot\left(1+\frac{b_2}{\kappa\tilde{\sigma}_1}\right)\cdot\left(\sigma_1^2+\frac{\sigma_1^2}{k(\xi)\tilde{\sigma}_1}+\frac{\tilde{\sigma}_1^2}{k(\xi)\tilde{\sigma}_1}\right)\cdot\frac{1+\frac{3\tilde{\sigma}_2^2}{2\eta\xi\tilde{\sigma_1}}}{1-\frac{3}{2\xi k(\xi)}},$
\end{flushleft}
\begin{flushleft}
$\bullet C_{4}(\kappa)=\frac{2^4\max(C_{thres,A}(\xi\tilde{\sigma}_1),C_{thres,B}(\xi\tilde{\sigma}_1))^3\left(1+\frac{1}{\pi^2k\circ g^{-1}(\kappa)}\right)^3}{\sqrt{3}\pi^2(\xi\tilde{\sigma}_1)^9},$
\end{flushleft}
\begin{flalign*}
\bullet MSE:=&\mathbb{E}(\Vert\widehat{\mathcal{C}}_{B}[\eta]-\mathcal{C}_{B}[\eta]\Vert_{L^2}^2)&\\
\leq & \frac{1}{\kappa\pi\tilde{\sigma}_1N^2}\left(\frac{3C_2(\kappa)C_A\left(1+\frac{3(1+c/N)\sqrt{\mu_1(2)}}{2g^{-1}(\kappa)\tilde{\sigma_1}}\right)}{2g^{-1}(\kappa)\tilde{\sigma_1} }+\frac{C_3(\kappa)\sqrt{\Delta(\kappa)}}{g^{-1}(\kappa)\tilde{\sigma}_1}+\frac{C_1(\kappa)}{N}\right)^2+\frac{C_4(\kappa)}{N^6}.
\end{flalign*}
\bibliographystyle{alpha}
\bibliography{biblio}

\newcommand{\etalchar}[1]{$^{#1}$}
\begin{thebibliography}{MNN{\etalchar{+}}20}

\bibitem[AGZ10]{AnGuZe}
Greg~W. Anderson, Alice Guionnet, and Ofer Zeitouni.
\newblock {\em An introduction to random matrices}, volume 118 of {\em
  Cambridge Studies in Advanced Mathematics}.
\newblock Cambridge University Press, Cambridge, 2010.

\bibitem[ATV17]{ATV}
Octavio Arizmendi, Pierre Tarrago, and Carlos Vargas.
\newblock Subordination methods for free deconvolution.
\newblock 2017.

\bibitem[BABP16]{AlBoBu}
Jo\"{e}l Bun, Romain Allez, Jean-Philippe Bouchaud, and Marc Potters.
\newblock Rotational invariant estimator for general noisy matrices.
\newblock {\em IEEE Trans. Inform. Theory}, 62(12):7475--7490, 2016.

\bibitem[BB04]{Bebe2}
S.~T. Belinschi and H.~Bercovici.
\newblock Atoms and regularity for measures in a partially defined free
  convolution semigroup.
\newblock {\em Math. Z.}, 248(4):665--674, 2004.

\bibitem[BB07]{Bebe}
S.~T. Belinschi and H.~Bercovici.
\newblock A new approach to subordination results in free probability.
\newblock {\em J. Anal. Math.}, 101:357--365, 2007.

\bibitem[BBP17]{BoPo}
Jo\"{e}l Bun, Jean-Philippe Bouchaud, and Marc Potters.
\newblock Cleaning large correlation matrices: tools from random matrix theory.
\newblock {\em Phys. Rep.}, 666:1--109, 2017.

\bibitem[BE85]{BaEm}
D.~Bakry and Michel \'{E}mery.
\newblock Diffusions hypercontractives.
\newblock In {\em S\'{e}minaire de probabilit\'{e}s, {XIX}, 1983/84}, volume
  1123 of {\em Lecture Notes in Math.}, pages 177--206. Springer, Berlin, 1985.

\bibitem[Bel05]{Bel}
Serban~Teodor Belinschi.
\newblock {\em Complex analysis methods in noncommutative probability}.
\newblock ProQuest LLC, Ann Arbor, MI, 2005.
\newblock Thesis (Ph.D.)--Indiana University.

\bibitem[Ben17]{Ben}
Tamir Bendory.
\newblock Robust recovery of positive stream of pulses.
\newblock {\em IEEE Trans. Signal Process.}, 65(8):2114--2122, 2017.

\bibitem[BES17]{BES}
Zhigang Bao, L\'{a}szl\'{o} Erd\H{o}s, and Kevin Schnelli.
\newblock Local law of addition of random matrices on optimal scale.
\newblock {\em Comm. Math. Phys.}, 349(3):947--990, 2017.

\bibitem[BGEM19]{AlBeEn}
Florent Benaych-Georges, Nathana\"{e}l Enriquez, and Alk\'{e}os Micha\"{\i}l.
\newblock Empirical spectral distribution of a matrix under perturbation.
\newblock {\em J. Theoret. Probab.}, 32(3):1220--1251, 2019.

\bibitem[BGH20]{BeGuHu}
S.~Belinschi, A.~Guionnet, and J.~Huang.
\newblock Large deviation principles via spherical integrals.
\newblock {\em arXiv preprint arXiv:2004.07117}, 2020.

\bibitem[Bia98]{Bia}
Philippe Biane.
\newblock Processes with free increments.
\newblock {\em Math. Z.}, 227(1):143--174, 1998.

\bibitem[BMS17]{BeMaSp}
Serban~T. Belinschi, Tobias Mai, and Roland Speicher.
\newblock Analytic subordination theory of operator-valued free additive
  convolution and the solution of a general random matrix problem.
\newblock {\em J. Reine Angew. Math.}, 732:21--53, 2017.

\bibitem[BV04]{BoVa}
Stephen Boyd and Lieven Vandenberghe.
\newblock {\em Convex optimization}.
\newblock Cambridge University Press, Cambridge, 2004.

\bibitem[Col03]{Col}
Beno\^{\i}t Collins.
\newblock Moments and cumulants of polynomial random variables on unitary
  groups, the {I}tzykson-{Z}uber integral, and free probability.
\newblock {\em Int. Math. Res. Not.}, (17):953--982, 2003.

\bibitem[CS06]{ColSni}
Beno\^{\i}t Collins and Piotr \'{S}niady.
\newblock Integration with respect to the {H}aar measure on unitary, orthogonal
  and symplectic group.
\newblock {\em Comm. Math. Phys.}, 264(3):773--795, 2006.

\bibitem[DDP17]{DeDuPe}
Quentin Denoyelle, Vincent Duval, and Gabriel Peyr\'{e}.
\newblock Support recovery for sparse super-resolution of positive measures.
\newblock {\em J. Fourier Anal. Appl.}, 23(5):1153--1194, 2017.

\bibitem[DP17]{DuPe}
Vincent Duval and Gabriel Peyr\'{e}.
\newblock Sparse regularization on thin grids {I}: the {L}asso.
\newblock {\em Inverse Problems}, 33(5):055008, 29, 2017.

\bibitem[EKN20]{ErKrNe}
L\'{a}szl\'{o} Erd\H{o}s, Torben Kr\"{u}ger, and Yuriy Nemish.
\newblock Local laws for polynomials of {W}igner matrices.
\newblock {\em J. Funct. Anal.}, 278(12):108507, 59, 2020.

\bibitem[Fan91]{Fan1}
Jianqing Fan.
\newblock On the optimal rates of convergence for nonparametric deconvolution
  problems.
\newblock {\em Ann. Statist.}, 19(3):1257--1272, 1991.

\bibitem[Fan92]{Fan2}
Jianqing Fan.
\newblock Deconvolution with supersmooth distributions.
\newblock {\em Canad. J. Statist.}, 20(2):155--169, 1992.

\bibitem[FKN19]{FuKoNe}
M.~Fukuda, R.~Koenig, and I.~Nechita.
\newblock Rtni—a symbolic integrator for haar-random tensor networks.
\newblock {\em Journal of Physics A: Mathematical and Theoretical}, 52(42),
  2019.

\bibitem[GZ00]{GuiZei}
A.~Guionnet and O.~Zeitouni.
\newblock Concentration of the spectral measure for large matrices.
\newblock {\em Electron. Comm. Probab.}, 5:119--136, 2000.

\bibitem[Hum72]{Hum}
James~E. Humphreys.
\newblock {\em Introduction to {L}ie algebras and representation theory}.
\newblock Springer-Verlag, New York-Berlin, 1972.
\newblock Graduate Texts in Mathematics, Vol. 9.

\bibitem[Huy10]{Hu}
Daan Huybrechs.
\newblock On the {F}ourier extension of nonperiodic functions.
\newblock {\em SIAM J. Numer. Anal.}, 47(6):4326--4355, 2010.

\bibitem[JT]{JoTa}
Emilien Joly and Pierre Tarrago.
\newblock Concentration bounds for the spectral deconvolution.
\newblock {\em Work in progess}.

\bibitem[Kar12]{Kar1}
Vladislav Kargin.
\newblock A concentration inequality and a local law for the sum of two random
  matrices.
\newblock {\em Probab. Theory Related Fields}, 154(3-4):677--702, 2012.

\bibitem[Kar15]{Kar2}
V.~Kargin.
\newblock Subordination for the sum of two random matrices.
\newblock {\em Ann. Probab.}, 43(4):2119--2150, 2015.

\bibitem[Lac06]{La}
Claire Lacour.
\newblock Rates of convergence for nonparametric deconvolution.
\newblock {\em C. R. Math. Acad. Sci. Paris}, 342(11):877--882, 2006.

\bibitem[LP11]{LePe}
Olivier Ledoit and Sandrine P\'{e}ch\'{e}.
\newblock Eigenvectors of some large sample covariance matrix ensembles.
\newblock {\em Probab. Theory Related Fields}, 151(1-2):233--264, 2011.

\bibitem[LW04]{LeWo1}
Olivier Ledoit and Michael Wolf.
\newblock A well-conditioned estimator for large-dimensional covariance
  matrices.
\newblock {\em J. Multivariate Anal.}, 88(2):365--411, 2004.

\bibitem[LW15]{LeWo2}
Olivier Ledoit and Michael Wolf.
\newblock Spectrum estimation: a unified framework for covariance matrix
  estimation and {PCA} in large dimensions.
\newblock {\em J. Multivariate Anal.}, 139:360--384, 2015.

\bibitem[MM13]{MeMe}
Elizabeth~S. Meckes and Mark~W. Meckes.
\newblock Concentration and convergence rates for spectral measures of random
  matrices.
\newblock {\em Probab. Theory Related Fields}, 156(1-2):145--164, 2013.

\bibitem[MNN{\etalchar{+}}20]{Mal}
Mylene Maïda, Tien~Dat Nguyen, Thanh Mai~Pham Ngoc, Vinvent Rivoirard, and
  Viet~Chi Tran.
\newblock Statistical deconvolution of the free fokker-planck equation at fixed
  time.
\newblock {\em arXiv preprint arXiv:2006.11899}, 2020.

\bibitem[MS17]{MiSp}
James~A. Mingo and Roland Speicher.
\newblock {\em Free probability and random matrices}, volume~35 of {\em Fields
  Institute Monographs}.
\newblock Springer, New York; Fields Institute for Research in Mathematical
  Sciences, Toronto, ON, 2017.

\bibitem[Neu88]{Neu}
A.~Neubauer.
\newblock Tikhonov-regularization of ill-posed linear operator equations on
  closed convex sets.
\newblock {\em J. Approx. Theory}, 53(3):304--320, 1988.

\bibitem[NS06]{NiSp}
Alexandru Nica and Roland Speicher.
\newblock {\em Lectures on the combinatorics of free probability}, volume 335
  of {\em London Mathematical Society Lecture Note Series}.
\newblock Cambridge University Press, Cambridge, 2006.

\bibitem[PV00]{PaVa}
L.~Pastur and V.~Vasilchuk.
\newblock On the law of addition of random matrices.
\newblock {\em Comm. Math. Phys.}, 214(2):249--286, 2000.

\bibitem[Ser10]{Ser}
Denis Serre.
\newblock {\em Matrices}, volume 216 of {\em Graduate Texts in Mathematics}.
\newblock Springer, New York, second edition, 2010.
\newblock Theory and applications.

\bibitem[Spe93]{Spe2}
Roland Speicher.
\newblock Free convolution and the random sum of matrices.
\newblock {\em Publ. Res. Inst. Math. Sci.}, 29(5):731--744, 1993.

\bibitem[Spe94]{Spe}
Roland Speicher.
\newblock Multiplicative functions on the lattice of noncrossing partitions and
  free convolution.
\newblock {\em Math. Ann.}, 298(4):611--628, 1994.

\bibitem[Vas01]{Vas}
Vladimir Vasilchuk.
\newblock On the law of multiplication of random matrices.
\newblock {\em Math. Phys. Anal. Geom.}, 4(1):1--36, 2001.

\bibitem[Voi91]{Voi1}
Dan Voiculescu.
\newblock Limit laws for random matrices and free products.
\newblock {\em Invent. Math.}, 104(1):201--220, 1991.

\bibitem[Voi00]{Voi2}
Dan Voiculescu.
\newblock The coalgebra of the free difference quotient and free probability.
\newblock {\em Internat. Math. Res. Notices}, (2):79--106, 2000.

\end{thebibliography}

\end{document}